\documentclass[11pt,
final,
english
]{article}
   
\usepackage[T1]{fontenc}       
\usepackage{lmodern}		
\usepackage{xcolor}
\usepackage{calc}
\usepackage{exscale}           
\usepackage{array}             
\usepackage{xspace}            
\usepackage{faktor}			   
\usepackage{xparse} 
\usepackage{appendix}
\usepackage{geometry}
\usepackage[expansion=false]{microtype}        
\usepackage{fancyhdr}          
\usepackage{xkeyval}           
\usepackage{enumitem}
\usepackage{tikz}              
\usepackage[english,strings]{babel}     
\usepackage[backend = biber,
backref = true,
style = alphabetic,
maxcitenames = 99,
maxbibnames = 99,
mincitenames = 5]{biblatex}
\usepackage[
final,                    
plainpages=false,              
pdfpagelabels=true,            
pdfencoding=auto,              
unicode=true,                  
hypertexnames=true,            
naturalnames=true,             
colorlinks,
linkcolor={blue!70!black},
citecolor={blue!50!black},
urlcolor={blue!80!black}
]{hyperref}                  
\usepackage{nchairx} 
\usepackage{csquotes}
\usepackage{xurl}
\hypersetup{breaklinks=true}

\usetikzlibrary[patterns,patterns.meta,cd]

\DeclareNameAlias{default}{family-given}
\makeatletter
\newcommand{\Naturals}{\mathbb{N}}
\newcommand{\Integers}{\mathbb{Z}}
\newcommand{\Reals}{\mathbb{R}}
\newcommand{\ComplexNum}{\mathbb{C}}
\newcommand{\@gerstennstar}[1]{\ch@irxbbracket*{#1}}
\newcommand{\@gerstennostar}[2][]{\ch@irxbbracket[#1]{#2}}
\newcommand{\GerstBracket}{\@ifstar\@gerstenstar\@gerstennostar}
\renewcommand{\@schoutenstar}[1]{\ch@irxbbracket*{#1}}
\renewcommand{\@schoutennostar}[2][]{\ch@irxbbracket[#1]{#2}}
\renewcommand{\@courantstar}[1]{\ch@irxbbracket*{#1}}
\renewcommand{\@courantnostar}[2][]{\ch@irxbbracket[#1]{#2}}
\newcommand{\Total}{{\scriptscriptstyle\script{T}}}
\newcommand{\Wobs}{{\scriptscriptstyle\script{N}}}
\newcommand{\Null}{{\scriptscriptstyle\script{0}}}
\newcommand{\Van}{{\scriptscriptstyle\script{I}}}
\newcommand{\TotalNotWobs}{{\scriptscriptstyle\faktor{\Total}{\Wobs}}}

\newcommand{\WobsNotNull}{{\scriptscriptstyle\faktor{\Wobs}{\Null}}}
\newcommand{\NullNotVan}{{\scriptscriptstyle\faktor{\Null}{\Van}}}
\newcommand{\TOTAL}{\script{T}}
\newcommand{\WOBS}{\script{N}}
\newcommand{\NULL}{\script{0}}
\newcommand{\str}{{\scriptstyle\mathrm{str}}}
\newcommand{\inj}{{\scriptstyle\mathrm{emb}}}
\newcommand{\LieRineAlg}{\categoryname{LieRinehartAlg}}

\newcommand{\LieAlgd}{\categoryname{LieAlgd}}
\newcommand{\LieBiAlgd}{\categoryname{LieBiAlgd}}
\newcommand{\ConMap}{{\Con\Map}}
\newcommand{\ConCinfty}{\Con\Cinfty}
\newcommand{\ConSecinfty}{\Con\Secinfty}
\newcommand{\ConHom}{{\Con\!\Hom}}
\newcommand{\ConEnd}{\Con\!\End}
\newcommand{\strtensor}[1][{}]{\mathbin{\boxtimes_{\scriptscriptstyle{#1}}}}
\newcommand{\regimage}{\operator{regim}}
\newcommand{\ConDer}{{\Con\!\Der}}
\newcommand{\ConVecFields}{\Con\VecFields}
\newcommand{\ConForms}{\Con\Forms}
\newcommand{\Con}{\categoryname{C}}
\newcommand{\injCon}{\Con^\inj}
\newcommand{\strCon}{\Con_\str}
\newcommand{\injstrCon}{\Con^\inj_\str}

\newcommand{\injConIndSet}{{\Con^\inj_{\scriptstyle\mathrm{ind}}\Sets}}

\newcommand{\injConAlg}{{\injCon\Algebras}}
\newcommand{\injstrConAlg}{{\injstrCon\Algebras}}

\newcommand{\injConMod}{\injCon\Modules}
\newcommand{\injstrConMod}{{\injstrCon\Modules}}

\newcommand{\Proj}{\categoryname{Proj}}
\newcommand{\strConProj}{\strCon\Proj}
\newcommand{\injConLieAlg}{\injCon\LieAlgs}

\newcommand{\injConLieRineAlg}{\injCon\LieRineAlg}
\newcommand{\injstrConLieRineAlg}{\injstrCon\LieRineAlg}

\newcommand{\injConLieAlgd}{\Con\LieAlgd}

\newcommand{\injConLieBiAlgd}{\Con\LieBiAlgd}

\newcommand{\ConMfld}{{\Con\Manifolds}}
\newcommand{\ConVect}{{\Con\Vect}}
\newcommand{\Bott}{{\scriptstyle\mathrm{Bott}}}
\newcommand{\Ann}{\operatorname{Ann}}
\newcommand{\vanishing}{\spacename{I}}
\newcommand{\Pnormalizer}{\mathcal{B}}

\newcommand{\VecFields}{\mathfrak{X}}
\newcommand{\genT}{\mathbb{T}}
\makeatother

\title{Constraint Vector Bundles and Reduction of Lie (Bi-)Algebroids}
\date{}
\author{\textbf{Marvin Dippell}\thanks{\texttt{mdippell@unisa.it}}\\[0.2cm]
	\begin{minipage}{8cm}
		\centering\small
		Dipartimento di Matematica\\
		Università degli Studi di Salerno\\
		via Giovanni Paolo II, 132\\
		84084 Fisciano (SA)\\
		Italy
	\end{minipage}
	\vspace*{0.5cm}\\
	\textbf{David Kern}\thanks{\texttt{david.kern@mathematik.uni-goettingen.de}}\\[0.2cm]
	\begin{minipage}{8cm}
		\centering\small
		Mathematisches Institut\\
		Georg-August-Universit\"{a}t G\"{o}ttingen\\
		Bunsenstra{\ss}e 3 - 5\\
		37073 G\"{o}ttingen\\
		Germany
	\end{minipage}
}

\begin{document}
\selectlanguage{english}

\maketitle	

\begin{abstract}
We present a framework for the reduction of various geometric structures extending  
the classical coisotropic Poisson reduction.
For this we introduce constraint manifolds and constraint vector bundles.
A constraint Serre-Swan theorem is proven, identifying constraint vector bundles with certain
finitely generated projective modules,
and a Cartan calculus for constraint differentiable forms and multivector fields
is introduced.
All of these constructions will be shown to be compatible with reduction. 
Finally, we apply this to obtain a reduction procedure for Lie (bi-)algebroids 
and Dirac manifolds. 
\end{abstract}

\tableofcontents

\section{Introduction}
\label{sec:Introduction}

The reduction of Poisson manifolds by coisotropic submanifolds is a widely used tool
in Poisson geometry, generalizing the classical Marsden-Weinstein reduction in
symplectic geometry \cite{marsden.weinstein:1974a,meyer:1973a} which itself can 
be seen as the mathematical incarnation of symmetry reduction in classical 
mechanics. 
Given a coisotropic submanifold
$\iota \colon C \hookrightarrow M$
of a Poisson manifold $(M,\pi)$
there is an associated distribution
$D \subseteq TC$
which we call the characteristic distribution of $C$.
If $C$ is closed and the leaf space of $D$ carries a canonical smooth structure
we denote it by
\begin{equation}
	M_\red \coloneqq C / D,
\end{equation}
and call it the reduced manifold of $M$ by $C$.
It can be shown that there exists a canonical Poisson structure $\pi_\red$
on $M_\red$, which can algebraically be constructed as follows:
The ideal $\vanishing_C$ of functions vanishing on $C$
is a Poisson subalgebra of the Poisson algebra $\Cinfty(M)$,
with Poisson bracket $\{\argument,\argument\}$ induced by $\pi$.
This allows us to consider its Poisson normalizer
\begin{equation}
	\Pnormalizer_C
	= \big\{ f \in \Cinfty(M) \mid \{f,\vanishing_C\} \subseteq \vanishing_C \big\},
\end{equation}
which is the biggest Poisson subalgebra of $\Cinfty(M)$ containing
$\vanishing_C$ as a Poisson ideal.
Since $\Pnormalizer_C$ consists of all functions which are leafwise constant on $C$
we obtain an isomorphism $\Cinfty(M_\red) \simeq \Pnormalizer_C / \vanishing_C$
which allows us to define $\pi_\red$ as the Poisson structure corresponding to the
quotient Poisson bracket on $\Pnormalizer_C / \vanishing_C$.
This procedure is presented in \cite{marsden.ratiu:1985a,ortega.ratiu:2004a}. 

There are various reasons to consider reductions of other geometric objects.
For a start, every Poisson manifold comes with related geometric structures, such as the
Lie algebroid structure on the cotangent bundle $T^*M$.
These structures are important in the study of Poisson manifolds and hence their behaviour 
under coisotropic reduction needs to be understood. 
Special cases of such a reduction in the context of Lie algebroids can be found 
in  \cite{carinena.nunesdacosta.santos:2005a,hawkins:2008a,jotzlean.ortiz:2014a}. 
But also for other generalized geometric objects a reduction procedure is desirable and during the last 
two decades reduction procedures for a vast number of objects has been studied, e.g. for Courant algebroids 
\cite{bursztyn.cattaneo.mehta.zambon:2023a}, Dirac structures 
\cite{marsden.yoshimura:2007a}, generalized complex structures  
\cite{bursztyn.cavalcanti.gualtieri:2005a,stienon.xu:2008a,vaisman:2007a,zambon:2008a}, 
contact structures \cite{willett:2002a,deleon.valcazar:2019a}, 
cosymplectic manifolds \cite{albert:1989a}, multisymplectic manifolds \cite{blacker:2021a,blacker.miti.ryvkin:2023a},
etc.
Recently, reduction has also been put into the context of graded geometry 
\cite{bursztyn.cattaneo.mehta.zambon:2023a}, see also 
\cite{cattaneo.zambon:2010a,cattaneo.zambon:2010b}.

While many of the aforementioned works focus on reduction with respect to Lie group actions
and momentum maps,
in this article we want to present a framework for reduction focusing on the more general coisotropic point of view.
For one thing, this will generalize many of the above constructions by allowing for quotients not necessarily given by
group actions, while on the other hand it will provide a common language to treat reduction in many geometric (and algebraic)
situations.
Thus it opens the door to study reduction of geometric structures not studied in the literature so far.

This framework is based on the notion of constraint algebra, which was first introduced in
\cite{dippell.esposito.waldmann:2019a}, under the name of coisotropic algebra, to investigate the behaviour of Morita equivalence under reduction and used in \cite{dippell.esposito.waldmann:2022a}
to study star products compatible with reduction.
The main strategy of our approach is to replace the classical notion of smooth manifolds
by so-called \emph{constraint manifolds}, which are essentially manifolds equipped with the additional structure needed to perform reduction.
More precisely, a constraint manifold $\mathcal{M}$ consists of a smooth manifold $M$ together with a closed embedded submanifold $C \subseteq M$ and a simple distribution $D$ on $C$.
Every such constraint manifold $\mathcal{M}$ can now easily be reduced to
$\mathcal{M}_\red = C/D$.
Starting with this basic notion we will introduce constraint analogues of many other classical geometric concepts, such as constraint vector bundles, constraint vector fields and constraint differential forms.
These will, by definition, be compatible with reduction in a specific sense.
These building blocks can then be used to define constraint versions of more sophisticated structures,
which we will demonstrate by introducing constraint Lie-Rinehart algebras, constraint Lie (bi-)algebroids and constraint Dirac manifolds.
In some sense these constraint notions can be understood as the collection of data needed to perform reduction of the corresponding classical structure.

Beside the examples presented here, our approach lends itself to be applied to
other situations, such as Courant algebroids and general Dirac structures, or
the study of constraint Lie groupoids as global counterparts of 
constraint Lie algebroids.
However, this is left for future work. 

The outline of this work is the following: 
In order to do constraint geometry effectively we start with introducing (strong) constraint algebras and their projective modules in \autoref{sec:AlgebraicPreliminaries} which later on will replace the classical algebras of functions and modules of sections.
We will also give definitions of (graded) constraint Lie algebras and constraint Gerstenhaber algebras, as they will be useful later on.
The main results of the paper are contained in \autoref{sec:ConGeometry}.
Here, after the introduction of constraint manifolds and constraint vector bundles,
we prove a constraint version of the classical Serre-Swan Theorem, see \autoref{thm:strConSerreSwan},
showing that taking sections gives an equivalence of categories between constraint vector bundles
and finitely generated projective strong constraint modules over the constraint algebra of functions
on a constraint manifold.
This result allows us then to introduce the full constraint Cartan calculus given by constraint vector fields and constraint differential forms together with insertion, Lie derivative and de Rham differential.
Finally, in \autoref{sec:Applications} we present first applications of the previously developed theory to the reduction of Lie-Rinehart algebras, Lie (bi-)algebroids and Dirac manifolds.
In the appendix \autoref{sec:morphismsOfLinearConnections} we collect some well-known facts about 
morphisms between vector bundles with linear connections.
Even though these are considered to be general knowledge, it is hard to 
find suitable sources for it in the literature. 

\paragraph{Remark on notation and convention}
We will heavily use the language of constraint algebraic (and geometric) objects.
Many of these, such as constraint algebras, have been introduced in
\cite{dippell.esposito.waldmann:2019a, dippell.esposito.waldmann:2022a, dippell.menke.waldmann:2022a},
albeit under a different name and with slight changes in the definitions.
For example a \emph{coisotropic triple of algebras}
as defined in \cite{dippell.esposito.waldmann:2019a}
is now called an embedded constraint algebra with the additional property of
$\algebra{A}_\Null \subseteq \algebra{A}_\Total$ being a left ideal.
While in \cite{dippell.esposito.waldmann:2022a}
the notion of coisotropic algebra corresponds to what we now call a constraint algebra.

A consistent naming scheme has been introduced in \cite{dippell:2023a}
and we will mostly adhere to this.
Nevertheless, though the distinction between constraint and embedded constraint objects
is an important feature of the general theory, in our applications we will only need to consider
embedded constraint objects.
This is why we will implicitly assume that all our constraint objects are embedded.
To keep a certain level of compatibility with \cite{dippell:2023a}
we will nevertheless refer to embedded constraint objects in definitions and names of categories.

\pagebreak

Let us also fix some conventions:
Even though most results still hold for (commutative) rings,
from now on $\field{k}$ will always denote a field.
Moreover, all algebras are supposed to be associative and unital $\field{k}$-algebras.
When it comes to geometry all manifolds are supposed to be real, smooth and connected.

\paragraph{Acknowledgement}
The first author wants to thank Chiara Esposito and Stefan Waldmann
for the many discussions and their constant support,
as well as Jonas Schnitzer and Andreas Kraft for all the valuable insights
into the geometry of reduction.

\section{Algebraic Preliminaries}
\label{sec:AlgebraicPreliminaries}

In this preliminary chapter we introduce the basic algebraic notions
needed in the main part of this work.
In particular we will need the concept of constraint algebras for
describing smooth functions on a constraint manifold, 
and constraint modules for sections of constraint vector bundles.
As in classical differential geometry, modules of sections of constraint
vector bundles will be characterized by being finitely generated projective.
Before we introduce this concept in \autoref{sec:ProjectiveModules}
we need to study constraint index sets in \autoref{sec:ConIndexSets}.
Finally, in \autoref{sec:ConLieAlgebras} we recall the
notion of (graded) constraint Lie algebras.
This will be needed in our discussion of constraint
Lie (bi)-algebroids in \autoref{sec:ConLieAlgebroids}.

The basic notions of constraint algebras and their modules have
already been introduced in \cite{dippell.esposito.waldmann:2019a}
under the name of coisotropic algebras and coisotropic modules.
Finitely generated projective modules over such coisotropic algebras
and their geometric counterparts have been studied in
\cite{dippell.menke.waldmann:2022a}.
It should be noted that the notion of projective constraint module
used here is similar to that used in \cite{dippell.menke.waldmann:2022a},
but differs in important aspects.

\subsection{Constraint Algebras and their Modules}
\label{sec:ConAlgebrasAndModules}

Recall again the situation of coisotropic reduction in Poisson geometry:
Given a coisotropic submanifold $C \subseteq (M,\pi)$
we can consider the vanishing ideal $\vanishing_C$
and its Poisson normalizer $\Pnormalizer_C \subseteq \Cinfty(M)$.
Note that in order to obtain the algebra of functions on the reduced manifold only the fact that
$\vanishing_C \subseteq \Pnormalizer_C$ is a two-sided ideal is needed.
The triple $(\Cinfty(M),\Pnormalizer_C,\vanishing_C)$
motivates now the following definition of a constraint algebra.

\begin{definition}[Constraint algebra]\
	\label{def:ConstraintAlgebras}
	\begin{definitionlist}
		\item An \emph{(embedded) constraint algebra}
		is a triple
		$\algebra{A} = (\algebra{A}_\Total, \algebra{A}_\Wobs, \algebra{A}_\Null)$
		consisting of a unital associative 
		$\field{k}$-algebra
		$\algebra{A}_\Total$ together with a 
		a unital subalgebra
		$\algebra{A}_\Wobs \subseteq \algebra{A}_\Total$
		and a two-sided ideal
		$\algebra{A}_\Null \subseteq \algebra{A}_\Wobs$.
		\item A \emph{morphism $\phi \colon \algebra{A} \to \algebra{B}$
			of constraint $\field{k}$-algebras}
		is a unital algebra homomorphisms
		$\phi \colon \algebra{A}_\Total \to \algebra{B}_\Total$
		such that
		$\phi(\algebra{A}_\Wobs) \subseteq \algebra{B}_\Wobs$
		and
		$\phi(\algebra{A}_\Null) \subseteq \algebra{B}_\Null$.
		\item A constraint algebra $\algebra{A} = (\algebra{A}_\Total, \algebra{A}_\Wobs, \algebra{A}_\Null)$
		is called \emph{strong}, if
		$\algebra{A}_\Null \subseteq \algebra{A}_\Total$
		is a two-sided ideal.
		\item The category of constraint $\field{k}$-algebras is denoted by
		$\injConAlg_\field{k}$.
		Moreover, $\injstrConAlg_\field{k}$ the full subcategory of strong constraint algebras.
	\end{definitionlist}
\end{definition}

\begin{remark}
	The above notion of an embedded constraint algebra can be relaxed to the notion of constraint algebra
	by allowing for a non-injective algebra homomorphism
	$\iota_\algebra{A} \colon \algebra{A}_\Wobs \to \algebra{A}_\Total$
	instead of the inclusion $\algebra{A}_\Wobs \subseteq \algebra{A}_\Total$.
	This more general concept becomes important when considering
	quotients of constraint algebras or constraint modules,
	and therefore is needed in order to study cohomologies in the constraint setting,
	see \cite{dippell.esposito.waldmann:2022a}.
	Nevertheless, in the setting of this paper we will mostly be concerned with embedded
	constraint algebraic objects and we will therefore 
	drop the word embedded most of the time.
	For an in-depth treatment of non-embedded constraint algebra see
	\cite{dippell:2023a}.
\end{remark}

We will mostly drop the underlying field $\field{k}$ in the notation.
Note that we can view $\field{k}$ itself as a constraint algebra by using
$\field{k} = (\field{k},\field{k},0)$.
Moreover, we often denote the restriction of a constraint morphism to the $\WOBS$- or $\NULL$-component
by $\phi_\Wobs$ and $\phi_\Null$, respectively.

By the definition of constraint algebras
we obtain a reduction functor
\begin{equation}
	\red \colon \injConAlg \to \Algebras
\end{equation}
given by $\algebra{A}_\red \coloneqq \algebra{A}_\Wobs / \algebra{A}_\Null$.
The restriction to the subcategory $\injstrConAlg$ will also be denoted by $\red$.

Let us define modules over constraint algebras.
In the following we will only consider right modules, the definition of left- and bi-modules should be clear, details can be found in \cite{dippell.esposito.waldmann:2022a,dippell:2023a}.

\begin{definition}[Modules over constraint algebras]
	\label{def:ConstraintModules}
	Let $\algebra{A} \in \injConAlg$
	be a constraint algebra.
	\begin{definitionlist}
		\item An \emph{(embedded) constraint $\algebra{A}$-module}
		is a triple
		$\module{E} = (\module{E}_\Total, \module{E}_\Wobs,
		\module{E}_\Null)$,
		consisting of a right $\algebra{A}_\Total$-module
		$\module{E}_\Total$
		and $\algebra{A}_\Wobs$-submodules
		$\module{E}_\Wobs \subseteq \module{E}_\Total$
		and 
		$\module{E}_\Null \subseteq \module{E}_\Wobs$.
		\item A \emph{morphism 
		$\Phi\colon \module{E} \to \module{F}$
		between constraint $\algebra{A}$-modules}
		is an $\algebra{A}_\Total$-module morphism
		$\Phi \colon \module{E}_\Total \to \module{F}_\Total$
		such that
		$\Phi(\module{E}_\Wobs) \subseteq \module{F}_\Wobs$
		and
		$\Phi(\module{E}_\Null) \subseteq \module{F}_\Null$.
		\item A constraint $\algebra{A}$-module
		$\module{E}$ is called \emph{strong},
		if $\module{E}_\Null \subseteq \module{E}_\Total$ is a 
		$\algebra{A}_\Total$-submodule and
		$\module{E}_\Total \cdot \algebra{A}_\Null \subseteq \module{E}_\Null$
		holds.
		\item The category of constraint $\algebra{A}$-modules
		is denoted by $\injConMod(\algebra{A})$.
		The full subcategory of strong constraint modules is denoted by
		$\injstrConMod(\algebra{A})$.
	\end{definitionlist}
\end{definition}

It should be clear that morphisms of constraint modules can be generalized
to morphisms along a morphism of the underlying constraint algebras.

\begin{proposition}[Mono- and epimorphisms in $\injConMod(\algebra{A})$]
	\label{prop:MonoEpisConModk}
	Let $\Phi \colon \module{E} \to \module{F}$ be a morphism of 
	constraint $\algebra{A}$-modules.
	\begin{propositionlist}
		\item $\Phi$ is a monomorphism if and only if $\Phi$ is injective.
		\item $\Phi$ is an epimorphism if and only if $\Phi$ and 
		$\Phi_\Wobs$ are surjective.
		\item $\Phi$ is a regular monomorphism if and only if it is a 
		monomorphism with
		$\Phi_\Wobs^{-1}(\module{F}_\Null) = \module{E}_\Null$.
		\item $\Phi$ is a regular epimorphism if and only if it is an 
		epimorphism with
		$\Phi_\Wobs(\module{E}_\Null) = \module{F}_\Null$.
		\item $\Phi$ is an isomorphism if and only if $\Phi$ is a monomorphism and a regular epimorphism.
		\item $\Phi$ is an isomorphism if and only if $\Phi$ is a regular monomorphism and an epimorphism.
	\end{propositionlist}
\end{proposition}

\begin{definition}[Constraint submodule]
	Let $\module{E}$ be a constraint $\algebra{A}$-module.
	A \emph{constraint submodule} of $\module{E}$ 
	consists of submodules
	$\module{F}_\Total \subseteq \module{E}_\Total$
	and $\module{F}_\Wobs \subseteq \module{E}_\Wobs$
	with
	$\module{F}_\Wobs \subseteq \module{F}_\Total$.
\end{definition}

Every submodule can be understood as a regular monomorphism
$i \colon \module{F} \to \module{E}$
defined on the constraint module
$\module{F} = (\module{F}_\Total, \, \module{F}_\Wobs, \, i_\Wobs^{-1}(\module{E}_\Null))$.

\begin{definition}[Constructions of constraint modules]
	\label{def:ConstructionsConMod}
	Let $\Phi \colon \module{E} \to \module{F}$
	be a morphism of constraint $\algebra{A}$-modules.
	\begin{definitionlist}
		\item The \emph{kernel} of $\Phi$ is defined by
		\begin{equation}
			\ker(\Phi) \coloneqq \big( \ker(\Phi_\Total), \,\,
			\ker(\Phi_\Wobs), \,\,
			\ker(\Phi_\Wobs) \cap \module{E}_\Null \big).
		\end{equation}
		\item The \emph{image} of $\Phi$ is defined by
		\begin{equation}
			\image(\Phi) \coloneqq 
			\big( \image(\Phi_\Total), \,\,
			\image(\Phi_\Wobs), \,\,
			\image(\Phi_\Wobs\at{\module{E}_\Null}) \big).
		\end{equation}
		\item The \emph{regular image} of $\Phi$ is defined by
		\begin{equation}
			\regimage(\Phi) \coloneqq 
			\big( \image(\Phi_\Total), \,\,
			\image(\Phi_\Wobs), \,\,
			\image(\Phi_\Wobs) \cap \module{F}_\Null \big).
		\end{equation}
		\item The \emph{direct sum} of $\module{E}$ and $\module{F}$ is defined by
		\begin{equation}
			(\module{E} \oplus \module{F}) \coloneqq 
			\big( \module{E}_\Total \oplus \module{F}_\Total, \,\,
			\module{E}_\Wobs \oplus \module{F}_\Wobs, \,\,
			\module{E}_\Null \oplus \module{F}_\Null \big).
		\end{equation}
	\end{definitionlist}
\end{definition}

Observe that the regular image is a constraint submodule, while the image is not.

\begin{proposition}[Module structure on module morphisms]
	\label{prop:ModuleStructureOnModuleHom}
	Let $\algebra{A}$ be a constraint algebra and let
	$\module{E}, \module{F} \in \injConMod(\algebra{A})$.
	Then the $\algebra{A}$-module morphisms form a constraint 
	$\field{k}$-module given by
	\begin{equation}
		\begin{split}
			\ConHom_{\algebra{A}}(\module{E},\module{F})_\Total
			&\coloneqq \Hom_{\algebra{A}_\Total}(\module{E}_\Total, 
			\module{F}_\Total),\\
			\ConHom_{\algebra{A}}(\module{E},\module{F})_\Wobs
			&\coloneqq \big\{ \Phi \in 
			\Hom_{\algebra{A}_\Total}(\module{E}_\Total, 
			\module{F}_\Total) \mid
			\Phi(\module{E}_\Wobs) \subseteq \module{F}_\Wobs
			\text{ and } \Phi_\Wobs(\module{E}_\Null) 
			\subseteq \module{F}_\Null \big\},\\
			\ConHom_{\algebra{A}}(\module{E},\module{F})_\Null
			&\coloneqq \left\{ \Phi \in 
			\Hom_{\algebra{A}_\Total}(\module{E}_\Total,\module{F}_\Total) \mid 
			\Phi(\module{E}_\Wobs) \subseteq \module{F}_\Null 
			\right\}.
		\end{split}
	\end{equation}
	If $\algebra{A}$ is commutative $\ConHom_{\algebra{A}}(\module{E},\module{F})$
	becomes a constraint $\algebra{A}$-module.
\end{proposition}

\begin{definition}[Dual module]
	\label{def:DualModule}
	Let $\algebra{A} \in \injConAlg$ and let $\module{E} \in \injConMod(\algebra{A})$ 
	be a constraint $\algebra{A}$-module.
	We call the constraint $\field{k}$-module
	$\module{E}^* \coloneqq \ConHom_{\algebra{A}}(\module{E},\algebra{A})$
	the \emph{dual module of $\module{E}$}.
	If $\algebra{A}$ is commutative $\module{E}^*$ becomes a constraint $\algebra{A}$-module.
\end{definition}

Another important construction for constraint modules is that of tensor products.
For strong constraint modules there exist two slightly different notions of tensor product.

\begin{definition}[Constraint tensor products]
	\label{def:TensorProducts}
	Let $\algebra{A} \in \injstrConAlg$ be commutative and let
	$\module{E},\module{F} \in \injstrConMod(\algebra{A})$.
	\begin{definitionlist}
		\item The constraint module 
		$\module{E} \tensor[\algebra{A}] \module{F} \in \injstrConMod(\algebra{A})$
		defined by
		\begin{equation}
			\begin{split}
				(\module{E} \tensor[\algebra{A}] \module{F})_\Total
				&\coloneqq \module{E}_\Total \tensor[\algebra{A}_\Total] 
				\module{F}_\Total, \\
				(\module{E} \tensor[\algebra{A}] \module{F})_\Wobs
				&\coloneqq \module{E}_\Wobs \tensor[\algebra{A}_\Total] \module{F}_\Wobs,\\
				(\module{E} \tensor[\algebra{A}] \module{F})_\Null
				&\coloneqq \module{E}_\Null \tensor[\algebra{A}_\Total] \module{F}_\Wobs 
				+ \module{E}_\Wobs \tensor[\algebra{A}_\Total] \module{F}_\Null,
			\end{split}
		\end{equation}
		is called the \emph{tensor product} of $\module{E}$ and $\module{F}$.
		\item The constraint module 
		$\module{E} \strtensor[\algebra{A}] \module{F} \in \injstrConMod(\algebra{A})$
		defined by
		\begin{equation}
			\begin{split}
				(\module{E} \strtensor[\algebra{A}] \module{F})_\Total
				&\coloneqq \module{E}_\Total \tensor[\algebra{A}_\Total] 
				\module{F}_\Total,\\
				(\module{E} \strtensor[\algebra{A}] \module{F})_\Wobs
				&\coloneqq \module{E}_\Wobs \tensor[\algebra{A}_\Total] \module{F}_\Wobs 
				+ \module{E}_\Null \tensor[\algebra{A}_\Total] \module{F}_\Total 
				+ \module{E}_\Total \tensor[\algebra{A}_\Total] \module{F}_\Null,
				\\
				(\module{E} \strtensor[\algebra{A}] \module{F})_\Null
				&\coloneqq \module{E}_\Null \tensor[\algebra{A}_\Total] \module{F}_\Total 
				+ \module{E}_\Total \tensor[\algebra{A}_\Total] \module{F}_\Null,
			\end{split}
		\end{equation}
		is called the \emph{strong tensor product} of $\module{E}$ and $\module{F}$.
	\end{definitionlist}
\end{definition}

Note that every constraint module $\injConMod(\algebra{A})$
over a constraint algebra $\algebra{A} \in \injConAlg$
is in particular a strong constraint $\field{k}$-module.
Hence the tensor products $\tensor[\field{k}]$ and $\strtensor[\field{k}]$
are also defined for non-strong constraint modules.

The reduction of a (strong) constraint $\algebra{A}$-module $\module{E}$
is the $\algebra{A}_\red$-module defined by
\begin{equation}
	\module{E}_\red \coloneqq \module{E}_\Wobs / \module{E}_\Null.
\end{equation}
This defines a functor $\red \colon \injConMod(\algebra{A}) \to \Modules(\algebra{A}_\red)$.
Similarly we get reduction functors for all the different categories of constraint modules as defined above.
Reduction is compatible with direct sums and inner homs in the following sense, see \cite{dippell.esposito.waldmann:2022a, dippell:2023a}.

\begin{proposition}
	\label{prop:ReductionAndDirectSumsHoms}
Let $\algebra{A} \in \injConAlg$
be a constraint algebra and let
$\module{E},\module{F} \in \injConMod(\algebra{A})$.
\begin{propositionlist}
	\item There exists a natural isomorphism
	$(\module{E} \oplus \module{F})_\red \simeq \module{E}_\red \oplus \module{F}_\red$.
	\item There exist natural isomorphisms
	$(\module{E} \tensor[\field{k}] \module{F})_\red
	\simeq \module{E}_\red \tensor[\field{k}] \module{F}_\red
	\simeq (\module{E} \strtensor[\field{k}] \module{F})_\red$.
	\item There exists a natural injective morphism
	$\ConHom_\algebra{A}(\module{E},\module{F})_\red \to \Hom_{\algebra{A}_\red}(\module{E}_\red,\module{F}_\red)$.
\end{propositionlist}
\end{proposition}

Note that the compatibility of reduction with the two types of tensor products was only shown for $\algebra{A} = \field{k}$.
For a general commutative constraint $\algebra{A}$
the reduction functor may not be compatible with the tensor products
$\tensor[\algebra{A}]$ and $\strtensor{\algebra{A}}$, as the following example shows.

\begin{example}
	Let $\field{k} = \Reals$ and consider the commutative strong constraint $\Reals$-algebra
	$\algebra{A} = (\ComplexNum,\Reals,0)$
	and the $\algebra{A}$-module $\module{E} = (\ComplexNum, \ComplexNum,0)$.
	Then
	$\module{E} \tensor[\algebra{A}] \module{E}
	= (\ComplexNum, \ComplexNum,0)$
	reduces to
	$(\module{E} \tensor[\algebra{A}] \module{E})_\red
	= \ComplexNum$
	while
	$\module{E}_\red \tensor[\algebra{A}_\red] \module{E}_\red
	= \Reals \tensor[\Reals] \Reals
	\simeq \Reals$.
\end{example}
Nevertheless, we will see in \autoref{sec:ConGeometry} that in the case of $\algebra{A}$ being a suitable algebra of functions the reduction functor will be compatible with these tensor products.

\subsection{Constraint Index Sets}
\label{sec:ConIndexSets}

In this small section we introduce constraint index sets
and some simple constructions for these.
Constraint index sets will be needed for the definition of
finitely generated projective constraint modules and will 
be a useful tool when working with sections of constraint vector bundles.

\begin{definition}[Constraint index sets]\
	\label{def:ConstraintIndexSets}
	\begin{definitionlist}
		\item An \emph{(embedded) constraint index set} consists of a 
		set $M_\Total$ together with  
		subsets $M_\Null \subseteq M_\Wobs \subseteq M_\Total$.
		\item A morphism $f \colon M \to N$ of constraint sets $M$ and $N$
		is given by a map
		$f \colon M_\Total \to N_\Total$
		such that
		$f(M_\Wobs) \subseteq N_\Wobs$
		and $f_\Wobs(M_\Null) \subseteq N_\Null$.
		\item The category of constraint index sets and their morphisms 
		is denoted by $\injConIndSet$.
	\end{definitionlist}
\end{definition}

As before we drop the prefix ``embedded'' in the following.

\begin{remark}
	The cardinality $\abs{\argument}$ yields a map from finite constraint index sets to
	\begin{equation}
		\Con\Naturals_0^3 \coloneqq \{(n_\Total, n_\Wobs, n_\Null) \in \Naturals_0^3 \mid n_\Null \leq n_\Wobs \leq n_\Total\},
	\end{equation}
	and isomorphic constraint index sets have the same cardinality.
	Conversely, to every $n \in \Con\Naturals_0^3$ we can associate the finite constraint index set
	$(\{1, \dotsc, n_\Null\}, \{1,\dotsc, n_\Wobs\},\{1, \dotsc, n_\Total\})$.
	We will often use this identification implicitly and, for example, write
	$k \in n_\Wobs$ instead of $k \in \{1, \dotsc, n_\Wobs\}$.
	In particular, when we apply the constructions below to triples of natural numbers
	this means we apply them to their associated finite constraint index sets,
	e.g. $n_\red = n_\Wobs \setminus n_\Null = \{n_\Null+1, \dotsc, n_\Wobs\}$.
\end{remark}

As in the case of constraint modules we can turn the set of morphisms
between constraint index sets into a constraint index set itself:

\begin{proposition}[Closed monoidal structure on $\injConIndSet$]
	\label{def:ClosedMonoidalStructureCSet}
	Let $M, N \in \injConIndSet$ be constraint index sets.
	Setting
	\begin{equation}
		\begin{split}
			\ConMap(M,N)_\Total &\coloneqq \Map(M_\Total,N_\Wobs),\\
			\ConMap(M,N)_\Wobs &\coloneqq \left\{ f \in \Map(M_\Total,N_\Total) \mid f(M_\Wobs) \subseteq N_\Wobs
			\text{ and } f(M_\Null) \subseteq N_\Null \right\}, \\
			\ConMap(M,N)_\Null &\coloneqq \left\{ f \in \Map(M_\Total,N_\Total) \mid f(M_\Wobs) \subseteq N_\Null \right\}
		\end{split}
	\end{equation}
	defines a constraint index set $\ConMap(M,N)$.
\end{proposition}

The following constructions of constraint index sets will be important later on.

\begin{definition}[Tensor products and dual]
	\label{def:TensorProdDualConIndSet}
	Let $M,N \in \injConIndSet$.
	\begin{definitionlist}
		\item The \emph{coproduct} of $M$ and $N$ is defined by
		\begin{equation}
			\begin{split}
				(M \sqcup N)_\Total 
				&\coloneqq M_\Total \sqcup N_\Total, \\
				(M \sqcup N)_\Wobs
				&\coloneqq M_\Wobs \sqcup N_\Wobs, \\
				(M \sqcup N)_\Null 
				&\coloneqq (M_\Null \sqcup N_\Null).
			\end{split}
		\end{equation}
		\item The \emph{tensor product} of $M$ and $N$ is defined by
		\begin{equation}
			\begin{split}
				(M \tensor N)_\Total 
				&\coloneqq M_\Total \times N_\Total, \\
				(M \tensor N)_\Wobs
				&\coloneqq M_\Wobs \times N_\Wobs, \\
				(M \tensor N)_\Null 
				&\coloneqq (M_\Wobs \times N_\Null) \cup (M_\Null \times N_\Wobs).
			\end{split}
		\end{equation}
		\item The \emph{strong tensor product} of $M$ and $N$ is defined by
		\begin{equation}
			\begin{split}
				(M \strtensor N)_\Total  
				&\coloneqq M_\Total \times N_\Total,\\
				(M \strtensor N)_\Wobs 
				&\coloneqq (M_\Wobs \times N_\Wobs) \cup (M_\Total \times N_\Null) \cup (M_\Null \times N_\Total), \\
				(M \strtensor N)_\Null
				&\coloneqq (M_\Total \times N_\Null) \cup (M_\Null \times N_\Total).
			\end{split}
		\end{equation}
		\item The \emph{dual} of $M$ is defined by
		\begin{equation}
			\begin{split}
				(M^*)_\Total &\coloneqq M_\Total,\\
				(M^*)_\Wobs &\coloneqq M_\Total \setminus M_\Null,\\
				(M^*)_\Null &\coloneqq M_\Total \setminus M_\Wobs.
			\end{split}
		\end{equation}
		\item The \emph{reduction} of $M$ is defined by
		\begin{equation}
			M_\red \coloneqq M_\Wobs \setminus M_\Null.
		\end{equation}
	\end{definitionlist}
\end{definition}

The above constructions are compatible in the following way:

\begin{proposition}
	\label{prop:DualsConIndSet}
	Let $M,N \in \injConIndSet$.
	\begin{propositionlist}
		\item \label{prop:DualsConIndSet_1}
		We have
		\begin{equation}
			(M \tensor N)_\red = M_\red \times N_\red.
		\end{equation}
		\item \label{prop:DualsConIndSet_2}
		We have
		\begin{equation}
			(M \strtensor N)_\red = M_\red \times N_\red.
		\end{equation}
		\item \label{prop:DualsConIndSet_3}
		We have 
		\begin{equation}
			(M^*)_\red = M_\red.
		\end{equation}
		\item \label{prop:DualsConIndSet_4}
		We have
		\begin{equation}
			(M\tensor N)^* = M^* \strtensor N^*.
		\end{equation}
		\item \label{prop:DualsConIndSet_5}
		We have
		\begin{equation}
			(M \strtensor N)^* = M^* \tensor N^*.
		\end{equation}
		\item \label{prop:DualsConIndSet_6}
		We have
		\begin{equation}
			(M^*)^* = M.
		\end{equation}
	\end{propositionlist}
\end{proposition}

\subsection{Projective Modules}
\label{sec:ProjectiveModules}

In this section we collect some basics about projective constraint modules.
For more details see \cite{dippell:2023a}.
There are multiple possible ways to introduce projective constraint modules.
We will define them using constraint equivalents of dual bases.

\begin{definition}[Projective modules]
	\label{def:ConProjectiveModules}
	Let $\algebra{A} \in \injstrConAlg$ be a strong constraint algebra and let
	$M \in \injConIndSet$.
	\begin{definitionlist}
		\item A strong constraint right $\algebra{A}$-module
		$\module{E} \in \injstrConMod(\algebra{A})$
		is called \emph{projective} if there exist families
		$(e_n)_{n \in M_\Total} \subseteq \module{E}_\Total$ and
		$(e^n)_{n \in M_\Total} \subseteq (\module{E}_\Total)^* =
		\Hom_{\algebra{A}_\Total}(\module{E}_\Total , \algebra{A}_\Total)$,
		such that
		\begin{equation}
			\label{prop:StrDualBasis_Eq1}
			x = \sum_{n \in M_\Total} e_n e^n(x)
			\quad
		\end{equation}
		for all
		$x \in \module{E}_\Total$,
		where for fixed $x$ only finitely many of the
		$e^n(x)$ differ from $0$.
		Moreover, the following properties need to be satisfied:
		\begin{definitionlist}
			\item It holds
			$e_{n} \in \module{E}_\Wobs$ for
			$n \in M_\Wobs$.
			\item It holds $e_n \in \module{E}_\Null$
			for $n \in M_\Null$.
			\item It holds
			$e^{n} \in (\module{E}^*)_\Wobs$ for $n \in M_\Total 
			\setminus M_\Null = (M^*)_\Wobs$.	
			\item It holds
			$e^n \in (\module{E}^*)_\Null$ for
			$n \in M_\Total \setminus M_\Wobs = (M^*)_\Null$.
		\end{definitionlist}
		\item Such a projective strong constraint module $\module{E}$
		is called \emph{finitely generated projective}, if $M$ can be chosen to be a finite constraint index set.
		\item The full subcategory of $\injstrConMod(\algebra{A})$
		consisting of finitely generated projective strong constraint modules will be denoted by $\strConProj(\algebra{A})$.
	\end{definitionlist}
\end{definition}

The families $(e_n)_{n \in M_\Total} \subseteq \module{E}_\Total$ and
$(e^n)_{n \in M_\Total} \subseteq (\module{E}_\Total)^*$
will be called a dual basis of $\module{P}$ and will often be denoted by
$(\{e_i\}_{i \in M}, \{e^i\}_{i \in M^*})$.

The category $\strConProj(\algebra{A})$ is closed under the usual constructions introduced in \autoref{sec:AlgebraicPreliminaries}, as a straightforward computation shows.
For more details see again \cite{dippell:2023a}.

\begin{proposition}
	Let $\algebra{A} \in \injstrConAlg$ be commutative and let
	$\module{E},\module{F} \in \strConProj(\algebra{A})$ be given with dual bases
	$(\{e_i\}_{i \in M}, \{e^i\}_{i \in M^*})$
	and
	$(\{f_i\}_{i \in N}, \{f^i\}_{i \in N^*})$, respectively.
	\begin{propositionlist}
		\item $\module{E} \oplus \module{F}$ is finitely generated projective with dual basis
		$(\{e_i,f_i\}_{i \in M \sqcup N}, \{e^i,f^i\}_{i \in (M\sqcup N)^*})$.
		\item $\module{E}^*$ is finitely generated projective with dual basis
		$(\{e^i\}_{i \in M^*}, \{e_i\}_{i \in M})$.
		\item $\module{E} \tensor[\algebra{A}] \module{F}$ is finitely generated projective with dual basis
		$(\{e_i \tensor f_j\}_{(i,j) \in M \tensor N}, \{e^i \tensor f^j\}_{(i,j) \in (M \tensor N)^*})$.
		\item $\module{E} \strtensor[\algebra{A}] \module{F}$ is finitely generated projective with dual basis
		$(\{e_i \tensor f_j\}_{(i,j) \in M \strtensor N}, \{e^i \tensor f^j\}_{(i,j) \in (M \strtensor N)^*})$.
	\end{propositionlist}
\end{proposition}

These constructions for finitely generated projective modules now exhibit various compatibilities.

\begin{proposition}
	\label{prop:DualTensorHomIsos}
	Let $\algebra{A} \in \injstrConAlg$ be commutative and  $\module{E}, \module{F} \in 
	\strConProj(\algebra{A})$.
	\begin{propositionlist}
		\item \label{prop:DualTensorHomIsos_1}
		There exists a canonical isomorphism
		$\module{F} \strtensor[\algebra{A}] \module{E}^* \simeq 
		\ConHom_\algebra{A}(\module{E},\module{F})$.
		\item \label{prop:DualTensorHomIsos_2}
		There exists a canonical isomorphism 
		$\module{E}^* \tensor[\algebra{A}] \module{F}^* \simeq 
		(\module{E} \strtensor[\algebra{A}] \module{F})^*$.
		\item \label{prop:DualTensorHomIsos_3}
		There exists a canonical isomorphism
		$\module{E}^* \strtensor[\algebra{A}] \module{F}^* \simeq 
		(\module{E} \tensor[\algebra{A}] \module{F})^*$.
		\item \label{prop:DualTensorHomIsos_4}
		There exists a canonical isomorphism
		$\ConHom_\algebra{A}(\module{E} \tensor[\algebra{A}] \module{F}, 
		\module{G})	\simeq \ConHom_\algebra{A}(\module{F}, \module{G} \strtensor[\algebra{A}] \module{E}^*)$.
	\end{propositionlist}
\end{proposition}

\begin{proof}
	We first construct constraint maps for \ref{prop:DualTensorHomIsos_1},
	\ref{prop:DualTensorHomIsos_2} and \ref{prop:DualTensorHomIsos_3}. 
	In a second step show that these are isomorphisms.
	Thus for \ref{prop:DualTensorHomIsos_1} consider the map
	\begin{equation*}
		\label{eq:ConHomAndTensorDual}
		\module{F}_\Total \tensor[\algebra{A}_\Total] 
		\module{E}^*_\Total
		\ni y \tensor \alpha 
		\mapsto (x \mapsto y \cdot \alpha(x)) 
		\in 
		\Hom_{\algebra{A}_\Total}(\module{E}_\Total,\module{F}_\Total).
	\end{equation*}
	To show that it is a constraint $\algebra{A}$-module morphism
	consider first
	$y \tensor \alpha \in (\module{F} \strtensor 
	\module{E}^*)_\Null$.
	If $x \in \module{E}_\Wobs$, then
	$y \cdot \alpha(x) \in \module{F}_\Null \cdot 
	\algebra{A}_\Total + \module{F}_\Total \cdot \algebra{A}_\Null
	\subseteq \module{F}_\Null$.
	Now let $y \tensor \alpha \in 
	\module{F}_\Wobs \tensor[\algebra{A}_\Total] (\module{E}^*)_\Wobs
	\subseteq (\module{F} \strtensor[\algebra{A}] 
	\module{E}^*)_\Wobs$.
	If $x \in \module{E}_\Null$, then
	$y \cdot \alpha(x) \in \module{F}_\Wobs \cdot \algebra{A}_\Null 
	\subseteq \module{F}_\Null$.
	If $x \in \module{E}_\Wobs$, then
	$y \cdot \alpha(x) \in \module{F}_\Wobs \cdot \algebra{A}_\Wobs
	\subseteq \module{F}_\Wobs$.
	Thus the above map is a constraint morphism.
	
	For \ref{prop:DualTensorHomIsos_2} and \ref{prop:DualTensorHomIsos_3} consider the map
	\begin{equation} \label{eq:TensorProductDualsMap}
		\module{E}^*_\Total \tensor[\algebra{A}_\Total] \module{F}^*_\Total
		\ni \alpha \tensor \beta \mapsto
		(x \tensor y \mapsto \alpha(x)\cdot \beta(y))
		\in (\module{E}_\Total \tensor[\algebra{A}_\Total] \module{F}_\Total)^*.
		\tag{$*$}
	\end{equation}
	We show that this induces a constraint morphism
	$\module{E}^* \tensor[\algebra{A}] \module{F}^* \to 
	(\module{E} \strtensor[\algebra{A}] \module{F})^*$.
	For this let
	$\alpha \tensor \beta \in (\module{E}^* \tensor[\algebra{A}] \module{F}^*)_\Null$.
	\begin{cptitem}
		\item For $x \tensor y \in (\module{E} \strtensor[\algebra{A}] \module{F})_\Null$
		we have
		$\alpha(x)\cdot\beta(y) 
		\in \algebra{A}_\Null \cdot \algebra{A}_\Total + \algebra{A}_\Total \cdot \algebra{A}_\Null
		= \algebra{A}_\Null$.
		\item For $x \tensor y \in \module{E}_\Wobs \tensor[\algebra{A}_\Total] \module{F}_\Wobs
		\subseteq (\module{E} \strtensor[\algebra{A}] \module{F})_\Wobs$
		we have
		$\alpha(x)\cdot\beta(y) 
		\in \algebra{A}_\Null \cdot \algebra{A}_\Wobs + \algebra{A}_\Wobs \cdot \algebra{A}_\Null
		= \algebra{A}_\Null$.
	\end{cptitem}
	Next suppose
	$\alpha \tensor \beta \in (\module{E}^* \tensor[\algebra{A}] \module{F}^*)_\Wobs$.
	\begin{cptitem}
		\item For $x \tensor y \in (\module{E} \strtensor[\algebra{A}] \module{F})_\Null$
		we have
		$\alpha(x)\cdot\beta(y) 
		\in \algebra{A}_\Null \cdot \algebra{A}_\Total + \algebra{A}_\Total \cdot \algebra{A}_\Null
		= \algebra{A}_\Null$.
		\item For $x \tensor y \in \module{E}_\Wobs \tensor[\algebra{A}_\Total] \module{F}_\Wobs
		\subseteq (\module{E} \strtensor[\algebra{A}] \module{F})_\Wobs$
		we have
		$\alpha(x)\cdot\beta(y) 
		\in \algebra{A}_\Wobs \cdot \algebra{A}_\Wobs = \algebra{A}_\Wobs$.
	\end{cptitem}
	This shows that $\module{E}^* \tensor[\algebra{A}] \module{F}^* \to 
	(\module{E} \strtensor[\algebra{A}] \module{F})^*$ is a constraint morphism.
	A completely analogous analysis shows that \eqref{eq:TensorProductDualsMap}
	also induces a constraint morphism 
	$\module{E}^* \strtensor[\algebra{A}] \module{F}^* \to 
	(\module{E} \tensor[\algebra{A}] \module{F})^*$, which is needed for
	\ref{prop:DualTensorHomIsos_3}.

	In all three cases we show that the well-known inverse maps on the $\TOTAL$-components
	are in fact constraint maps, and therefore yield constraint 
	inverses.
	To do this we fix dual bases
	$(\{e_i\}_{i \in M}, \{e^i\}_{i \in M^*})$
	and
	$(\{f_j\}_{j \in N}, \{f^j\}_{j \in N^*})$
	of $\module{E}$ and $\module{F}$, respectively.
	For the first part consider the map 
	\begin{equation*}
		\label{eq:DualTensorHomIsos_InverseHom}
		\Hom_{\algebra{A}_\Total}(\module{E}_\Total,\module{F}_\Total)
		\ni \Phi \mapsto
		\sum_{i \in M_\Total} \Phi(e_i) \tensor e^i
		\in \module{F}_\Total \tensor \module{E}^*_\Total.
		\tag{$**$}
	\end{equation*}
	This is the classical inverse on the $\TOTAL$-component.
	Hence we need to show that \eqref{eq:DualTensorHomIsos_InverseHom}
	is a constraint morphism.
	For this let $\Phi \in 
	\ConHom_\algebra{A}(\module{E},\module{F})_\Wobs$
	be given.
	\begin{cptitem}
		\item If $i \in M_\Total\setminus M_\Wobs = M^*_\Null$, then
		$\Phi(e_i) \tensor e^i \in \module{F}_\Total \tensor 
		\module{E}^*_\Null \subseteq
		(\module{F} \strtensor[\algebra{A}] \module{E}^*)_\Null$.
		\item If $i \in M_\Wobs\setminus M_\Null \subseteq 
		M^*_\Wobs$, then, since in particular $i \in M_\Wobs$ holds,
		we obtain
		$\Phi(e_i) \tensor e^i \in \module{F}_\Wobs \tensor 
		\module{E}^*_\Wobs \subseteq
		(\module{F} \strtensor[\algebra{A}] \module{E}^*)_\Wobs$.
		\item If $i \in M_\Null$, then
		$\Phi(e_i) \tensor e^i \in \module{F}_\Null \tensor 
		\module{E}^*_\Total \subseteq
		(\module{F} \strtensor[\algebra{A}] \module{E}^*)_\Null$.
	\end{cptitem}
	Hence \eqref{eq:DualTensorHomIsos_InverseHom} preserves the 
	$\WOBS$-component.
	To show that it also preserves the $\NULL$-component, we only 
	need to reconsider the second case from above.
	\begin{cptitem}
		\item If $i \in M_\Wobs\setminus M_\Null \subseteq 
		M^*_\Wobs$, then, since in particular $i \in M_\Wobs$ holds,
		we obtain
		$\Phi(e_i) \tensor e^i \in \module{F}_\Wobs \tensor 
		\module{E}^*_\Null \subseteq
		(\module{F} \strtensor[\algebra{A}] \module{E}^*)_\Null$.
	\end{cptitem}
	This shows that \eqref{eq:DualTensorHomIsos_InverseHom}
	is a constraint inverse.

	For part \ref{prop:DualTensorHomIsos_2} we need to show that 
	\begin{equation*}
		\label{eq:DualTensorHomIsos_InverseTensor}
		(\module{E}_\Total 
		\tensor[\algebra{A}_\Total] \module{F}_\Total)^*
		\ni
		\alpha
		\mapsto \sum_{(i,j) \in M_\Total \times N_\Total}
		\alpha(e_i \tensor f_j) \cdot e^i \tensor f^j
		\in \module{E}_\Total^* \tensor[\algebra{A}_\Total] 
		\module{F}_\Total^*
		\tag{$**$}
	\end{equation*}
	defines a constraint morphism 
	$(\module{E} \strtensor[\algebra{A}] \module{F})^*
	\to \module{E}^* \tensor[\algebra{A}] \module{F}^*$.
	Recall that the families
	\begin{align*}
		\big(\{ e_i \tensor f_j \}_{(i,j) \in M \strtensor N},
		\{e^i \tensor f^j\}_{(i,j) \in M^* \tensor N^*}\big)\\
		\shortintertext{and}
		\big(\{ e^i \tensor f^j \}_{(i,j) \in M^* \tensor N^*},
		\{e_i \tensor f_j\}_{(i,j) \in M \strtensor N}\big)
	\end{align*}
	form dual bases of
	$\module{E} \strtensor[\algebra{A}] \module{F}$
	and
	$\module{E}^* \tensor[\algebra{A}] \module{F}^*$, respectively.
	Now suppose $\alpha \in (\module{E} 
	\strtensor[\algebra{A}] \module{F})^*_\Wobs$.
	\begin{cptitem}
		\item If $(i,j) \in (M^* \tensor N^*)_\Null$, then
		$\alpha(e_i \tensor f_j) \cdot e^i \tensor f^j
		\in \algebra{A}_\Total \cdot (\module{E}^* 
		\tensor[\algebra{A}] \module{F}^*)_\Null
		= (\module{E}^* 
		\tensor[\algebra{A}] \module{F}^*)_\Null$.
		\item If $(i,j) \in (M_\Wobs\setminus M_\Null) \times (N_\Wobs\setminus N_\Null)
		\subseteq (M^* \tensor N^*)_\Wobs$, then,
		since also $(i,j) \in (M \strtensor N)_\Wobs$ holds, 
		we get
		\begin{equation*}
			\alpha(\underbrace{e_i \tensor f_j}_{\in (\module{E} 
				\strtensor \module{F})_\Wobs}) \cdot e^i 
			\tensor f^j
			\in \algebra{A}_\Wobs \cdot
			(\module{E}^* \tensor[\algebra{A}] 
			\module{F}^*)_\Wobs
			= (\module{E}^* \tensor[\algebra{A}] 
			\module{F}^*)_\Wobs.
		\end{equation*}
		\item If $(i,j) \in (M \strtensor N)_\Null$, then
		\begin{equation*}
			\alpha(\underbrace{e_i \tensor f_j}_{(\module{E} 
				\strtensor 
				\module{F})_\Null}) \cdot e^i \tensor f^j
			\in \algebra{A}_\Null \cdot (\module{E}^* 
			\tensor[\algebra{A}] \module{F}^*)_\Total
			\subseteq (\module{E}^* \tensor[\algebra{A}] 
			\module{F}^*)_\Null.
		\end{equation*}
	\end{cptitem}
	Thus \eqref{eq:DualTensorHomIsos_InverseTensor}
	preserves the $\WOBS$-component.
	Next take $\alpha \in (\module{E} \strtensor[\algebra{A}] 
	\module{F})^*_\Null$.
	Since $(\module{E} \strtensor[\algebra{A}] \module{F})^*_\Null
	\subseteq (\module{E} \strtensor[\algebra{A}] 
	\module{F})^*_\Wobs$ we only need to check one of the above cases.
	\begin{cptitem}
		\item If $(i,j) \in (M_\Wobs\setminus M_\Null) \times (N_\Wobs\setminus N_\Null)
		\subseteq (M^* \tensor N^*)_\Wobs$, then,
		since also since $(i,j) \in (M \strtensor N)_\Wobs$ holds, 
		we get
		\begin{equation*}
			\alpha(\underbrace{e_i \tensor f_j}_{\in (\module{E} 
				\strtensor \module{F})_\Wobs}) \cdot e^i 
			\tensor f^j
			\in \algebra{A}_\Null \cdot
			(\module{E}^* \tensor[\algebra{A}] 
			\module{F}^*)_\Wobs
			= (\module{E}^* \tensor[\algebra{A}] 
			\module{F}^*)_\Null.
		\end{equation*}
	\end{cptitem}
	Hence \eqref{eq:DualTensorHomIsos_InverseTensor}
	also preserves the $\NULL$-component, showing that it is 
	a constraint inverse.
	
	For part \ref{prop:DualTensorHomIsos_3} we would proceed similarly,
	but this time we would need to show that 
	\eqref{eq:DualTensorHomIsos_InverseTensor}
	defines a constraint morphism
	$(\module{E} \tensor[\algebra{A}] \module{F})^*
	\to \module{E}^* \strtensor[\algebra{A}] \module{F}^*$.
	
	Part \ref{prop:DualTensorHomIsos_4} then follows using the above isomorphisms:
	\begin{align*}
		\ConHom_\algebra{A}(\module{E} \tensor[\algebra{A}] \module{F}, 
		\module{G})
		\simeq \module{G} \strtensor[\algebra{A}] (\module{E}
		\tensor[\algebra{A}] \module{F})^*
		\simeq \module{G} \strtensor[\algebra{A}] \module{E}^*
		\strtensor[\algebra{A}] \module{F}^*
		\simeq \ConHom_\algebra{A}(\module{F}, \module{G} 
		\strtensor[\algebra{A}] \module{E}^*).
	\end{align*}
\end{proof}

In particular it follows from \autoref{prop:DualTensorHomIsos} \ref{prop:DualTensorHomIsos_1}
that also $\ConHom_\algebra{A}(\module{E},\module{F})$
is finitely generated projective.

The notion of projectivity is compatible with the reduction functor 
of strong constraint modules:

\begin{proposition}[Reduction of projective strong constraint modules]
	\label{prop:ProjectiveStrReduction}
	Let $\algebra{A} \in \injstrConAlg$ 
	and $\module{E} \in \strConProj(\algebra{A})$
	with dual basis given by
	$(\{e_i\}_{i \in M}, \{e^i\}_{i \in M^*})$.
	Then $\module{E}_\red$ is finitely generated projective, with dual basis
	$(\{(e_i)_\red\}_{i \in M_\red}, \{(e^i)_\red\}_{i \in M_\red})$
\end{proposition}

\subsection{Graded Constraint (Lie) Algebras}
\label{sec:ConLieAlgebras}

Let us quickly recall the notion of graded constraint $\field{k}$-modules
from \cite{dippell.esposito.waldmann:2022a}.
If not mentioned otherwise we will always consider $\Integers$-gradings:
A \emph{(embedded) graded constraint module}
$\module{M}^\bullet$ is given by a graded
$\field{k}$-module
$\module{M}_\Total^\bullet$ together with graded
submodules $\module{M}^\bullet_\Null \subseteq \module{M}^\bullet_\Wobs \subseteq \module{M}^\bullet_\Total$.
For every $x \in \module{M}_\Total^k$ we denote its degree by $\deg(x) = k$. 
A \emph{morphism between graded constraint 
$\field{k}$-modules of degree $k$} between $\module{M}^\bullet$ and 
$\module{N}^\bullet$ is a $\field{k}$-module morphism 
$\Phi \colon \module{M}_\Total^\bullet \to \module{N}_\Total^{\bullet + k}$, such that 
\begin{equation}
	\Phi(\module{M}^\bullet_\Wobs) \subseteq \module{N}^{\bullet + k}_\Wobs 
	\qquad \text{and} \qquad
	\Phi(\module{M}^\bullet_\Null) \subseteq \module{N}^{\bullet + k}_\Null. 
\end{equation}

An \emph{(embedded) graded constraint $\field{k}$-algebra} is an 
graded constraint $\field{k}$-module 
$\algebra{A}^\bullet 
= (\algebra{A}^\bullet_\Total, \algebra{A}^\bullet_\Wobs, \algebra{A}^\bullet_\Null)$
together with an associative multiplication given by a graded constraint morphism of degree $0$
\begin{equation}
	\mu \colon \algebra{A}^\bullet \tensor[\field{k}] \algebra{A}^\bullet \to \algebra{A}^\bullet,\qquad
	\mu(a \tensor b) = a \cdot b,
\end{equation}
together with a unit $1_\algebra{A} \in \algebra{A}_\Wobs$.
If the multiplication is defined as a graded constraint morphism of 
degree $0$
\begin{equation}
	\mu \colon \algebra{A}^\bullet \strtensor[\field{k}] \algebra{A}^\bullet 
	\to \algebra{A}^\bullet,\qquad
	\mu(a \tensor b) = a \cdot b,
\end{equation} 
then $\algebra{A}^\bullet$ is called \emph{strong}.
The grading on the tensor products $\tensor$ and $\strtensor$
is given as usual by
$\deg(a \tensor b) = \deg(a) + \deg(b)$.

Note, a graded constraint $\field{k}$-algebra $\algebra{A}^\bullet$ 
concentrated in degree $0$ is a constraint $\field{k}$-algebra.
Recall that a homogeneous ideal is an ideal such that all projections to a fixed degree are again contained in it, see \cite{hartshorne:1977a} for a definition.
In general, the $\NULL$-component of a graded constraint algebra 
$\algebra{A}^\bullet$ forms a homogeneous algebra ideal in $\algebra{A}^\bullet_\Wobs$ (and 
in $\algebra{A}^\bullet_\Total$, if $\algebra{A}^\bullet$ is strong). 
For another graded constraint $\field{k}$-algebra $\algebra{B}^\bullet$ 
a \emph{morphism of graded constraint $\field{k}$-algebras} between 
$\algebra{A}^\bullet$ and $\algebra{B}^\bullet$ is a degree $0$ graded 
constraint $\field{k}$-module morphism 
$\Phi \colon \algebra{A}^\bullet \to \algebra{B}^\bullet$, which preserves the 
graded multiplication.  

\begin{example}\
	\label{ex:gradedConAlgs}
		Let $\algebra{A}$ be a commutative constraint algebra and $\module{E}$ 
		be a constraint $\algebra{A}$-module. 
		Then the exterior algebra 
		\begin{equation}
			\Anti^\bullet_{\tensor}\module{E} 
			= \bigoplus^{\infty}_{i = 0} \Anti^i_{\tensor} \module{E}
		\end{equation} 
		with respect to $\tensor_{\algebra{A}}$ forms an $\Naturals_0$-graded 
		constraint algebra with 
		$\Anti^0_{\tensor}\module{E} = \algebra{A}$. 
		Assuming $\algebra{A}$ and $\module{E}$ to be a strong constraint 
		algebra and a strong constraint module, respectively, leads to 
		the strong graded constraint algebra $\Anti^\bullet_{\strtensor}\module{E}$ 
		with respect to $\strtensor_{\algebra{A}}$ by an analogous 
		construction. 
\end{example}

\begin{definition}[Differential graded constraint algebras]
	An \emph{(embedded) differential graded\newline constraint algebra} is a
	graded constraint algebra $\algebra{A}^\bullet$ 
	together with a morphism 
	$\D \colon \algebra{A}^\bullet \to \algebra{A}^{\bullet + 1}$ on the 
	underlying graded constraint $\field{k}$-module, such that
	\begin{equation}
		\D^2 = 0,
	\end{equation}
	and which satisfies the \emph{graded Leibniz rule}
	\begin{equation}
		\D(a \cdot b) = \D(a) \cdot b + (-1)^{\deg(a)}a \cdot \D(b)
	\end{equation} 
	for $a, b \in \algebra{A}^\bullet_\Total$.
	It is called differential graded strong constraint algebra 
	if $\algebra{A}$ is a graded strong constraint algebra.
\end{definition}

Every differential graded constraint algebra $\algebra{A}$
can be reduced to a differential graded algebra given by
$\algebra{A}_\red = \algebra{A}_\Wobs / \algebra{A}_\Null$
with differential $\D_\red$.

A classical definition of the following can be found among others in \cite{schwarz:1993a}. 

\begin{definition}[Graded constraint $\field{k}$-Lie algebra]\
	\label{def:ConstraintLieAlgebra}
	\begin{definitionlist}
		\item An \emph{(embedded) graded constraint Lie algebra of degree $k$} 
		is given by an graded constraint $\field{k}$-module 
		$\liealg{g}^\bullet = (\liealg{g}^\bullet_\Total,\liealg{g}^\bullet_\Wobs,\liealg{g}^\bullet_\Null)$
		together with a graded constraint morphism of degree $k$
		\begin{equation}
			\label{eq:ConLieBracket}
			[\argument, \argument ] \colon \liealg{g} \tensor[\field{k}] \liealg{g} \to \liealg{g},
		\end{equation}
		called \emph{graded Lie bracket of degree $k$}, such that
		\begin{definitionlist}
			\item $\left[\xi, \eta\right] 
			= -(-1)^{(\deg(\xi) + k)(\deg(\eta) + k)} \left[\eta, \xi\right]$ and 
			\item $\left[\xi, \left[\eta, \chi\right]\right] 
			= \left[\left[\xi, \eta\right], \chi\right] 
			+ (-1)^{(\deg(\xi) + k)(\deg(\eta) + k)}
			\left[\eta, \left[\xi, \chi\right]\right]$ 
		\end{definitionlist}
		for all $\xi, \eta, \chi \in \liealg{g}_\Total$.
		\item A morphism $\Phi \colon \liealg{g} \to \liealg{h}$
		of graded constraint Lie algebras is a morphism of graded constraint 
		$\field{k}$-modules which forms a commutative square with the graded 
		Lie brackets viewed as graded constraint morphisms, i.e. the following 
		diagram commutes: 
		\begin{equation}
			\begin{tikzcd}
				{\liealg{g} \tensor[\field{k}] \liealg{g}}
				\arrow[r,"{[\argument, \argument]}_{\liealg{g}}"]
				\arrow[d,"\Phi \tensor \Phi"{swap}]
				& {\liealg{g}}
				\arrow[d,"\Phi"] \\
				{\liealg{h} \tensor[\field{k}] \liealg{h}}
				\arrow[r,"{[\argument, \argument]}_{\liealg{h}}"]
				& {\liealg{h}} 
			\end{tikzcd}
		\end{equation}
		\item The category of graded constraint Lie algebras will be denoted by
		$\injConLieAlg^\bullet$.
	\end{definitionlist}
\end{definition}

A graded constraint Lie algebra $\liealg{g}^\bullet$ of degree $0$ concentrated in $\liealg{g}^0$
is simply given by a constraint $\field{k}$-module together with a constraint morphism
$[\argument, \argument] \colon \liealg{g} \tensor[\field{k}] \liealg{g} \to \liealg{g}$
fulfilling the usual properties of a Lie bracket.
These objects will simply be called \emph{constraint Lie algebras}.

\begin{example}\
	\label{ex:ConLieAlg}
	\begin{examplelist}
		\item Let $\module{E}$ be a constraint $\field{k}$-module.
		The endomorphisms
		$\ConEnd_\field{k}(\module{E}) \coloneq \ConHom_{\field{k}}(\module{E}, \module{E})$
		form a constraint Lie algebra with Lie bracket given by the usual commutator
		$[\argument,\argument]^{\module{E}_\Total}$ on
		$\ConEnd_\field{k}(\algebra{E})_\Total$.
		\item \label{ex:ConLieAlg_Derivations}
		Let $\algebra{A} \in \injConAlg$ be a constraint algebra.
		The constraint derivations $\ConDer(\algebra{A})$
		defined by
		\begin{equation}
			\label{eq:ConDerivations}
		\begin{split}
			\ConDer(\algebra{A})_\Total
			&:= \Der(\algebra{A}_\Total),\\
			\ConDer(\algebra{A})_\Wobs
			&:=	\big\{D \in \Der(\algebra{A}_\Total)
			\; \big| \;
			D(\algebra{A}_\Wobs) \subseteq \algebra{A}_\Wobs
			\text{ and }
			D(\algebra{A}_\Null) \subseteq \algebra{A}_\Null \big\},	\\
			\ConDer(\algebra{A})_\Null
			&:=	\big\{D\in \Der(\algebra{A}_\Total)
			\; \big| \;
			D(\algebra{A}_\Wobs) \subseteq \algebra{A}_\Null \big\}
		\end{split}
		\end{equation}
		form a constraint Lie algebra
		which can be seen as a constraint Lie subalgebra
		of $\ConEnd_\field{k}(\algebra{A})$.
	\end{examplelist}
\end{example}

Note that we used $\tensor[\field{k}]$ and not the strong tensor product 
$\strtensor[\field{k}]$ in \eqref{eq:ConLieBracket}.
The reason being that $\ConDer(\algebra{A})$ would in general not be a constraint Lie algebra
if we would have used $\strtensor[\field{k}]$.

Since by \eqref{eq:ConLieBracket} for every graded constraint Lie algebra $\liealg{g}$ we know that
$\liealg{g}_\Null$ is a Lie ideal inside $\liealg{g}_\Wobs$ we obtain a well-defined reduction functor
\begin{equation}
	\red \colon \injConLieAlg^\bullet \to \LieAlgs^\bullet
\end{equation}
by mapping every graded constraint Lie algebra $\liealg{g}$
to $\liealg{g}_\red \coloneqq \liealg{g}_\Wobs / \liealg{g}_\Null$.

According to \autoref{prop:ReductionAndDirectSumsHoms}
there is a natural injection
\begin{equation}
	\label{eq:InsertionReductionDerivations}
	\ConDer(\algebra{A})_\red \hookrightarrow \Der(\algebra{A}_\red)
\end{equation}
for any constraint algebra $\algebra{A}$.
But in general we can not expect to have equality here.
In the geometric situation we will find that 
\eqref{eq:InsertionReductionDerivations}
actually becomes an isomorphism.

A final algebraic concept we will need is that of a constraint Gerstenhaber algebra. 

\pagebreak

\begin{definition}[Constraint Gerstenhaber algebras]\
	An \emph{(embedded) constraint Gerstenhaber algebra} is a graded constraint algebra 
	$\algebra{G}^\bullet
	= (\algebra{G}^\bullet_\Total, \algebra{G}^\bullet_\Wobs, \algebra{G}^\bullet_\Null)$,
	which is equipped with an embedded 
	graded constraint Lie algebra structure 
	$\GerstBracket{\argument, \argument}$ of degree $-1$, 
	where $\GerstBracket{\argument, \argument}$ is called 
	\emph{constraint Gerstenhaber bracket}, such that the 
	graded Leibniz rule
	\begin{equation}
		\GerstBracket{\xi, \eta \cdot \chi} 
		= \GerstBracket{\xi, \eta} \cdot \chi 
		+ (-1)^{(\deg(\xi) - 1) \deg(\eta)} \eta \cdot \GerstBracket{\xi, \chi}
	\end{equation}
	is satisfied for every
	$\xi, \eta, \chi \in \algebra{G}_\Total$. 
	A constraint Gerstenhaber algebra 
	$\algebra{G}^\bullet$ is called 
	\emph{strong}, if the underlying graded constraint 
	algebra is strong. 
\end{definition}

Since every constraint Gerstenhaber algebra is in particular a graded constraint algebra and a graded constraint Lie algebra we know that
$\algebra{G}_\red$ carries the structures of a graded algebra and a graded Lie algebra.
Moreover, it is easy to see that the graded Leibniz rule on $\algebra{G}$
induces a graded Leibniz rule on $\algebra{G}_\red$, showing that $\algebra{G}_\red$ forms itself a Gerstenhaber algebra.

\section{Constraint Geometry}
\label{sec:ConGeometry}

In this chapter we study the reduction of geometric structures by applying the constraint approach
from \autoref{sec:AlgebraicPreliminaries} to differential geometry.
\autoref{sec:ConManifolds} introduces constraint manifolds as our main objects of interest and
shows that functions on such constraint manifolds yields constraint algebras.
After defining constraint vector bundles and their most important constructions in \autoref{sec:ConVectorBundles}
we take a closer look at their relation with constraint modules in \autoref{sec:ConSections}.
In particular, we prove a constraint version of the classical Serre-Swan Theorem,
showing that taking sections gives an equivalence of categories between constraint vector bundles
and finitely generated projective constraint modules in the sense of \autoref{sec:ProjectiveModules}.
Finally, in \autoref{sec:CartanCalculus} a Cartan calculus for constraint manifolds.

\subsection{Constraint Manifolds}
\label{sec:ConManifolds}

The basic motivation for the definition of a constraint manifold is that it is supposed
to capture the minimal geometric information needed to do reduction.
For us this means we have a smooth manifold $M$
together with an embedded submanifold $C$ which itself
carries an equivalence relation allowing for a smooth quotient.
In the following we will always assume that this equivalence relation is given by
a distribution $D$ on $C$.
Distributions admitting a smooth leaf space will be called \emph{simple}.
Note that these are in particular regular and involutive.

\begin{definition}[Constraint manifold]\
	\label{def:ConstraintManifold}
	\begin{definitionlist}
		\item A \emph{constraint manifold} $\mathcal{M} = (M_\Total,M_\Wobs,D_\mathcal{M})$
		consists of a smooth manifold $M_\Total$, a closed embedded submanifold
		$\iota_\mathcal{M} \colon M_\Wobs \to M_\Total$ and a simple distribution 
		$D_\mathcal{M} \subseteq TM_\Wobs$ on $M_\Wobs$.
		\item A smooth map
		$\phi \colon \mathcal{M} \to \mathcal{N}$
		(or \emph{constraint map}) between constraint 
		manifolds is given by a smooth map
		$\phi \colon M_\Total \to N_\Total$
		such that $\phi(M_\Wobs) \subseteq N_\Wobs$
		and $T\phi(D_\mathcal{M}) \subseteq D_\mathcal{N}$.
		\item The category of constraint manifolds and smooth maps is 
		denoted by
		$\ConMfld$.
	\end{definitionlist}
\end{definition}

If we consider only a single constraint manifold we will often write
$\mathcal{M} = (M,C,D)$, with $D \subseteq TC$ the distribution on 
the closed submanifold $C \subseteq M$, instead of using subscripts.
Additionally, we will denote the inclusion of $C$ in $M$
by $\iota \colon C \to M$.

We will call the finite constraint index set
$\dim(\mathcal{M}) \coloneqq \big(\dim(M), \dim(C), \rank(D) \big)$
the \emph{constraint dimension} of the constraint manifold
$\mathcal{M} = (M,C,D)$.

Since for every constraint manifold
$\mathcal{M} = (M,C,D)$
the distribution $D$ is simple we always obtain a smooth reduced manifold
\begin{equation}
	\mathcal{M}_\red \coloneqq C / D.
\end{equation}
Thus we can define the reduction functor $\red \colon \ConMfld \to \Manifolds$
by mapping every constraint manifold $\mathcal{M}$ to
$\mathcal{M}_\red$ and every constraint morphism $\phi \colon \mathcal{M} \to \mathcal{N}$ to
\begin{equation}
	\phi_\red \colon \mathcal{M}_\red \to \mathcal{N}_\red,
	\qquad
	\phi_\red([p]) \coloneqq [\phi(p)].
\end{equation}

\begin{example}\
	\label{ex:ConManifolds}
	\begin{examplelist}
		\item \label{ex:ConManifolds_LieGroupAction}
		Let $\group{G}$ be a Lie group acting via
		$\Phi \colon \group{G} \times C \to C$ in a free and proper way 
		on a closed submanifold $C \subseteq M$.
		Then the images of the infinitesimal action
		$T_e\Phi_p \colon \liealg{g} \to T_pC$, for all $p \in C$, define 
		a simple distribution on $C$, inducing the structure of a 
		constraint manifold.
		The reduction is then given by the orbit space $C / \group{G}$.
		\item \label{ex:ConManifolds_Coisotropic}
		Let $C \subseteq M$ be a closed coisotropic submanifold of a Poisson 
		manifold $(M,\pi)$.
		Then, if the characteristic distribution
		$D$ is simple, $\mathcal{M} = (M,C,D)$
		defines a constraint manifold, and the reduction is given by
		$\mathcal{M}_\red = C / D$.
		\item \label{ex:ConManifolds_Euclidean}
		Let $d = (d_\Total, d_\Wobs, d_\Null) \in \mathbb{N}^3$
		such that $d_\Total \geq d_\Wobs \geq d_\Null$.
		Then $\Reals^{d_\Wobs} \subseteq \Reals^{d_\Total}$
		together with the distribution 
		$T\Reals^{d_\Null} \subseteq T\Reals^{d_\Wobs}$
		defines a constraint manifold, 
		which we denote by
		$\Reals^d = (\Reals^{d_\Total}, \Reals^{d_\Wobs}, \Reals^{d_\Null})$.
		The reduction is then given by
		$(\Reals^d)_\red \simeq \Reals^{d_\Wobs - d_\Null} = \Reals^{d_\red}$.
	\end{examplelist}
\end{example}

Constraint submanifolds can be introduced as (equivalence classes of)
constraint injective immersions.
To do this we would need the additional concept of constraint sets, see \cite{dippell:2023a}.
Instead let us give a more direct definition of constraint submanifolds as follows.
Note that all submanifolds are considered to be embedded.

\begin{definition}[Constraint submanifold]
	\label{def:ConSubMfld}
	A \emph{constraint submanifold} $\mathcal{N}$
	of a constraint manifold $\mathcal{M} = (M_\Total,M_\Wobs,D)$
	consists of a submanifold
	$N_\Total \subseteq M_\Total$
	and a submanifold
	$N_\Wobs \subseteq M_\Wobs$
	such that
	$N_\Wobs \subseteq N_\Total$
	is a closed submanifold
	and the distribution $D$ is tangent to $N_\Wobs$,
	i.e.
	$D\at{N_\Wobs} \subseteq TN_\Wobs$.
\end{definition}

We will also write $\mathcal{N} \subseteq \mathcal{M}$
for a constraint submanifold.

\begin{lemma}[Local structure of constraint manifolds]
	\label{lem:LocalStructureConManfifold}
	Let $\mathcal{M} = (M,C,D)$ be a constraint manifold of dimension 
	$d = (d_\Total, d_\Wobs, d_\Null)$.
	\begin{lemmalist}
		\item If $U \subseteq M$ is open, then
		$\mathcal{M}\at{U} \coloneqq (U, U \cap C, D\at{U})$
		is a constraint submanifold manifold of $\mathcal{M}$.
		\item For every $p \in C$ there exists a coordinate chart $(U,x)$ around $p$
		such that
		\begin{align}
			x(U \cap C) &= \big(\mathbb{R}^{d_\Wobs} \times \{0\}\big) \cap x(U)\\
			\shortintertext{and}
			x(U \cap L_p) &= \big(\mathbb{R}^{d_\Null} \times \{0\}\big) \cap x(U),
		\end{align}
		where $L_p$ denotes the leaf of the distribution $D$ through $p$.
	\end{lemmalist}
\end{lemma}

\begin{proof}
	The first part is clear.
	For the second part choose a foliation chart on $C \cap U$ and extend it as a submanifold chart to $U$.
\end{proof}

\pagebreak

To every constraint manifold we can associate a strong constraint algebra
of smooth functions.

\begin{proposition}[Functions on constraint manifolds]\
	\label{prop:FunctionsOnConManifolds}
	Associating to every constraint manifold $\mathcal{M} = (M,C,D)$ the constraint algebra
	\begin{equation}
		\begin{split}
			\ConCinfty(\mathcal{M})_\Total
			&\coloneqq \Cinfty(M,\mathbb{R}),\\
			\ConCinfty(\mathcal{M})_\Wobs
			&\coloneqq \left\{ f \in \Cinfty(M,\mathbb{R})
			\mid \Lie_X f\at{C} = 0 \text{ for all } X \in \Secinfty(D) 
			\right\},\\
			\ConCinfty(\mathcal{M})_\Null
			&\coloneqq \left\{ f \in \Cinfty(M,\mathbb{R})
			\mid f\at{C} = 0 \right\},
		\end{split}
	\end{equation}
	and to every constraint morphism $\phi \colon \mathcal{M} \to \mathcal {N}$ of constraint manifolds the constraint algebra morphism
	\begin{equation}
		\phi^* \colon \ConCinfty(\mathcal{N}) \to \ConCinfty(\mathcal{M}),
		\qquad
		\phi^*(f) \coloneqq f \circ \phi
	\end{equation}
	yields a functor $\ConCinfty \colon \ConMfld \to 
	\injstrConAlg^\opp$.
\end{proposition}

\begin{proof}
	Note that $\ConCinfty(\mathcal{M})_\Wobs$
	is a subalgebra of $\Cinfty(M,\mathbb{R})$ by the fact that
	$\Lie_X$ is $\mathbb{R}$-linear and satisfies a Leibniz rule.
	The $\NULL$-component is obviously contained in the $\WOBS$-component 
	and, since it is just the vanishing ideal of $C$, it is a two-sided 
	ideal in $\Cinfty(M,\mathbb{R})$.
	This shows that $\ConCinfty(\mathcal{M})$ is indeed an embedded 
	strong constraint algebra.
	Now given a smooth constraint map
	$\phi \colon \mathcal{M} \to \mathcal{N}$
	we have
	$(\phi^*f)(p) = f(\phi(p)) = 0$
	for $f \in \ConCinfty(\mathcal{N})_\Null$
	and all $p\in M_\Wobs$.
	Thus $\phi^*(\ConCinfty(\mathcal{N})_\Null) \subseteq 
	\ConCinfty(\mathcal{N})_\Null$.
	To show that $\phi^*$ also preserves the $\WOBS$-component 
	let $f \in \ConCinfty(\mathcal{N})_\Wobs$ be given.
	Then for $X_p \in D_\mathcal{M}\at{p}$, $p \in \mathcal{M}_\Wobs$
	we have
	$X_p(\phi^*f) = T_p\phi(X_p)f = 0$
	since $T_p\phi(X_p) \in D_\mathcal{N}\at{\phi(p)}$
	by assumption.
	This shows $\phi^*f \in \ConCinfty(\mathcal{M})_\Wobs$.
\end{proof}

\begin{example}\
	\label{ex:ConFunctionsOnConMfld}
	\begin{examplelist}
		\item \label{ex:ConFunctionsOnConMfld_1}
		Let $\mathcal{M} = (M,C,D)$ be a constraint manifold
		of dimension $d = (d_\Total, d_\Wobs, d_\Null)$,
		$p \in C$ and $(U,x)$ an adapted  chart around $p$ as in
		\autoref{lem:LocalStructureConManfifold}.
		Then
		\begin{equation}
			\begin{split}
				x^i &\in \Cinfty(\mathcal{M}\at{U})_\Null \text{ if } i \in \{d_\Wobs+1, \dotsc, d_\Total\} = (d^*)_\Null, \\
				x^i &\in \Cinfty(\mathcal{M}\at{U})_\Wobs \text{ if } i \in \{d_\Null+1, \dotsc, d_\Total\} = (d^*)_\Wobs, \\
				x^i &\in \Cinfty(\mathcal{M}\at{U})_\Total \text{ if } i \in 
				\{1, \dotsc, d_\Total\} = (d^*)_\Total.
			\end{split}
		\end{equation}
		\item \label{ex:ConFunctionsOnConMfld_2}
		Let $C \subseteq M$ be a coisotropic submanifold of a Poisson manifold $(M, \pi)$ and
		denote as before by $\mathcal{M} = (M,C,D)$ the corresponding constraint manifold.
		Then, as for any constraint manifold, 
		the $\NULL$-component
		$\ConCinfty(\mathcal{M})_\Null = \vanishing_C$ is the vanishing ideal of $C$, and
		additionally
		\begin{equation}
			\ConCinfty(\mathcal{M})_\Wobs 
			= \Pnormalizer_C
			= \left\{ f \in \Cinfty(M) \mid \{f,g\} \in \vanishing_C \text{ for all } g \in \vanishing_C \right\}
		\end{equation}
		is the Poisson normalizer of $\vanishing_C$.
	\end{examplelist}
\end{example}

Constructing the embedded strong constraint algebra of smooth functions on a constraint manifold
then commutes with reduction:

\begin{proposition}[Constraint functions vs. reduction]
	\label{prop:ConFunctionsVSReduction}
	There exists a natural isomorphism such that the following 
	diagram commutes:
	\begin{equation}
		\label{prop:ConFunctionsVSReduction_diag}
		\begin{tikzcd}
			\ConMfld
			\arrow[r,"\ConCinfty"]
			\arrow[d,"\red"{swap}]
			&\injstrConAlg^\opp
			\arrow[d,"\red"]\\
			\Manifolds
			\arrow[r,"\Cinfty"]
			&\Algebras^\opp
		\end{tikzcd}
	\end{equation}
\end{proposition}

\begin{proof}
	Observe that every
	$f \in \ConCinfty(\mathcal{M})_\Wobs$
	drops to a function $f_\red \in \Cinfty(\mathcal{M}_\red)$, 
	and the kernel of this map is exactly given by the vanishing ideal
	$\ConCinfty(\mathcal{M})_\Null$.
	Hence we obtain an inclusion
	$\ConCinfty(\mathcal{M})_\red \subseteq \Cinfty(\mathcal{M}_\red)$.
	To show surjectivity of this map, choose a tubular neighbourhood $V$ of
	$C$ with projection $\pr_V \colon V \to C$ and a bump function
	$\chi \in \Cinfty(M,\mathbb{R})$ with $\chi\at{C} = 1$ and 
	$\chi\at{M\setminus V} = 0$.
	Note that the closedness of $C$ is needed for the existence of such a 
	$\chi$.
	Then every function $f \in \Cinfty(\mathcal{M}_\red)$
	can first be pulled back to a function $\pi_\red^*f$ on $C$ and
	afterwards pulled back to $V$ via $\pr_V^*(\pi_\red^*f)$,
	where $\pi_\red \colon C \to \mathcal{M}_\red$ denotes the projection 
	to the quotient.
	Finally, we can extend it to all of $M$ using $\chi$ obtaining
	$\hat{f} \coloneqq \chi\cdot (\pr_V^*(\pi_\red^*f))$.
	Since $\hat{f}\at{C} = \pi_\red^*f$ we clearly get
	$\hat{f} \in \ConCinfty(\mathcal{M})_\Wobs$
	and $(\hat{f})_\red = f$.
	Hence we get $\ConCinfty(\mathcal{M})_\red = 
	\Cinfty(\mathcal{M}_\red)$.
	For the naturality consider a smooth constraint map
	$\phi \colon \mathcal{M} \to \mathcal{N}$.
	Then for every $f \in \Cinfty(\mathcal{N})_\red$ we have
	\begin{equation*}
		(\phi^*)_\red (f_\red)
		= (\phi^*(f))_\red
		= (f \circ \phi)_\red
		= f_\red \circ \phi_\red
		= (\phi_\red)^*(f_\red).
	\end{equation*}
	This shows that \eqref{prop:ConFunctionsVSReduction_diag} commutes up to a natural isomorphism.
\end{proof}

\subsection{Constraint Vector Bundles}
\label{sec:ConVectorBundles}

To introduce the notion of constraint vector bundle
we need the concept of a partial connection 
(or partial covariant derivative), cf. \cite{bott:1972a}.
If $D \subseteq C$ is an involutive regular distribution on a smooth 
manifold $C$ and $E \to C$ is a vector bundle over $C$, then
a \emph{$D$-connection on $E$} is a $\Reals$-linear map
\begin{equation}
	\nabla \colon \Secinfty(D) \tensor[\Reals] \Secinfty(E) \to \Secinfty(E) 
\end{equation}
fulfilling the usual properties of a covariant derivative.
Note that $D$-connections always exist by restricting a 
connection on $E$ to $D$.
Moreover, every $D$-connection can be extended to a 
covariant derivative on $E$ by choosing a complement $D^\perp$ of $D$ inside
$TC$ and a $D^\perp$-connection, then taking the sum of those.
Given a curve $\gamma \colon I \to C$ inside a fixed leaf of $D$
connecting $p,q \in C$ we obtain corresponding parallel transport
$\Parallel_\gamma \colon E_p \to E_q$.
If this parallel transport is independent of the chosen (leafwise) 
path, we will call the $D$-connection $\nabla$ \emph{holonomy-free}.

\begin{definition}[Constraint vector bundle]
	\label{def:ConVectorBundle}
Let $\mathcal{M} = (M_\Total,M_\Wobs,D_\mathcal{M})$ and 
$\mathcal{N} = (N_\Total,N_\Wobs,D_\mathcal{N})$ be constraint manifolds. 
\begin{definitionlist}
	\item A \emph{constraint vector bundle}
	$E = (E_\Total,E_\Wobs,E_\Null,\nabla)$ over $\mathcal{M}$
	consists of a vector bundle $E_\Total \to M_\Total$,
	a subbundle $E_\Wobs \to M_\Wobs$ of $\iota^\#E_\Total$,
	a subbundle $E_\Null \to M_\Wobs$ of $E_\Wobs$
	and a holonomy-free $D_\mathcal{M}$-connection $\nabla$ 
	on $E_\Wobs / E_\Null$.	
	\item Let $E = (E_\Total,E_\Wobs,E_\Null, \nabla^E)$ and 
	$F = (F_\Total, F_\Wobs, F_\Null,\nabla^F)$ be constraint vector bundles 
	over $\mathcal{M}$ and $\mathcal{N}$, respectively.
	A morphism $\Phi \colon E \to F$ of constraint vector bundles over 
	$\phi \colon \mathcal{M} \to \mathcal{N}$ is given by a vector 
	bundle morphism $\Phi_\Total \colon E_\Total \to F_\Total$
	such that
	\begin{definitionlist}[label=\alph*.)]
		\item $\Phi_\Total$ restricts to a vector bundle morphism
		$\Phi_\Wobs = \Phi_\Total\at{E_\Wobs} \colon E_\Wobs \to F_\Wobs$,
		\item $\Phi_\Wobs(E_\Null) \subseteq F_\Null$ and 
		\item the induced vector bundle morphism 
		$\overline{\Phi} \colon E_\Wobs / E_\Null \to F_\Wobs / F_\Null$ 
		is a morphism between vector bundles with linear connections , i.e.
		\begin{equation}
			\label{eq:ConstraintVBMorphism}
			\nabla^{E^*}_{v_p} \overline\Phi^*\alpha
			= (\overline\Phi_{p})^*\big(\nabla^{F^*}_{T\phi(v_p)} \alpha\big)
		\end{equation}
		holds for all $v_p \in D_\mathcal{M}\at{p}$, $\alpha \in \Secinfty((F_\Wobs/F_\Null)^*)$ and 
		$p \in C$,
		see \autoref{def:MorphLinConnection}. 
	\end{definitionlist}
	\item The category of constraint vector bundles with morphisms over arbitrary smooth maps is denoted by
	$\ConVect$, while the category of constraint vector bundles over a fixed constraint manifold $\mathcal{M}$
	with morphisms over $\id_\mathcal{M}$
	is denoted by $\ConVect(\mathcal{M})$.
\end{definitionlist}
\end{definition}

\begin{remark}
	\label{rem:DistributionsOnVBAsDsVB}
	In classical differential geometry vector bundles can be viewed as manifolds 
	with a linear fibration. 
	Using this viewpoint for constraint geometry means the following:
	Let us start with a constraint vector bundle
	$E = (E_\Total \to M, E_\Wobs \to C, E_\Null \to C, \nabla)$
	over the constraint manifold 
	$\mathcal{M} = (M, C, D_{\mathcal{M}})$, then the corresponding 
	constraint manifold is given by 
	$E = (E_\Total, E_\Wobs, D_{E})$, 
	where the simple distribution $D_{E}$ is  
	equivalent to the information $(D_{\mathcal{M}}, E_\Null, \nabla)$. 
	This was worked out in \cite{jotzlean.ortiz:2014a}, see also 
	\cite{kern:2023a}, where they are called linear distributions. 

	Starting with the triple $(D_{\mathcal{M}}, E_\Null, \nabla)$ we 
	identify $E_\Wobs / E_\Null$ with a subbundle $E_\WobsNotNull$ in $E_\Wobs$, 
	such that $E_\Wobs \simeq E_\Null \oplus E_\WobsNotNull$. 
	Extending $\nabla$ to a linear connection on $E_\Wobs$
	gives an equivalent, but non-canonical, linear splitting $\Sigma$
	of the double vector bundle $TE_\Wobs$,
	and an isomorphism $TE_\Wobs \simeq E_\Wobs \times_{C} TC \times_{C} E_{\Wobs}$,
	called decomposition, see \cite{jotzlean:2018b}. 
	For more information on double vector bundles, see e.g. 
	\cite{graciasaz.mehta:2010a,mackenzie:1992a}. 
	Under this identification we can define
	$D_{E} \coloneqq E_\Wobs \times_{C} D_{\mathcal{M}} \times_{C} E_\Null$
	and thus obtain a double subbundle
	\begin{equation}
		\begin{tikzcd}
			{D_{E}}
			\arrow[r]
			\arrow[d]
			& {D_{\mathcal{M}}}
			\arrow[d] \\ 
			E_\Wobs 
			\arrow[r]
			& C
		\end{tikzcd}
		\qquad \subseteq \qquad 
		\begin{tikzcd}
			{TE_\Wobs} 
			\arrow[r]
			\arrow[d]
			& {TC} 
			\arrow[d] \\
			{E_\Wobs}
			\arrow[r] 
			& {C}
		\end{tikzcd},
	\end{equation}
	see \cite{jotzlean.ortiz:2014a}. 
	Moreover, the linear splitting leads to a direct sum decomposition 
	\begin{equation}
		D_{E} \simeq \Ver(E_\Null) \oplus \Hor^{\Sigma}(D_{\mathcal{M}}),
	\end{equation}
	where $\Ver(E_\Null)$ denotes the vertical bundle over $E_\Wobs$ pointing in 
	$E_\Null$-direction, i.e. 
	\begin{equation}
		\Ver(E_\Null) = \Big\{ \left.\frac{\D}{\D t}\right|_{t = 0}(s_p + t r_p)  \in TE_\Wobs \bigm| s_p \in E_\Wobs, r_p \in E_\Null, p \in C \Big\}, 
	\end{equation}
	and $\Hor^{\Sigma}(D_{\mathcal{M}})$ denotes the horizontally lifted vectors 
	over $D_{\mathcal{M}}$, i.e. 
	\begin{equation}
		\Hor^{\Sigma}(D_{\mathcal{M}}) 
		= \Big\{ \Sigma(s_p, v_p) \in TE_\Wobs \bigm| s_p \in E_\Wobs, v_p \in D_{\mathcal{M}}, p \in C \Big\}. 
	\end{equation}

	Conversely, suppose a double subbundle $D_{E}$ subordinate to $TE_\Wobs$ is given.	
	We define $E_\Null$ to be the core of $D_{E}$ and together with the 
	side bundles $E_\Wobs$ and $D_{\mathcal{M}}$ and a chosen linear splitting $\Omega$, there is an isomorphism 
	$E_\Wobs \times_{C} D_{\mathcal{M}} \times_{C} E_\Null \simeq D_{E}$ 
	of double vector bundles.
	Every such linear splitting $\Omega$ can be extended to a linear splitting 
	$\Sigma$ for $TE_\Wobs$,
	leading to an isomorphism 
	$E_\Wobs \times_{C} TC \times_{C} E_\Wobs \simeq TE_\Wobs$.
	The linear splitting $\Sigma$ corresponds to a linear 
	connection 
	$\widehat{\nabla} \colon \Secinfty(TC) \tensor[\Reals] \Secinfty(E_\Wobs) \to \Secinfty(E_\Wobs)$, 
	which in turn gives rise to a linear 
	$D_{\mathcal{M}}$-connection $\nabla$ on $E_\Wobs / E_\Null$. 
	This leads to the claimed triple $(D_{\mathcal{M}}, E_\Null, \nabla)$. 
\end{remark}

For every constraint vector bundle $E$ over a constraint manifold
$\mathcal{M} = (M,C,D)$
and $p \in M$ we can consider the fibre
$E_\Total\at{p}$.
If $p \in C$ is a point in the submanifold, we have subspaces defined by the subbundles
$E_\Wobs$ and $E_\Null$, leading to a constraint vector space, i.e. a constraint module over the constraint algebra
$(\Reals,\Reals,0)$, given by
\begin{equation}
	E\at{p} \coloneqq \big(E_\Total\at{p}, E_\Wobs\at{p}, E_\Null\at{p}\big).
\end{equation}
For $p \in M\setminus C$ we define $E\at{p} \coloneqq (E_\Total\at{p},0,0)$.
Since $M$ and $C$ are supposed to be connected the dimension of this constraint vector space is independent of the base point $p \in C$.
Thus we call the constraint index set
\begin{equation}
	\rank(E) \coloneqq \big(\rank(E_\Total), \rank(E_\Wobs), \rank(E_\Null)\big)
\end{equation}
the \emph{rank of $E$}.

\begin{example}
	Instances of constraint vector bundles have, under different names, appeared in the literature before.
	\begin{examplelist}
		\item In \cite[Def 2.2]{cabrera.ortiz:2022a} the notion of quotient data 
		$(q_M,K,\Delta)$ for a vector bundle
		$E \to M$ is introduced, see also \cite[§2.1]{mackenzie:2005a}.
		Here $q_M \colon M \to \tilde{M}$ denotes a surjective submersion with connected fibres,
		$K \subseteq E$ is a subbundle and
		$\Delta$ is a smooth assignment taking a pair of points $x,y \in M$ on the same
		$q_M$-fibre to a linear isomorphism
		$\nabla_{x,y} \colon E_y / K_y \to E_x / K_x$.
		This directly gives a constraint vector bundle
		$(E,E,K)$ over $(M,M,\ker(Ty_M))$ with $\bar\nabla$ the 
		partial 
		connection induced by $\Delta$.
		\item By Batchelor's Theorem \cite{batchelor:1980a, bonavolonta.poncin:2013a}
		graded manifolds
		of degree one correspond to vector 
		bundles 
		over manifolds.
		Under this identification a graded submanifold of a degree 
		one graded manifold corresponds
		to a constraint vector bundle $(E,\iota^\#E,F)$ over 
		$(M,C,0)$, see \cite{cueca:2019a}.
	\end{examplelist}
\end{example}

\begin{example}[Trivial constraint vector bundle]\
	\label{ex:TrivialConVectorBundle}
	Let $\mathcal{M} = (M,C,D)$ be a constraint manifold and 
	$k = (k_\Total, k_\Wobs, k_\Null)$ a finite constraint index set.
	Then
	\begin{equation}
		\mathcal{M} \times \Reals^k \coloneqq 
		\big(M \times \mathbb{R}^{k_\Total}, C \times \mathbb{R}^{k_\Wobs}, C \times \mathbb{R}^{k_\Null}, \Lie\big)
	\end{equation}
	defines a constraint vector bundle.
	Here $\Lie$ denotes the component-wise Lie derivative.
\end{example}

We will call a constraint vector bundle of that form \emph{trivial}.
Constraint vector bundles isomorphic to trivial vector bundles will be called
\emph{trivializable}.

\begin{definition}[Constraint sub-bundle]
	\label{def:ConSubBundle}
	Let $E = (E_\Total,E_\Wobs,E_\Null,\nabla^E)$ 
	be a constraint vector bundle over a constraint manifold
	$\mathcal{M} = (M,C,D)$.
	A \emph{constraint sub-bundle} $F$ of $E$ consists
	of a vector sub-bundle $F_\Total \subseteq E_\Total$
	together with a subbundle $F_\Wobs \subseteq E_\Wobs$
	such that the following properties hold:
	\begin{definitionlist}
		\item $F_\Wobs$ is a subbundle of $\iota^\# F_\Total$ over $C$.
		\item The intersection $F_\Wobs \cap E_\Null$ is a subbundle of $E_\Null$ and $F_\Wobs$.
		\item $\nabla^E$ restricts to a $D$-connection on
		$\frac{F_\Wobs + E_\Null}{E_\Null}$.
	\end{definitionlist}
\end{definition}

Note that $\frac{F_\Wobs + E_\Null}{E_\Null} \simeq \frac{F_\Wobs}{F_\Wobs \cap E_\Null}$,
and therefore every constraint sub-bundle $F$
can be seen as a constraint vector bundle
$(F_\Total,F_\Wobs,F_\Wobs \cap E_\Null, \nabla^F)$,
with $\nabla^F$ being the restrictions of $\nabla^E$ to
$\frac{F_\Wobs}{F_\Wobs \cap E_\Null}$.
The inclusion $i \colon F \hookrightarrow E$ then forms a morphism of constraint vector bundles.
In combination with \autoref{def:ConSubMfld} it should be clear how to define constraint sub-bundles on constraint submanifolds, but we will not need this generalization in the following.

\begin{proposition}
	\label{prop:ConKernSubbundle}
	Let $E$ and $F$ be constraint vector bundles over a constraint manifold $\mathcal{M}$.
	Suppose $\Phi \colon E \to F$ is a constraint vector bundle morphism over $\id_\mathcal{M}$ such that
	$\Phi_\Total$, $\Phi_\Wobs$ and
	$\Phi_\Wobs\at{E_\Null}$ have constant rank.
	Then $\ker \Phi$ given by
	\begin{equation}
	\begin{split}
		(\ker \Phi)_\Total &\coloneqq \ker(\Phi_\Total) \\
		(\ker \Phi)_\Wobs &\coloneqq \iota^\#\ker(\Phi_\Total) \cap E_\Wobs
	\end{split}
	\end{equation}
	is a constraint sub-bundle of $E$.
\end{proposition}

\begin{proof}
	Since $\Phi_\Total$ and $\Phi_\Wobs$ have constant rank we know that 
	$(\ker\Phi)_\Total$ and $(\ker\Phi)_\Wobs = \ker(\Phi_\Wobs)$ are subbundles, 
	and clearly
	$(\ker\Phi)_\Wobs \subseteq (\ker\Phi_\Total)$
	is a subbundle.
	Moreover, $(\ker\Phi)_\Wobs \cap E_\Null \simeq \ker(\Phi_\Null)$
	is a subbundle since $\Phi_\Null$ has constant rank.
	Finally, the fact that $\Phi$ is compatible with $\nabla$ shows that $\nabla$ restricts to
	$\frac{\ker(\Phi_\Wobs)}{\ker(\Phi_\Null)}$.
\end{proof}

Recall from \cite{bott:1972a} that for a given manifold $C$ with a regular involutive distribution
$D \subseteq TC$ there exists a canonical $D$-connection on the normal bundle
$TC / D$, the so-called \emph{Bott connection}, given by
\begin{equation} \label{eq:BottConnection}
	\nabla^\Bott_X \cc{Y} = \cc{[X,Y]}
\end{equation}
for $X \in \Secinfty(D)$ and $\cc{Y} \in \Secinfty(TC / D)$.
Here $\cc{Y}$ denotes the equivalence class of
$Y \in \Secinfty(TC)$.
Moreover, it is well-known that since $C/D$ is smooth the Bott connection
is holonomy free.

\begin{definition}[Constraint tangent bundle]
	\label{def:ConTangentBundle}
	Let $\mathcal{M} = (M, C, D)$ be a constraint manifold.
	The constraint vector bundle
	$T\mathcal{M} \coloneqq (TM,TC,D,\nabla^\Bott)$
	over $\mathcal{M}$
	is called the 
	\emph{(constraint) tangent bundle} of $\mathcal{M}$.
\end{definition}

\begin{proposition}[Constraint tangent bundle functor]
	Mapping every constraint manifolds to
	its constraint tangent bundle and smooth maps
	$\phi \colon \mathcal{M} \to \mathcal{N}$ between constraint manifolds
	$\mathcal{M}$ and $\mathcal{N}$ to the tangent map
	$T\phi \colon T\mathcal{M} \to T\mathcal{N}$
	defines a functor
	\begin{equation}
		T \colon \ConMfld \to \ConVect.
	\end{equation}
\end{proposition}

\begin{proof}
	For the $\TOTAL$-components the statement is clear, and since $T\phi$ is completely
	determined by $T\phi \colon TM_\Total \to TN_\Total$ the only thing left to show is that $T\phi$ is actually a constraint morphism.
	Since $\phi$ maps $M_\Wobs$ to $N_\Wobs$
	we immediately get that $\iota^\#T\phi$ restricts to
	$T\phi \colon TM_\Wobs \to TN_\Wobs$.
	Moreover, by \autoref{def:ConstraintManifold} we have
	$T\phi(D_\mathcal{M}) \subseteq D_\mathcal{N}$.
	It remains to show that $T\phi$ is compatible with the Bott connections.
	For this let
	$v_p \in D_\mathcal{M}\at{p}$
	and
	$X \in \Secinfty(D_{\mathcal{M}})$
	such that $X(p) = v_p$.
	Moreover, we know there exist
	open subsets $V \subseteq M$ around $p$
	and
	$U \subseteq N_\Wobs$ around $\phi(p)$
	with $V \subseteq \phi^{-1}(U)$	and
	$Y \in \Secinfty(TU)$ 
	such that
	$X\at{V} \sim_{\phi} Y$.
	Then for arbitrary
	$\alpha \in \Secinfty(\Ann(D_\mathcal{N}))$
	and
	$w_p \in T_pM_\Wobs$
	we compute
	\begin{align*}
		\Big(\nabla^{T^*\mathcal{M}}_{v_p} (\overline{T\phi})^* \overline{\alpha}\Big) (\overline{w_p})
		&= \overline{\Lie_X (T\phi)^* \alpha} \at[\big]{p} (\overline{w_p}) \\
		&= \overline{ (T\phi)^*\Lie_Y \alpha}\at[\big]{p} (\overline{w_p}) \\
		&= \overline{\Lie_Y \alpha}\at[\big]{\phi(p)} \Big(\overline{T\phi}(\overline{w_p})\Big) \\
		&= \Big(\nabla^{T^*\mathcal{N}}_{Y(\phi(p))} \overline{\alpha}\Big) \Big(\overline{T\phi}(\overline{w_p})\Big) \\
		&= (\overline{T_p\phi})^*\Big(\nabla^{T^*\mathcal{N}}_{T\phi(v_p)}\overline{\alpha} \Big)(\overline{w_p}) ,
	\end{align*}
	which proves the claim.
\end{proof}

On every constraint vector bundle $E$ the subbundle $E_\Null$ 
together with the $D$-connection $\nabla^E$ defines an 
equivalence relation on $E_\Wobs$ by 
$v_p \sim_E w_p$ if and only if
$p \sim_\mathcal{M} q$
and there exists a path $\gamma \colon I \to C$ in the leaf of $p$ 
such that $\cc{w_q} = \Parallel_{\gamma,p\to q}(\cc{v_p})$ is the parallel 
transport of $\cc{v_p}$ along $\gamma$.
Here $\cc{\argument}$ denotes the equivalence class in $E_\Wobs / 
E_\Null$.
Since $\nabla^E$ is holonomy-free this is independent of the chosen 
leafwise path, and thus indeed gives a well-defined equivalence 
relation. 
Infinitesimally, the equivalence relation $\sim_E$ corresponds to the 
distribution $D_E$ which is induced by $(D_\mathcal{M}, E_\Null, \nabla)$, see 
\autoref{rem:DistributionsOnVBAsDsVB}. 

\begin{proposition}[Reduction of constraint vector bundles]
	\label{prop:ReductionOfConVect}
	Let $E = (E_\Total,E_\Wobs,E_\Null,\nabla^E)$ be a constraint 
	vector bundle over a constraint manifold
	$\mathcal{M} = (M,C,D)$.
	\begin{propositionlist}
		\item \label{prop:ReductionOfConVect_1}
		There exists a unique vector bundle structure on 
		\begin{equation}
			\pr_{E_\red} \colon E_\Wobs / \mathord{\sim}_E \to \mathcal{M}_\red,
			\qquad \pr_E([v_p]) = \pi_\mathcal{M}(p),
		\end{equation}
		with $\pi_\mathcal{M} \colon C \to \mathcal{M}_\red$,
		such that the quotient map
		\begin{equation}
			\pi_E \colon E_\Wobs \to E_\Wobs / \mathord{\sim}_E,
			\qquad \pi_E(v_p) = [v_p]
		\end{equation}
		is a submersion and a vector bundle morphism over 
		$\pi_\mathcal{M}$.
		\item \label{prop:ReductionOfConVect_2}
		There exists an isomorphism
		\begin{equation}
			\Theta \colon (E_\Wobs / E_\Null) \to 
			\pi_\mathcal{M}^\#(E_\Wobs / \sim_E),
			\qquad \Theta(\cc{v_p}) = (p,[v_p])
		\end{equation}
		of vector bundles fulfilling
		\begin{equation}
			\label{eq:ReductionOfConVect_Parallel}
			\Theta^{-1}(q,[v_p]) = \Parallel_{\gamma,p\to q}(\cc{v_p})
		\end{equation}
		for $v_p \in E_\Wobs\at{p}$ and $p \sim_\mathcal{M} q$.
	\end{propositionlist}	
\end{proposition}

\begin{proof}
	We can split the quotient procedure into two steps.
	First we consider the quotient vector bundle
	$E_\Wobs / E_\Null \to C$ with quotient map
	$\pi_{E_\Null} \colon E_\Wobs \to E_\Wobs / E_\Null$
	being a submersion and vector bundle morphism.
	Now the $D$-connection $\nabla^E$ induces an equivalence 
	relation on $E_\Wobs / E_\Null$ by
	$\cc{v_p} \sim_{\nabla^E} \cc{w_p}$
	if and only if $p \sim_\mathcal{M} q$ and
	$\cc{w_p} = \Parallel_{\gamma,p\to q}(\cc{v_p})$.
	In the language of Lie groupoids it is easy to see that
	the parallel transport of $\nabla^E$ defines a linear action
	of the Lie groupoid
	$R(\pi_\mathcal{M}) = 
	C \deco{}{\pi_\mathcal{M}}{\times}{}{\pi_\mathcal{M}} C$
	on $(E_\Wobs / E_\Null)$.
	Then \cite[Lemma 4.1]{higgins.mackenzie:1990a}
	gives the existence of a unique vector bundle structure on 
	$\pr_\nabla \colon (E_\Wobs / E_\Null) / \sim_{\nabla^E} \to 
	\mathcal{M}_\red$ such that the quotient map
	$\pi_\nabla \colon (E_\Wobs / E_\Null) \to (E_\Wobs / E_\Null) / 
	\sim_{\nabla^E}$ is a submersion and a vector bundle morphism over
	$\pi_\mathcal{M}$.
	Combining these we obtain a unique vector bundle structure on
	$E_\Wobs /\mathbin{\sim_E} \simeq (E_\Wobs / E_\Null) / \sim_{\nabla^E}$
	such that $\pi_E = \pi_{\nabla^E} \circ \pi_{E_\Null}$
	is a submersion and vector bundle morphism over $\pi_\mathcal{M}$.
	The second part is again directly given by \cite[Lemma 
	4.1]{higgins.mackenzie:1990a}.
\end{proof}

\begin{definition}[Reduced vector bundle]
	\label{def:ReducedVectorBundle}
	Let $E = (E_\Total, E_\Wobs,E_\Null,\nabla^E)$ be a constraint vector bundle
	over a constraint manifold
	$\mathcal{M} = (M,C,D)$.
	Then the vector bundle
	$E_\red \coloneqq (E_\Wobs / E_\Null) / \sim_{\nabla^E}$ 
	over $\mathcal{M}_\red$ is called the \emph{reduced vector bundle} of $E$.
\end{definition}

\begin{example}
	Consider a trivial bundle $\mathcal{M} \times \Reals^k$ as in \autoref{ex:TrivialConVectorBundle}.
	Then
	\begin{equation}
		(C \times \Reals^{k_\Wobs}) / (C \times \Reals^{k_\Null}) \simeq C \times \Reals^{k_\Wobs - k_\Null}
	\end{equation}
	and since the $D$-connection is just given by the Lie derivative we get
	$(\mathcal{M} \times \Reals^k)_\red \simeq \mathcal{M}_\red \times \Reals^{k_\Wobs - k_\Null}$.
\end{example}

Morphisms of constraint vector bundles are designed to yield well-defined morphisms between the reduced vector bundles, allowing for a reduction functor, as expected.

\begin{proposition}[Reduction functor]
	\label{prop:RedFunctorVect}
	Mapping constraint vector bundles to their reduced bundles defines a functor
	$\red \colon \ConVect \to \Vect$.
\end{proposition}

\begin{proof}
	We need to show that morphisms of constraint vector bundles induce morphisms between the respective reduced bundles.
	For this let $\Phi \colon E \to F$ be a morphism of constraint vector bundles
	$E \to \mathcal{M}$ and $F \to \mathcal{N}$
	over a smooth map $\phi \colon \mathcal{M} \to \mathcal{N}$.
	Since $\Phi$ restricts to a vector bundle morphism
	$\Phi_\Wobs \colon E_\Wobs \to F_\Wobs$
	which maps the subbundle $E_\Null$ to $F_\Null$
	we obtain a well-defined vector bundle morphism
	$\Phi_\Wobs \colon E_\Wobs / E_\Null \to F_\Wobs / F_\Null$,
	which is compatible with the covariant derivates in the sense of
	\eqref{eq:ConstraintVBMorphism}.
	Now suppose that $v_p \sim_E w_q$.
	Then by \autoref{prop:CharacterizationsMorConnection}~\ref{prop:CharacterizationsMorConnection_Parallel}
	we know
	$\Phi(\cc{w_q})	 = \Parallel_{\hat\gamma,\phi(p)\to \phi(q)} \Phi(\cc{v_p})$,
	showing that $\Phi$ preserves the equivalence relation and thus drops to a map
	$\Phi_\red \colon E_\red \to F_\red$.
	It is smooth since locally there exist sections of the projection map 
	$\pi_\red \colon E_\Wobs \to E_\red$.
	Moreover, it is clearly fiberwise linear, hence defining a vector bundle morphism
	$\Phi_\red \colon E_\red \to F_\red$.
\end{proof}

\begin{proposition}
	\label{prop:ReductionOfTangentBundle}
	There exists a natural isomorphism making the following diagram commute:
	\begin{equation}
		\begin{tikzcd}
			\ConMfld
			\arrow[r,"T"]
			\arrow[d,"\red"{swap}]
			&\ConVect
			\arrow[d,"\red"]\\
			\Manifolds
			\arrow[r,"T"]
			& \Vect
		\end{tikzcd}
	\end{equation}
\end{proposition}

\begin{proof}
	For every constraint manifold $\mathcal{M}$
	we construct an isomorphism
	$\Psi_\mathcal{M} \colon \pi_\mathcal{M}^\#(T\mathcal{M})_\red \to \pi_\mathcal{M}^\#T(\mathcal{M}_\red)$.
	From \autoref{prop:ReductionOfConVect} we know that $\pi_\mathcal{M}^\#(T\mathcal{M})_\red \simeq TC/D$.
	Moreover, recall that we can pull back every germ $f_{[p]} \in \Cinfty_{[p]}(\mathcal{M}_\red)$
	to $\pi_\mathcal{M}^*f_{[p]} \in \Cinfty_p(C)$.
	Thus we can define
	\begin{equation*}
		\Psi_\mathcal{M} \colon TC/D
		\ni [v_p]
		\mapsto (p, v_p \circ \pi_\mathcal{M}^*)
		\in \pi_\mathcal{M}^\#T(\mathcal{M}_\red),
	\end{equation*}
	giving a fiberwise injective vector bundle morphism since $D$ is obviously the kernel.
	To show surjectivity let $(p,w_{[p]}) \in \pi_\mathcal{M}^\#T(\mathcal{M}_\red)$ be given.
	Since $\pi_\mathcal{M}$ is a surjective submersion, there exists a local section
	$\sigma \colon V \to C$ on an open neighbourhood $V \subseteq \mathcal{M}_\red$ around
	$[p]$.
	With this we can set $v_p(f_p) \coloneqq w_{[p]}((\sigma^*f)_{[p]})$ for any
	$f \in \Cinfty(\pi_\mathcal{M}^{-1}(V))$,
	and thus $\Psi_\mathcal{M}([v_p]) = (p,w_{[p]})$.
	This shows that $\Psi_\mathcal{M}$ is a fiberwise isomorphism and hence an isomorphism of vector bundles.
	Then $\Psi_\mathcal{M}$ induces the isomorphism $(T\mathcal{M})_\red \simeq T(\mathcal{M}_\red)$
	as required.
\end{proof}

As an important tool to work with constraint vector bundles we need the existence of local frames adapted to the additional structures.
For this observe that every constraint vector bundle $E = (E_\Total,E_\Wobs,E_\Null,\nabla)$
over a constraint manifold $\mathcal{M} = (M,C,D)$
can be restricted to an open subset $U \subset M$ giving 
a constraint vector bundle
$E\at{U} = (E_\Total\at{U}, E_\Wobs\at{U \cap C}, E_\Null\at{U \cap C}, \nabla\at{U\cap C})$
over $\mathcal{M}\at{U}$.

\begin{lemma}
	\label{lem:AdaptedLocalFrames}
	Let $E = (E_\Total, E_\Wobs, E_\Null, \nabla)$ be a constraint vector bundle of rank
	$k = (k_\Total, k_\Wobs, k_\Null)$ 
	over a constraint manifold
	$\mathcal{M} = (M,C,D)$.
	Let furthermore $E_\WobsNotNull \to C$ and $E_\TotalNotWobs \to C$ be subbundles of
	$\iota^\#E_\Total$ such that
	$E_\Wobs = E_\Null \oplus E_\WobsNotNull$
	and
	$\iota^\# E_\Total = E_\Wobs \oplus E_\TotalNotWobs$.
	Then for every $p \in C$ there exists a local frame $e_1, \dotsc, e_{k_\Total} \in \Secinfty(E_\Total\at{U})$
	on an open neighbourhood $U \subseteq M$ around $p$ such that
	\begin{lemmalist}
		\item $\iota^\#e_i \in \Secinfty(E_\Null\at{U \cap C})$ for all $i=1,\dotsc,k_\Null$,
		\item $\iota^\#e_i \in \Secinfty(E_\WobsNotNull\at{U\cap C})$ and
		$\nabla_X \iota^\#e_i = 0$ for all  $X \in \Secinfty(D)$ and $i = k_\Null + 1, \dotsc, k_\Wobs$,
		\item $\iota^\#e_i \in \Secinfty(E_\TotalNotWobs\at{U \cap C})$ for all $i = k_\Wobs +1, \dotsc, k_\Total$.
	\end{lemmalist}
\end{lemma}

\begin{proof}
	Take a local frame $g_1, \dotsc, g_{k_\Wobs - k_\Null}$ of $E_\red$ on an open neighbourhood
	$\check{V} \subseteq \mathcal{M}_\red$ of $\pi_\mathcal{M}(p)$.
	Using \autoref{prop:ReductionOfConVect} \ref{prop:ReductionOfConVect_2}
	we obtain a local frame $g_1, \dotsc, g_{k_\Wobs - k_\Null}$ of $E_\WobsNotNull \simeq E_\Wobs / E_\Null$
	on the open neighbourhood $ \pi_\mathcal{M}^{-1}(\check{V})$
	with $\nabla_X g_i = 0$ for all $X \in \Secinfty(D)$
	and $i = 1, \dotsc, k_\Wobs - k_\Null$.
	Choose additionally local frames $f_1, \dotsc, f_{k_\Null}$ of $E_\Null$ and
	$h_1, \dotsc,  h_{k_\Total - k_\Wobs}$ of $E_\TotalNotWobs$ on a possibly smaller
	open neighbourhood $V$.
	Using a tubular neighbourhood $\pr_U \colon U \to C \cap V$ of $C \cap V$ inside $V$ we can pull back those local
	frames to a local frame 
	\begin{equation*}
		e_i \coloneqq \begin{cases}
			\pr_U^\#f_i &\text{ if } i = 1, \dotsc, k_\Null \\
			\pr_U^\#g_{i-n_\Null} &\text{ if } i = k_\Null + 1, \dotsc, k_\Wobs \\
			\pr_U^\#h_{i-n_\Wobs} &\text{ if } i = k_\Wobs  +1, \dotsc, k_\Total
		\end{cases}
	\end{equation*}
	of $E_\Total$ fulfilling the required properties.	
\end{proof}

In order to introduce the usual constructions of vector bundles also in the constraint setting we will need
the following lemma concerning the constraint homomorphism bundles:

\begin{lemma}
	\label{lem:WobsNullQuotientHom}
	Let $E,F$ be constraint vector bundles over a constraint manifold
	$\mathcal{M} = (M,C,D)$.
	Define vector bundles
	\begin{align}
		\ConHom(E,F)_\Wobs 
		& \coloneqq \left\{ \Phi_p \in \Hom(\iota^\#E_\Total, \iota^\#F_\Total)
		\bigm| \Phi_p(E_\Wobs\at{p}) \subseteq F_\Wobs\at{p}
		\text{ and } \Phi_p(E_\Null\at{p}) \subseteq F_\Null\at{p} \right\} \\
		\shortintertext{and}
		\ConHom(E,F)_\Null 
		& \coloneqq \left\{ \Phi_p \in \Hom(\iota^\#E_\Total, \iota^\#F_\Total)
		\bigm| \Phi_p(E_\Wobs\at{p}) \subseteq F_\Null\at{p} \right\}
	\end{align}
	over $C$.
	Then
	\begin{align}
		\Theta \colon \ConHom(E,F)_\Wobs / \ConHom(E,F)_\Null
		&\to \Hom(E_\Wobs/ E_\Null, F_\Wobs/F_\Null)
		\shortintertext{defined by}
		\Theta(\cc{\Phi_p})(\cc{v_p}) &\coloneqq \cc{\Phi_p(v_p)}
	\end{align}
	is an isomorphism of vector bundles.
	Here $\cc{\argument}$ denotes the projection to the quotient.
\end{lemma}

\begin{proof}
	It is clear that $\Theta$ is a well-defined map.
	Moreover, it is a vector bundle morphism 
	since it is essentially given by evaluation.
	The fiberwise injectivity is again clear by definition, while for the 
	fiberwise surjectivity we need to choose complements
	$E_\Null^\perp$ and $E_\Wobs^\perp$
	of
	$E_\Null$ inside $E_\Wobs$ and
	$E_\Wobs$ inside $\iota^\#E_\Total$.
	Thus $\iota^\#E_\Total = E_\Null \oplus E_\Null^\perp \oplus E_\Wobs^\perp$
	and $E_\Wobs / E_\Null \simeq E_\Null^\perp$.
	Then for every $\Psi_p \in \Hom(E_\Wobs/E_\Null, F_\Wobs/F_\Null)$
	set $\Phi(v_p) = \Psi(v_p)$ for all $v_p \in E_\Null^\perp$
	and $\Phi(v_p) = 0$ for all $v_p \in E_\Null$ or $v_p \in E_\Wobs^\perp$.
	With this we have $\Theta(\Phi_p) = \Psi_p$.
	Thus we have an isomorphism of vector bundles as claimed.
\end{proof}

Let us now introduce some important constructions on constraint vector bundles.

\begin{proposition}
	\label{prop:ConstructionsCVect}
	Let $E = (E_\Total,E_\Wobs,E_\Null,\nabla^E)$ and
	$F = (F_\Total,F_\Wobs,F_\Null,\nabla^F)$ be 
	constraint vector bundles over the constraint manifold 
	$\mathcal{M} = (M_\Total, M_\Wobs,D_{\mathcal{M}})$.
	Let $\mathcal{N} = (N_\Total, N_\Wobs,D_{\mathcal{N}})$ be another 
	constraint manifold and $\phi \colon \mathcal{N} \to \mathcal{M}$ 
	be a constraint map.
	\begin{propositionlist}
		\item \label{prop:ConstructionsCVect_DirectSum}
		Defining $E \oplus F$ by
		\begin{equation}
			\begin{split}
				(E \oplus F)_\Total
				&\coloneqq E_\Total \oplus F_\Total,\\
				(E \oplus F)_\Wobs
				&\coloneqq E_\Wobs \oplus F_\Wobs,\\
				(E \oplus F)_\Null
				&\coloneqq E_\Null \oplus F_\Null,\\
				\nabla^{E \oplus F}
				&\coloneqq \nabla^{E} \oplus \nabla^F
			\end{split}
		\end{equation}
		yields a constraint vector bundle over $\mathcal{M}$, called the
		\emph{direct sum},
		with
		\begin{equation}
			\label{eq:ConstructionsCVect_SumRank}
			\rank (E \oplus F) = \rank(E) + \rank(F).
		\end{equation}
		\item \label{prop:ConstructionsCVect_Tensor}
		Defining $E \tensor F$ by
		\begin{equation}
			\begin{split}
				(E \tensor F)_\Total
				&\coloneqq E_\Total \tensor F_\Total,\\
				(E \tensor F)_\Wobs
				&\coloneqq E_\Wobs \tensor F_\Wobs,\\
				(E \tensor F)_\Null
				&\coloneqq E_\Null \tensor F_\Wobs + E_\Wobs \tensor F_\Null,\\
				\nabla^{E \tensor F}
				&\coloneqq \nabla^{E} \tensor \id + \id \tensor \nabla^F
			\end{split}
		\end{equation}
		yields a constraint vector bundle over $\mathcal{M}$, called the \emph{tensor product},
		with
		\begin{equation}
			\label{eq:ConstructionsCVect_TensorRank}
			\rank(E \tensor F) = \rank(E) \tensor \rank(F).
		\end{equation}
		\item \label{prop:ConstructionsCVect_StrTensor}
		Defining $E \strtensor F$ by
		\begin{equation}
			\begin{split}
				(E \strtensor F)_\Total
				&\coloneqq E_\Total \tensor F_\Total,\\
				(E \strtensor F)_\Wobs
				&\coloneqq E_\Wobs \tensor F_\Wobs
				+ E_\Null \tensor \iota^\#F_\Total 
				+ \iota^\#E_\Total \tensor F_\Null,\\
				(E \strtensor F)_\Null
				&\coloneqq E_\Null \tensor \iota^\#F_\Total 
				+ \iota^\#E_\Total \tensor F_\Null,\\
				\nabla^{E \strtensor F}
				&\coloneqq  \nabla^{E} \tensor \id + \id \tensor \nabla^F
			\end{split}
		\end{equation}
		yields a constraint vector bundle over $\mathcal{M}$, called 
		the \emph{strong tensor product},
		with	
		\begin{equation}
			\label{eq:ConstructionsCVect_StrTensorRank}
			\rank(E \strtensor F) = \rank(E) \strtensor \rank(F).
		\end{equation}
		\item \label{prop:ConstructionsCVect_Dual}
		Defining $E^*$ by
		\begin{equation} \label{eq:ConDualBundle}
			\begin{split}
				(E^*)_\Total
				&= (E_\Total)^*,\\
				(E^*)_\Wobs
				&= \Ann_{\iota^\#E_\Total}(E_\Null),\\
				(E^*)_\Null
				&= \Ann_{\iota^\#E_\Total}(E_\Wobs),
			\end{split}
		\end{equation}
		with $\Ann_{\iota^\#E_\Total}(E_\Null)$ and $\Ann_{\iota^\#E_\Total}(E_\Wobs)$ the annihilator subbundles
		of $E_\Null$ and $E_\Wobs$ with respect to $\iota^\#E_\Total$
		and $\nabla^{E^*}$ the dual covariant derivative,
		yields a constraint vector bundle over $\mathcal{M}$, called the
		\emph{dual vector bundle},
		with
		\begin{equation}
			\label{eq:ConstructionsCVect_DualRank}
			\rank (E^*)
			= \rank(E)^*.
		\end{equation}		
		\item \label{prop:ConstructionsCVect_Hom}
		Defining $\ConHom(E,F)$ by
		\begin{equation}
			\begin{split}
				\ConHom(E,F)_\Total 
				& \coloneqq \Hom(E_\Total, F_\Total),\\
				\ConHom(E,F)_\Wobs 
				& \coloneqq \left\{ \Phi_p \in \Hom(\iota^\#E_\Total, \iota^\#F_\Total)
				\bigm| \Phi_p(E_\Wobs\at{p}) \subseteq F_\Wobs\at{p}\right.\\
				&\qquad\qquad\qquad\quad\qquad\qquad\qquad\left.\text{ and } \Phi_p(E_\Null\at{p}) \subseteq F_\Null\at{p} \right\},\\
				\ConHom(E,F)_\Null 
				& \coloneqq \left\{ \Phi_p \in \Hom(\iota^\#E_\Total, \iota^\#F_\Total)
				\bigm| \Phi_p(E_\Wobs\at{p}) \subseteq F_\Null\at{p} \right\}, \\
				\nabla^{\Hom}_X A 
				&\coloneqq \nabla_X^F \circ A - A \circ \nabla_X^E,
			\end{split}
		\end{equation}
		where $A \in \Secinfty(\ConHom(E,F)_\Wobs / \ConHom(E,F)_\Null)$
		is identified with
		$A \colon \Secinfty(E_\Wobs/E_\Null) \to \Secinfty(F_\Wobs/F_\Null)$
		using \autoref{lem:WobsNullQuotientHom}
		and $X \in \Secinfty(D)$,
		yields a constraint vector bundle, called the \emph{homomorphism bundle},
		with
		\begin{equation}
			\label{eq:ConstructionsCVect_HomRank}
			\rank(\ConHom(E,F)) = \rank(E^*) \strtensor \rank(F).
		\end{equation}
		\item \label{prop:ConstructionsCVect_PullBack}
		Defining $\phi^\# F$ by
		\begin{equation}
			\begin{split}
				\left(\phi^\# F \right)_\Total 
				& \coloneqq \phi^\#_\Total F_\Total,\\
				\left(\phi^\# F\right)_\Wobs 
				& \coloneqq \phi^\#_\Wobs F_\Wobs,\\
				\left(\phi^\# F\right)_\Null 
				& \coloneqq \phi^\#_\Wobs F_\Null,\\
				\nabla^{\phi^\# F}
				& \coloneqq \phi^\#_\Wobs \nabla^F,
			\end{split}
		\end{equation}
		where $\phi^\#_\Wobs \nabla^F$ denotes the pull-back linear	connection 
		of $\nabla^F$, see  
		\autoref{prop:ExistenceOfPullBackLinearConnection}, yields a 
		constraint vector bundle over $\mathcal{N}$, called the
		\emph{constraint pull-back vector bundle of $F$}, with
		\begin{equation}
			\rank(\phi^\#F) = \rank(F).
		\end{equation}
	\end{propositionlist}
\end{proposition}

\begin{proof}
	\ref{prop:ConstructionsCVect_DirectSum}: Note that the direct sum of subbundles is a subbundle of the direct sum 
	and we have 
	$(E \oplus F)_\Wobs / (E \oplus F)_\Null \simeq (E_\Wobs / E_\Null) \oplus (F_\Wobs / F_\Null)$.
	Moreover, the parallel transport of $\nabla^{E \oplus F}$ is given by the direct sum of the parallel transports of $\nabla^E$ and $\nabla^F$, and thus it is holonomy-free.
	
	\ref{prop:ConstructionsCVect_Tensor}: We need to show that $(E \tensor F)_\Null$ actually forms a subbundle
	of $E_\Wobs \tensor F_\Wobs$.
	Let $p \in C$ be given, then the dimension of
	$(E_\Null \tensor F_\Wobs)\at{p} \cap (E_\Wobs \tensor F_\Null)\at{p}
	= (E_\Null \tensor F_\Null )\at{p}$ is independent of $p$, and thus
	$(E \tensor F)_\Null$ has constant rank and therefore defines a subbundle
	of $E_\Wobs \tensor F_\Wobs$.
	The parallel transport of $\nabla^{E \tensor F}$ on 
	$(E \tensor F)_\Wobs / (E \tensor F)_\Null
	\simeq (E_\Wobs / E_\Null) \tensor (F_\Wobs / F_\Null)$
	is given by the tensor product of the parallel transports, and hence is holonomy-free.
	
	\ref{prop:ConstructionsCVect_StrTensor}: 
	With an analogous argument we see that $(E \strtensor F)_\Wobs$ and $(E \strtensor F)_\Null$
	are well-defined subbundles with
	$(E \strtensor F)_\Wobs / (E \strtensor F)_\Null
	\simeq (E_\Wobs / E_\Null) \tensor (F_\Wobs / F_\Null)$
	and holonomy-free covariant derivative.
	
	\ref{prop:ConstructionsCVect_Dual}:
	For the dual bundle we have by definition subbundles
	\begin{equation*}
		\Ann_{\iota^\#E_\Total}(E_\Wobs) \subseteq \Ann_{\iota^\#E_\Total}(E_\Null)\subseteq \iota^\#(E_\Total)^*
	\end{equation*}
	holds.
	Moreover, $\Ann_{\iota^\#E_\Total}(E_\Null) /\Ann_{\iota^\#E_\Total}(E_\Wobs) \simeq (E_\Wobs / E_\Null)^*$
	holds and since $\nabla^{E}$ is holonomy-free so is the dual covariant derivative $\nabla^{E^*}$.
	
	\ref{prop:ConstructionsCVect_Hom}: For the homomorphism bundle note that $\iota^\#\Hom(E_\Total, F_\Total) \simeq \Hom(\iota^\#E_\Total, \iota^\#F_\Total)$.
	By using adapted local frames as in \autoref{lem:AdaptedLocalFrames} it is then easy to see that
	$\ConHom(E,F)_\Wobs$ and $\ConHom(E,F)_\Null$ form subbundles of
	$\iota^\#\ConHom(E,F)_\Total$.
	Moreover, since $\nabla^{\Hom}$ is the covariant derivative obtained from the isomorphism
	$\Hom(E_\Wobs / E_\Null, F_\Wobs / F_\Null) \simeq (E_\Wobs / E_\Null)^* \tensor (F_\Wobs / F_\Null)$
	and duals as well as tensor products of holonomy-free covariant derivatives are again holonomy-free,
	so is $\nabla^{\Hom}$.
	
	\ref{prop:ConstructionsCVect_PullBack}: By choosing an adapted local frame for 
	$F$ on some open neighbourhood $U \subseteq M_\Total$ as 
	in \autoref{lem:AdaptedLocalFrames} immediately clarifies why 
	$(\phi^\# F)_\Wobs$ and $(\phi^\# F)_\Null$ are subbundles of 
	$(\phi^\# F)_\Total$. 
	Since we have
	$(\phi^\# F)_\Wobs / (\phi^\# F)_\Null \simeq \phi^\# (F_\Wobs / F_\Null)$, it 
	follows by \autoref{prop:ExistenceOfPullBackLinearConnection} that 
	$\nabla^{\phi^\# F}$ is holonomy-free. 
\end{proof}

Note that the dual bundle $E^*$ 
is indeed given by the homomorphism bundle $\ConHom(E,\mathcal{M} \times \Reals)$, and we clearly have $(E^*)^* \simeq E$.
Moreover, all the above constructions can easily seen to be functorial.
In particular, dualizing is a contravariant endofunctor of $\ConVect(\mathcal{M})$.

\begin{example}\
Let $\mathcal{M} = (M,C,D)$ be a constraint manifold.
\begin{examplelist}
	\item The constraint cotangent bundle of $\mathcal{M}$ is given by
	\begin{equation}
		\begin{split}
			(T^*\mathcal{M})_\Total 
			&= T^*M\\
			(T^*\mathcal{M})_\Wobs
			&= \Ann_{\iota^\#T^*M}(D)\\
			(T^*\mathcal{M})_\Null 
			&= \Ann_{\iota^\#T^*M}(TC).
		\end{split}
	\end{equation}
	Note that we can canonically identify
	$\Ann_{\iota^\#T^*M}(D) / \Ann_{\iota^\#T^*M}(TC) \simeq \Ann_{TC}(D)$.
	Under this identification $\cc{\iota^\#\alpha}$ becomes just
	the pullback (or restriction) $\iota^*\alpha \in \Ann_{TC}(D)$ of the form $\alpha$ to $C$.
	Then the dual Bott connection is given by
	\begin{equation}
		\nabla^{\Bott}_X \iota^*\alpha = \Lie_X \iota^*\alpha.
	\end{equation}
	\item For any given constraint vector bundle $E$ over $\mathcal{M}$ we can construct its
	antisymmetric tensor powers with respect to both $\tensor$ and $\strtensor$.
	Denoting by $\wedge$ the antisymmetric tensor product of classical vector bundles we obtain
	\begin{equation}
		\begin{split}
		(\Anti_{\tensor}^k E)_\Total &= \Anti^k E_\Total,\\
		(\Anti_{\tensor}^k E)_\Wobs &= \Anti^k E_\Wobs,\\
		(\Anti_{\tensor}^k E)_\Null &= \Anti^{k-1} E_\Wobs \wedge E_\Null,
	\end{split}
	\qquad\text{and}\qquad
	\begin{split}
		(\Anti_{\strtensor}^k E)_\Total &= \Anti^k E_\Total,\\
		(\Anti_{\strtensor}^k E)_\Wobs &= \Anti^k E_\Wobs + \Anti^{k-1} E_\Total \wedge E_\Null,\\
		(\Anti_{\strtensor}^k E)_\Null &= \Anti^{k-1} E_\Total \wedge E_\Null.
	\end{split}
	\end{equation}
	Here we suppressed the pullback $\iota^\#$ for the $\TOTAL$-bundles, since
	from the context it is clear that we only can take tensor products of vector bundles
	over $C$.
\end{examplelist}
\end{example}

At this point we could prove the expected compatibilities between the various constructions of constraint vector bundles, but we will postpone this in order to be able to reuse \autoref{prop:DualTensorHomIsos}.

\subsection{Sections of Constraint Vector Bundles}
\label{sec:ConSections}

In order to motivate the definition of sections of constraint vector bundles consider the total space of a constraint vector bundle $E = (E_\Total,E_\Wobs,E_\Null,\nabla)$ over a constraint manifold
$\mathcal{M} = (M,C,D)$ in the following way:
The vector bundle $E_\Total$ is clearly a smooth manifold, and since $C \subseteq M$ is a closed submanifold so is
$E_\Wobs \subseteq E_\Total$.
Additionally, by identifying $E_\Wobs / E_\Null$ with $E_\WobsNotNull$ such that
$E_\Wobs \simeq E_\Null \oplus E_\WobsNotNull$, there is a distribution $D_E$ on $E_\Wobs$ which is given by
$D_E \coloneqq \Ver(E_\Null) \oplus \Hor^{\Sigma}(D_{\mathcal{M}})$,
see \autoref{rem:DistributionsOnVBAsDsVB}.
Thus we can understand the total space of a constraint vector bundle as a constraint manifold.
The vector bundle projection $\pr \colon E \to \mathcal{M}$ turns out to be a smooth map of constraint manifolds.
Thus a constraint section of $E$ should be a constraint map $s \colon \mathcal{M} \to E$ such that 
$\pr \circ s = \id_\mathcal{M}$.
This means in particular that $s$ restricted to $C$ yields a section $\iota^\# s$ of $E_\Wobs$.
Moreover, $\iota^\# s$ should map equivalent points in $C$ to equivalent vectors in $E_\Wobs$.
In other words, $\iota^\# s$ should either map to $E_\Null$ or be covariantly constant along the leaves of $D$.
These considerations motivate the following definition of the constraint module of sections.

\begin{proposition}[Functor of constraint sections]
	\label{prop:FunctorConSections}
	Let $\mathcal{M} = (M,C,D)$ be a constraint manifold.
	Mapping a constraint vector bundle 
	$E = (E_\Total,E_\Wobs,E_\Null,\nabla)$ to
	\begin{equation}
		\begin{split}
			\ConSecinfty(E)_\Total
			&\coloneqq \Secinfty(E_\Total) \\
			\ConSecinfty(E)_\Wobs
			&\coloneqq \left\{ s \in \Secinfty(E_\Total) \bigm|  \iota^\#s \in \Secinfty(E_\Wobs),
			\nabla_X\cc{\iota^\#s} = 0 \text{ for all } X \in \Secinfty(D) \right\}, \\
			\ConSecinfty(E)_\Null
			&\coloneqq \left\{ s \in \Secinfty(E_\Total) \bigm| \iota^\#s \in 
			\Secinfty(E_\Null) \right\},
		\end{split}
	\end{equation}
	and a constraint vector bundle morphism $\Phi \colon E \to F$ over the identity to
	\begin{equation}
		\Phi \colon \ConSecinfty(E) \to \ConSecinfty(F),
		\qquad
		\Phi(s)(p) \coloneqq \Phi\big(s(p)\big)
	\end{equation}
	defines a functor 
	$\ConSecinfty \colon \ConVect(\mathcal{M}) \to \injstrConMod(\ConCinfty(\mathcal{M}))$.
\end{proposition}

\begin{proof}
	First note that $\ConSecinfty(E)_\Total$ is clearly a $\ConCinfty(\mathcal{M})_\Total$-module.
	In general we have $\iota^\#(f\cdot s) = \iota^*f \cdot  \iota^\#s$ for
	$s \in \Secinfty(E_\Total)$ and $f \in \Cinfty(M)$.
	Thus for $f \in \ConCinfty(\mathcal{M})_\Wobs$ and
	$s \in \ConSecinfty(E)_\Wobs$ we have
	$\iota^\#(f\cdot s) \in \Secinfty(E_\Wobs)$
	and
	\begin{align*}
		\nabla^E_X(\cc{\iota^\#(f \cdot s)})
		= \nabla^E_X(\iota^*f \cdot \cc{\iota^\#s})
		= \Lie_X\iota^*f \cdot \cc{\iota^\#s} + \iota^*f \cdot \nabla^E_X \cc{\iota^\#s}
		= \iota^*f \cdot \nabla^E_X \cc{\iota^\#s}
		= 0
	\end{align*}
	for all $X \in \Secinfty(D)$, where we used $\Lie_X \iota^*f = 0$.
	Now let $s \in \Secinfty(E_\Total)$ and $f \in \ConCinfty(\mathcal{M})_\Null$ be given, then
	$\iota^\#(f \cdot s) = \iota^*f \cdot \iota^\#s = 0 \in \Secinfty(E_\Null)$.
	If $s \in \Secinfty(E)_\Null$ and $f \in \Cinfty(M)$, we get again
	$\iota^\# (f \cdot s) \in \Secinfty(E_\Null)$.
	Hence we see that $\ConSecinfty(E)$ is indeed a strong $\ConCinfty(\mathcal{M})$-module.
	Let now $\Phi \colon E \to F$ be a constraint morphism of constraint vector bundles over $\id_{\mathcal{M}}$.
	Then $\Phi$ can be restricted to a morphism between the $\WOBS$- or $\NULL$-components, meaning that
	$\Phi$ commutes with $\iota^\#$.
	Moreover, since $\Phi$ is by definition compatible with the partial connections, it maps flat sections
	to flat sections.
	Hence $\Phi$ induces a constraint module morphism between the modules of sections.
\end{proof}

It should be stressed that $\ConSecinfty(E)_\Wobs$ and $\ConSecinfty(E)_\Null$ consist
of globally defined sections, with additional properties on $C$.
In particular, $\ConSecinfty(E)_\Null$ consists of those sections of $E_\Total$ which on $C$
are sections of the subbundle $E_\Null$,
while $\ConSecinfty(E)_\Wobs$ consists of sections of $E_\Total$ such that on $C$
it is a section of the subbundle $E_\Wobs$ whose $E_\Null$ component can be arbitrary, 
but everything complementary to $E_\Null$ needs to be covariantly constant along the leaves.

\begin{example}[(Co-)Vector fields]
	Let $\mathcal{M} = (M,C,D)$ be a constraint manifold.
	\begin{examplelist}
		\item For the constraint tangent bundle $T\mathcal{M}$ we get
		\begin{equation}
		\label{eq:constraintVectorFields}
			\begin{split}
				\ConSecinfty(T\mathcal{M})_\Total 
				&= \Secinfty(TM),\\
				\ConSecinfty(T\mathcal{M})_\Wobs
				&= \left\{ X \in \Secinfty(TM) \bigm| X\at{C} \in \Secinfty(TC) \text{ and}\right.\\
				&\hspace{8.5em}\left.[X,Y] \in \Secinfty(D) \text{ for all } Y \in \Secinfty(D) \right\},\\
				\ConSecinfty(T\mathcal{M})_\Null
				&= \left\{ X \in \Secinfty(TM) \bigm| X\at{C} \in \Secinfty(D) \right\},
			\end{split}
		\end{equation}
		by the definition of the Bott connection, see \eqref{eq:BottConnection}.
		\item For the constraint cotangent bundle $T^*\mathcal{M}$ we get
		\begin{equation}
			\begin{split}
				\ConSecinfty(T^*\mathcal{M})_\Total 
				&= \Secinfty(T^*M),\\
				\ConSecinfty(T^*\mathcal{M})_\Wobs
				&= \left\{ \alpha \in \Secinfty(T^*M) \bigm| \ins_X \iota^*\alpha = 0,\;
				\Lie_X \iota^*\alpha = 0  \text{ for all } X \in \Secinfty(D) \right\},\\
				\ConSecinfty(T^*\mathcal{M})_\Null
				&= \left\{ \alpha \in \Secinfty(T^*M) \bigm| \iota^*\alpha = 0 \right\},
			\end{split}
		\end{equation}
		by the definition of the dual vector bundle in \eqref{eq:ConDualBundle}.
		In other words $\ConSecinfty(T^*\mathcal{M})_\Wobs$ are exactly those one-forms on $M$ which are basic
		when restricted to $C$, and $\ConSecinfty(T^*\mathcal{M})_\Null$ are those which vanish on $C$.
		Here we have to carefully distinguish between the pullback $\iota^\#\alpha$ as a section of the pullback bundle $\iota^\#T^*M$
		and the pullback (or restriction) $\iota^*\alpha \in \Secinfty(T^*C)$ of the form $\alpha$ along $\iota$.
	\end{examplelist}
\end{example}

\begin{example}[Sections of sub-bundles]
	\label{ex:SectionsSubbundles}
	Let $F \subseteq E$ be a constraint sub-bundle
	of a constraint vector bundle $E$ over a constraint manifold $\mathcal{M}$.
	Then the inclusion $i \colon F \hookrightarrow E$
	induces a morphism
	$i \colon \ConSecinfty(F) \to \ConSecinfty(E)$
	of constraint $\ConCinfty(\mathcal{M})$-modules.
	Note that this is in fact a regular monomorphism,
	see \autoref{prop:MonoEpisConModk}, and thus
	$\ConSecinfty(F)$ is a constraint submodule
	of $\ConSecinfty(E)$.
\end{example}

\begin{example}\
	\label{ex:ConSections}
	Let $\mathcal{M} = (M,C,D)$ be a constraint manifold
	of dimension $d = (d_\Total, d_\Wobs, d_\Null)$,
	$p \in C$ and $(U,x)$ an adapted  chart around $p$ as in
	\autoref{lem:LocalStructureConManfifold}.
	Then
	\begin{equation}
		\begin{split}
			\frac{\del}{\del x^i} &\in 
			\ConSecinfty(T\mathcal{M}\at{U})_\Total
			\quad\text{ if }\quad i \in \{1, \dotsc, d_\Total\},\\
			\frac{\del}{\del x^i} &\in 
			\ConSecinfty(T\mathcal{M}\at{U})_\Wobs
			\quad\text{ if }\quad i \in \{1, \dotsc, d_\Wobs\}, \\
			\frac{\del}{\del x^i} &\in 
			\ConSecinfty(T\mathcal{M}\at{U})_\Null 
			\quad\text{ if }\quad i \in \{1, \dotsc, d_\Null \}.
		\end{split}
	\end{equation}
\end{example}

This example motivates the definition of a constraint local frame.

\begin{definition}[Constraint local frame]
	\label{def:ConLocalFrame}
	Let $E = (E_\Total, E_\Wobs, E_\Null,\nabla)$ be a constraint vector bundle of
	rank $k$ over
	a constraint manifold $\mathcal{M} = (M,C,D)$.
	A \emph{local frame} of $E$ on an open subset $U \subseteq M$, is
	a local frame $e_1, \dotsc, e_{k_\Total}$ of $E_\Total$ on $U$,
	such that
	\begin{definitionlist}
		\item $e_1, \dotsc, e_{k_\Wobs} \in \ConSecinfty(E\at{U})_\Wobs$
		and $\iota^\#e_1,\dotsc, \iota^\#e_{k_\Wobs}$ is a local 
		frame for $E_\Wobs$ on $U \cap C$, and
		\item $e_1, \dotsc, e_{k_\Null} \in \ConSecinfty(E\at{U})_\Null$
		and $\iota^\#e_1,\dotsc, \iota^\#e_{k_\Null}$ is a local 
		frame for $E_\Null$ on $U \cap C$.
	\end{definitionlist}
\end{definition}

The existence of local frames for constraint vector bundles is 
guaranteed by
\autoref{lem:AdaptedLocalFrames}.
To show that every $v_p \in E_\Total\at{p}$ is the value of some section $s \in \Secinfty(E_\Total)$
one can simply extend a local frame for $E_\Total$ to all of $M$ by means of a cut-off function.
Now for $v_p \in E_\Wobs\at{p}$ this is not so easy any more, since a cut-off function would need to
be an element of $\ConCinfty(\mathcal{M})_\Wobs$ itself to end up with a section in $\ConSecinfty(E)_\Wobs$.

\begin{corollary}
	\label{lem:ExtendingToConSections}
	Let $E = (E_\Total,E_\Wobs,E_\Null,\nabla)$ be a constraint
	vector bundle over a constraint manifold $\mathcal{M} = (M,C,D)$.
	\begin{lemmalist}
		\item \label{lem:ExtendingToConSections_1}
		For each $p \in C$ and $v_p \in E_\Null\at{p}$ there exists an
		$s \in \ConSecinfty(E)_\Null$ such that
		$s(p) = v_p$.
		\item  \label{lem:ExtendingToConSections_2}
		For each $p \in C$ and $v_p \in E_\Wobs\at{p}$ there exists an
		$s \in \ConSecinfty(E)_\Wobs$ such that
		$s(p) = v_p$.
	\end{lemmalist}
\end{corollary}

\begin{proof}
	For the first part choose a local frame $e_1, \dotsc, e_{n_0}$
	of $E_\Null$ around $p$ with $n_0 = \rank(E_\Null)$.
	Then using $v_p = \sum_{k=1}^{n_0} v_p^k e_k(p)$
	we can define a local section $\sum_{k=1}^{n_0} v_p^k e_k$
	which we extend to a section $\tilde{s} \in \Secinfty(E_\Null) 
	\subseteq \Secinfty(\iota^\#E_\Total)$
	by means of a bump function.
	In order to extend $\tilde{s}$ to a section of $E_\Total$ choose 
	a tubular neighbourhood $V \subseteq M$ of $C$ with bundle 
	projection $\pi_V \colon V \to C$.
	Then pulling back $\tilde{s}$ to $V$ via $\pi_V$ and afterwards 
	extending to all of $M$ using a suitable bump function gives
	a globally defined section $s \in \Secinfty(E_\Total)$
	with $\iota^\#s = \tilde{s} \in \Secinfty(E_\Null)$
	and $s(p) = \tilde{s}(p) = v_p$.
	Note that the existence of such a bump function requires the closedness of $C$.
	For \ref{lem:ExtendingToConSections_2} choose a complementary 
	vector bundle $E_\WobsNotNull \to C$ to $E_\Null$ inside of 
	$E_\Wobs$, i.e. $E_\Wobs = E_\Null \oplus E_\WobsNotNull$ and hence
	$E_\WobsNotNull \simeq E_\Wobs / E_\Null$.
	Then $v_p = v_p^0 + v_p^\perp$ with $v_p^0 \in E_\Null\at{p}$
	and $v_p^\perp \in E_\WobsNotNull\at{p}$.
	By \ref{lem:ExtendingToConSections_1} we find a section
	$s_0 \in \Secinfty(E)_\Null$ such that $s_0(p) = v_p^0$.
	Now choose $\check{s} \in \Secinfty(E_\red)$ such that
	$\check{s}(\pi_\mathcal{M}(p)) = [v_p^\perp]$.
	Then by \autoref{prop:ReductionOfConVect} \ref{prop:ReductionOfConVect_2}
	we can identify $\pi_\mathcal{M}^\#\check{s}$ with a section
	$s^\perp \in \Secinfty(E_\Wobs / E_\Null)$ such that
	$\nabla_X s^\perp = 0$ for all $X \in \Secinfty(D)$.
	Then using a tubular neighbourhood as before to extend $s_0 + s^\perp$ to all of
	$M$ we obtain the desired section.
\end{proof}

In classical differential geometry the famous Serre-Swan Theorem states that the category
$\Vect(M)$ of vector bundles over a fixed manifold $M$ is equivalent to the category
$\Proj(\Cinfty(M))$ of finitely generated projective $\Cinfty(M)$-modules.
By \autoref{prop:FunctorConSections} we know that sections of constraint vector bundles form
strong constraint modules over the commutative strong constraint algebra $\ConCinfty(\mathcal{M})$
of functions on the constraint manifold $\mathcal{M}$.
In fact, every module of sections is finitely generated projective in the sense of
\autoref{def:ConProjectiveModules}, as the following lemma shows:

\begin{lemma}
	\label{lem:ConstructionConDualBasis}
	Let $E = (E_\Total,E_\Wobs,E_\Null,\nabla)$
	be a constraint vector bundle over a given constraint manifold
	$\mathcal{M} = (M,C,D)$.
	Suppose the following data is given:
	\begin{lemmalist}
		\item A dual basis $\{f_j,f^j\}_{j\in J_\NullNotVan}$ of $\Secinfty(E_\Null)$
		indexed by some set $J_\NullNotVan$.
		\item A complementary vector bundle $E_\WobsNotNull \subseteq E_\Wobs$ of $E_\Null$
		together with a dual basis $\{g_j,g^j\}_{j \in J_{\WobsNotNull}}$ of 
		$\Secinfty(E_\WobsNotNull)$, indexed by some set $J_{\WobsNotNull}$,
		such that
		$\nabla^E \cc{g_j} = 0$ and $\nabla^{E^*} \cc{g^j} = 0$.
		\item A complementary vector bundle $E_\TotalNotWobs \subseteq \iota^\#E_\Total$ of $E_\Wobs$
		together with a dual basis $\{h_j,h^j\}_{j \in J_{\TotalNotWobs}}$ of 
		$\Secinfty(E_\TotalNotWobs)$, indexed by some set $J_{\TotalNotWobs}$.
	\end{lemmalist}
	Then there exists a finite set $J_\Van$ and a constraint dual basis $(\{e_j\}_{j \in J},\{e^j\}_{j \in J^*})$, 
	indexed by the constraint index set
	$J = (J_\Total,J_\Wobs,J_\Null)$
	defined by
	\begin{equation}
	\begin{split}
		J_\Total &\coloneqq J_\Van \sqcup J_\NullNotVan \sqcup J_\WobsNotNull \sqcup J_\TotalNotWobs,\\
		J_\Wobs &\coloneqq J_\Van \sqcup J_\NullNotVan \sqcup J_\WobsNotNull,\\
		J_\Null &\coloneqq J_\Van \sqcup J_\NullNotVan,\\
	\end{split}
	\end{equation}
	such that
	\begin{equation}
		\begin{split}
			e_j\at{C} &= 0\\
			e_j\at{C} &= f_j\\
			e_j\at{C} &= g_j\\
			e_j\at{C} &= h_j
		\end{split}
		\qquad
		\begin{split}
			e^j\at{C} &= 0\\
			e^j\at{C} &= f^j\\
			e^j\at{C} &= g^j\\
			e^j\at{C} &= h^j
		\end{split}
		\qquad
		\begin{split}
			&\text{for all } j \in J_\Van, \\
			&\text{for all } j \in J_\NullNotVan, \\
			&\text{for all } j \in J_\WobsNotNull, \\
			&\text{for all } j \in J_\TotalNotWobs.
		\end{split}
	\end{equation}
\end{lemma}

\begin{proof}
	From the given data we obtain a dual basis
	$\{c_j, c^j\}_{i \in J_C}$ of $\iota^\#E_\Total$ with
	$J_C = J_\Null \sqcup J_\WobsNotNull \sqcup J_\TotalNotWobs$
	\begin{equation*}
		c_j = \begin{cases}
			f_j & \text{ if } j \in J_\Null \\
			g_j & \text{ if } j \in J_\WobsNotNull \\
			h_j & \text{ if } j \in J_\TotalNotWobs
		\end{cases}\quad \text{ and }\quad
		c^j = \begin{cases}
			f^j & \text{ if } j \in J_\Null \\
			g^j & \text{ if } j \in J_\WobsNotNull \\
			h^j & \text{ if } j \in J_\TotalNotWobs
		\end{cases}.
	\end{equation*}
	To extend the dual basis to all of $M$ we choose a tubular neighbourhood 
	$\pr_V \colon V \to C$, with $\iota_V \colon V \hookrightarrow M$
	an open neighbourhood of $C$.
	Then we can pull back the $c_j$ and $c^j$ to obtain a dual basis of 
	$\iota_V^\#E_\Total$, which we again denote by $\{c_j,c^j\}_{j\in J_C}$.
	On the open subset $\iota_{M\setminus C} \colon M\setminus C \hookrightarrow M$
	choose another dual basis $\{d_j,d^j\}_{j \in J_\Van}$ of $\iota_{M\setminus C}^\#E_\Total$.
	We now need to patch these dual bases together.
	For this choose a quadratic partition of unity $\chi_1, \chi_2 \in \Cinfty(M)$
	with $\chi_1^2 + \chi_2^2 = 1$ and $\supp \chi_1 \subseteq V$ and
	$\supp \chi_2 \subseteq M \setminus C$.
	Then $\{e_j, e^j\}_{j \in J_\Total}$ with $J_\Total = J_C \sqcup J_\Van$ defined by
	\begin{equation*}
		e_j = \begin{cases}
			\chi_1\cdot c_j & \text{ if } j \in J_C \\
			\chi_2 \cdot d_j & \text{ if } j \in J_\Van
		\end{cases}
		\quad \text{ and }\quad
		e^i = \begin{cases}
			\chi_1 \cdot c^j & \text{ if } j \in J_C \\
			\chi_2 \cdot d^j & \text{ if } j \in J_\Van
		\end{cases}
	\end{equation*}
	forms a dual basis for $E_\Total$.
	It remains to show that this dual basis fulfils the properties of
	\autoref{def:ConProjectiveModules}.
	For this consider the constraint set $J$ with $J_\Total$ as above,
	$J_\Wobs \coloneqq J_\Van \sqcup J_\NullNotVan \sqcup J_\WobsNotNull$ and
	$J_\Null = J_\Van \sqcup J_\NullNotVan$.
	By construction we have $e_i \in \ConSecinfty(E)_\Wobs$ for $i \in I_\Wobs$
	and $e_i \in \ConSecinfty(E)_\Null$ for $i \in I_\Null$.
	From the fact that $g^j \in \Secinfty(\Ann E_\Null)$ and $h^j \in \Secinfty(\Ann E_\Wobs)$
	it follows that $e^i \in \ConSecinfty(E^*)_\Wobs$ for $i \in I_\Total \setminus I_\Null$
	and $e^i \in \ConSecinfty(E^*)_\Null$ for $i \in I_\Total \setminus I_\Wobs$.	
\end{proof}

Conversely, it can be shown that every finitely generated projective strong
constraint module over $\ConCinfty(\mathcal{M})$ is given by
the sections of a constraint vector bundle.
Together this is the content of the constraint Serre-Swan theorem:

\begin{theorem}[Constraint Serre-Swan]
	\label{thm:strConSerreSwan}
	Let $\mathcal{M} = (M,C,D)$ be a constraint manifold.
	The functor $\ConSecinfty \colon \ConVect(\mathcal{M}) \to \strConProj({\ConCinfty(\mathcal{M})})$
	is an equivalence of categories.
\end{theorem}

\begin{proof}
	We can always find complements and dual bases as required in
	\autoref{lem:ConstructionConDualBasis}.
	In particular, by choosing a dual basis for $E_\red$ and pulling it back to $C$
	wo obtain a dual basis of $E_\WobsNotNull$ fulfilling the required properties.
	This shows that $\ConSecinfty$ is indeed a functor into the category of finitely generated projective strong constraint modules.
	This functor is fully-faithful by the classical Serre-Swan Theorem.
	The fact that every finitely generated projective constraint module over $\ConCinfty(M)$
	is given by sections of a constraint vector bundle is somewhat easier to prove and will not be important to us.
	A proof for that can be found in \cite{dippell:2023a} and in a slightly different setting in
	\cite{dippell.menke.waldmann:2022a}.
	Together this shows that $\ConSecinfty$ is an equivalence of categories
\end{proof}

\begin{remark}
	In \cite{dippell.menke.waldmann:2022a} a similar result for
	non-strong projective constraint modules over
	$\ConCinfty(\mathcal{M})$ as a non-strong constraint algebra was found.
	The geometric objects used there are similar but not identical to the notion of constraint vector bundles, in particular the vector bundle $E_\Wobs$ is a subbundle of $E_\Total$ defined on all of $M$, and $\nabla$ is a partial connection on $\iota^\#E_\Wobs$ instead of $E_\Wobs / E_\Null$.
\end{remark}

As a first important property of the sections functor we show that it is compatible with direct sums.

\begin{proposition}
	\label{prop:ConSectionsDirectSum}
	Let $E = (E_\Total, E_\Wobs, E_\Null, \nabla^E)$ and
	$F = (F_\Total, F_\Wobs, F_\Null, \nabla^F)$ be constraint vector bundles over a constraint manifold
	$\mathcal{M} = (M,C,D)$.
	Then
	\begin{equation}
		\ConSecinfty(E \oplus F) \simeq \ConSecinfty(E) \oplus \ConSecinfty(F)
	\end{equation}
	as strong constraint $\ConCinfty(\mathcal{M})$-modules.
\end{proposition}

\begin{proof}
	From classical differential geometry we know that
	$\Phi \colon \Secinfty(E_\Total) \oplus \Secinfty(F_\Total) \to \Secinfty(E_\Total \oplus F_\Total)$
	given by $\Phi(s,s')(p) \coloneqq s(p) \oplus s'(p)$
	is an isomorphism of $\Cinfty(M)$-modules.
	Now let $s \in \ConSecinfty(E)_\Wobs$ and $s' \in \ConSecinfty(F)_\Wobs$ be given.
	Then clearly $\Phi(s,s')(p) = s(p) \oplus s'(p) \in E_\Wobs\at{p} \oplus F_\Wobs\at{p}$ for all $p \in C$.
	Moreover, it holds
	\begin{align}
		\nabla^\oplus_X \cc{\Phi(s,s')}
		= \nabla^E_X \cc{s} \oplus \nabla^F_X \cc{s'}
		= 0
	\end{align}
	by the definition of $\nabla^\oplus$ in \autoref{prop:ConstructionsCVect}.
	Thus $\Phi$ preserves the $\WOBS$-component.
	Next, let $s \in \ConSecinfty(E)_\Null$ and $s' \in \ConSecinfty(F)_\Null$ be given.
	Then $\Phi(s,s')(p) = s(p) \oplus s'(p) \in E_\Null\at{p} \oplus F_\Null\at{p}$
	for all $p \in C$
	shows that $\Phi$ also preserves the $\NULL$-components.
	For $\Phi$ to be a constraint isomorphism it remains to show that
	$\Phi^{-1}(\ConSecinfty(E \oplus F)_\Null) = \ConSecinfty(E)_\Null \oplus \ConSecinfty(F)_\Null$, cf. \autoref{prop:MonoEpisConModk}.
	For this let $t \in \ConSecinfty(E \oplus F)_\Null$ be given.
	Then we know that $t = s \oplus s'$ for some
	$s \in \Secinfty(E_\Total)$ and $t \in \Secinfty(F_\Total)$.
	For all $p \in C$ we have
	\begin{equation*}
		s(p) \oplus s'(p) = (s \oplus s')(p) = t(p)
		\in (E \oplus F)_\Null\at{p} = E_\Null\at{p} \oplus F_\Null\at{p},
	\end{equation*} 
	and thus $s \in \ConSecinfty(E)_\Null$ and $s' \in \ConSecinfty(F)_\Null$.
	Therefore, $\Phi$ is a constraint isomorphism.
\end{proof}

Similarly, sections of constraint vector bundles are compatible with internal homs:

\begin{proposition}
	\label{prop:ConSectionsHom}
	Let $E = (E_\Total, E_\Wobs, E_\Null, \nabla^E)$ and
	$F = (F_\Total, F_\Wobs, F_\Null, \nabla^F)$ be constraint vector bundles over a constraint manifold
	$\mathcal{M} = (M,C,D)$.
	Then
	\begin{equation}
		\ConSecinfty(\ConHom(E,F)) \simeq \ConHom_{\ConCinfty(\mathcal{M})}(\ConSecinfty(E), \ConSecinfty(F))
	\end{equation}
	as strong constraint $\ConCinfty(\mathcal{M})$-modules.
\end{proposition}

\begin{proof}
	On the $\TOTAL$-component we have the isomorphism
	\begin{equation*}
		\eta \colon \Secinfty(\Hom(E_\Total,F_\Total)) \to \Hom_{\Cinfty(M)}(\Secinfty(E_\Total),\Secinfty(F_\Total))
	\end{equation*}
	given by
	\begin{equation*}
		\eta(A)(s)\at{p} \coloneqq A\at{p}(s\at{p})
	\end{equation*}
	for all $p \in M$ and $s \in \Secinfty(E_\Total)$.
	We first show that $\eta$ is indeed a constraint morphism:
	If $A \in \ConSecinfty(\ConHom(E,F))_\Null$, then for every $p \in C$ and $s \in \ConSecinfty(E)_\Null$
	we have $\eta(A)(s)\at{p} = A\at{p}(s\at{p}) \in F_\Null\at{p}$
	since $s\at{p} \in E_\Null\at{p}$.
	Thus $\eta$ preserves the $\NULL$-component.
	Consider now $A \in \ConSecinfty(\ConHom(E,F))_\Wobs$.
	For all $p \in C$ and $s \in \ConSecinfty(E)_\Null$ we have
	$\eta(A)(s)\at{p} = A\at{p}(s\at{p}) \in F_\Null\at{p}$
	since $s\at{p} \in E_\Null\at{p}$.
	Moreover, if $s \in \ConSecinfty(E)_\Wobs$, then
	$\eta(A)(s)\at{p} = A\at{p}(s\at{p}) \in F_\Wobs\at{p}$
	and
	\begin{equation*}
		\nabla^F_X \cc{\eta(A)(s)\at{C}} = \cc{\eta}\big( \underbrace{\nabla^\ConHom_X \cc{A\at{C}}}_{=0}\big)(\cc{s\at{C}})
		+ \cc{\eta(A)}(\underbrace{\nabla^E_X \cc{s\at{C}}}_{=0})
		= 0.
	\end{equation*}
	Thus $\eta(A)(s) \in \ConSecinfty(F)_\Wobs$.
	Summarizing, this shows that $\eta$ is a constraint morphism.
	
	It remains to show that $\eta$ is regular surjective.
	For this recall from classical differential geometry that for every
	$A \in \Hom_{\Cinfty(M)}(\Secinfty(E_\Total),\Secinfty(F_\Total))$
	the corresponding preimage is given by
	$A(p)(s_p) \coloneqq A(s)(p)$
	for all $s_p \in E_\Total\at{p}$ and $s \in \Secinfty(E_\Total)$
	such that $s(p) = s_p$.
	Note that this does not depend on the choice of the section $s$.
	Here we use the usual abuse of notation.
	
	Now let $A \in \ConHom_{\ConCinfty(\mathcal{M})}(\ConSecinfty(E),\ConSecinfty(F))_\Wobs$ be given.
	Then for every $p \in C$ and $s_p \in E_\Wobs\at{p}$ there exists a section
	$s \in \ConSecinfty(E)_\Wobs$ with $s(p) = s_p$ by \autoref{lem:ExtendingToConSections}.
	Then we have $A(p)(s_p) = A(s)(p) \in F_\Wobs\at{p}$
	since $A(s) \in \ConSecinfty(F)_\Wobs$.
	Similarly, if $s_p \in E_\Null\at{p}$, then
	there exists $s \in \ConSecinfty(E)_\Null$ with $s(p) = s_p$
	and thus
	$A(p)(s_p) = A(s)(p) \in F_\Null\at{p}$.
	We also need to show that $\nabla^{\ConHom}_X\cc{A\at{C}} = 0$
	for all $X \in \Secinfty(D)$.
	For this let $p \in C$ and $\cc{s_p} \in (E_\Wobs / E_\Null)\at{p}$
	be given.
	Again by \autoref{lem:ExtendingToConSections} we find a section $s \in \ConSecinfty(E)_\Wobs$
	such that $\cc{s}(p) = \cc{s_p}$.
	Then
	\begin{align*}
		\big(\nabla^{\ConHom}_X\cc{A\at{C}}\big)(p)(\cc{s_p})
		&= \big(\nabla^\ConHom_X \cc{A\at{C}}\big)(\cc{s\at{C}})(p) \\
		&= \nabla^F_X\big(\cc{A\at{C}}(\cc{s\at{C}})\big)(p) - \cc{A\at{C}}\big(\nabla^E_X \cc{s\at{C}}\big)\\
		&= \nabla^F_X\big(\cc{A(s)\at{C}}\big)(p)\\
		&= 0,
	\end{align*}
	since $\nabla^E_X\cc{s\at{C}} = 0$
	and $A(s)\at{C} \in \ConSecinfty(F)_\Wobs$.
	This shows $A \in \ConSecinfty(\ConHom(E,F))_\Wobs$, and hence
	$\eta$ is surjective on the $\WOBS$-component.
	
	Recall that we additionally have to check that $\eta$ is also surjective on the $\NULL$-component.
	Thus let $A \in \ConHom_{\ConCinfty(\mathcal{M})}(\ConSecinfty(E),\ConSecinfty(F))_\Null$ be given.
	Then for $p \in C$ and $s_p \in E_\Wobs\at{p}$ we find again by \autoref{lem:ExtendingToConSections}
	a section $s \in \ConSecinfty(E)_\Wobs$ with $s(p) = s_p$.
	Then
	$A(p)(s_p) = A(s)(p) \in F_\Null\at{p}$ holds
	since we have $A(s) \in \ConSecinfty(F)_\Null$.
	This finally shows that $\eta$ is a regular epimorphism and hence a constraint isomorphism.
\end{proof}

\begin{corollary}
	\label{cor:SectionsOfDualConstraintVectorBundle}
	Let $E = (E_\Total,E_\Wobs,E_\Null,\nabla^E)$ be a constraint vector bundle over a constraint
	manifold $\mathcal{M} = (M,C,D)$.
	Then 
	\begin{equation}
		\ConSecinfty(E^*) \simeq \ConSecinfty(E)^*
	\end{equation}
	as strong constraint $\ConCinfty(\mathcal{M})$-modules.
\end{corollary}

\begin{proof}
	Choose $F = \mathcal{M} \times \Reals$ in \autoref{prop:ConSectionsHom}.
\end{proof}

\begin{proposition}
	Let $E = (E_\Total, E_\Wobs, E_\Null, \nabla^E)$ and
	$F = (F_\Total, F_\Wobs, F_\Null, \nabla^F)$ be constraint vector bundles over a constraint manifold
	$\mathcal{M} = (M,C,D)$.
	\begin{propositionlist}
		\item The sections functor
		$\ConSecinfty \colon (\ConVect(\mathcal{M}), \tensor)
		\to \big(\strConProj\big(\ConCinfty(\mathcal{M})\big), \tensor \big)$
		is monoidal, i.e. there exist natural isomorphisms
		\begin{equation}
			\ConSecinfty(E) \tensor[\ConCinfty(\mathcal{M})] \ConSecinfty(F)
		\simeq \ConSecinfty(E \tensor F)
		\end{equation}
		of strong constraint $\ConCinfty(\mathcal{M})$-modules.
		\item The sections functor
		$\ConSecinfty \colon (\ConVect(\mathcal{M}), \strtensor)
		\to \big(\strConProj\big(\ConCinfty(\mathcal{M})\big), \strtensor \big)$
		is monoidal,
		i.e. there exist natural isomorphisms
		\begin{equation}
			\ConSecinfty(E) \strtensor[\ConCinfty(\mathcal{M})] \ConSecinfty(F)
		\simeq \ConSecinfty(E \strtensor F)
		\end{equation}
		of strong constraint $\ConCinfty(\mathcal{M})$-modules.
	\end{propositionlist}
\end{proposition}

\begin{proof}
	Consider the maps 
	\begin{align*}
		\label{eq:ConSecinftyLaxMonoidal}
		I_{E,F} \colon \ConSecinfty(E) \tensor[\ConCinfty(\mathcal{M})] \ConSecinfty(F)
		\to \ConSecinfty(E \tensor F),\qquad
		I_{E,F}(s \tensor t)(p) \coloneqq s(p) \tensor t(p)
		\shortintertext{and}
		J_{E,F} \colon \ConSecinfty(E) \strtensor[\ConCinfty(\mathcal{M})] \ConSecinfty(F)
		\to \ConSecinfty(E \strtensor F),\qquad
		J_{E,F}(s \tensor t)(p) \coloneqq s(p) \tensor t(p).
	\end{align*}
	By carefully checking the definitions of the $\WOBS$- and $\NULL$-components it can be shown that these maps are in fact constraint, see \cite{dippell:2023a} for details.
	Moreover, by classical theory it is clear that these are components of a natural transformation.
	Let us show that the natural transformations 
	$I$ and $J$
	are in fact natural isomorphisms.
	For this we construct inverses.
	Let $E, F \in \ConVect(\mathcal{M})$ and let
	$(\{e_i\}_{i \in M}, \{e^i\}_{i \in M^*})$ as well as
	$(\{f_j\}_{j \in N}, \{f^j\}_{j \in N^*})$ be finite dual bases of $E$ and $F$, respectively.
	From classical differential geometry we know that
	$(\{e_i \tensor f_j\}_{(i,j) \in M_\Total \times N_\Total}, \{e^i \tensor f^j\}_{(i,j) \in M_\Total \times N_\Total})$
	is a dual basis of $\Secinfty(E_\Total \tensor F_\Total)$
	and that
	\begin{equation*}
		K(X) = \sum_{i \in M_\Total} \sum_{j \in N_\Total}
		(e^i \tensor f^j)(X) \cdot e_i \tensor[\Cinfty(M)] f_j,
	\end{equation*}
	for $X \in \Secinfty(E_\Total \tensor F_\Total)$,
	defines an inverse 
	$K \colon \Secinfty(E \tensor F) \to \Secinfty(E) \tensor[\Cinfty(M)] \Secinfty(F)$
	to $I$.
	To show that $K$ is a constraint morphism we prove that 
	the families
	$(\{e_i \tensor f_j\}_{(i,j) \in M \tensor N}, \{e^i \tensor f^j\}_{(i,j) \in (M \tensor N)^*})$
	form a dual basis for $\ConSecinfty(E \tensor F)$:
	\begin{cptitem}
		\item $(i,j) \in (M \tensor N)_\Wobs$:
		Then we know that
		\begin{equation*}
			e_i \tensor f_j = I_{E,F}(e_i \tensor[\Cinfty(M)] f_j) \in \ConSecinfty(E \tensor F)_\Wobs.
		\end{equation*}
		\item $(i,j) \in (M \tensor N)_\Null$:
		Then we know that
		$e_i \tensor f_j = I_{E,F}(e_i \tensor[\Cinfty(M)] f_j) \in \ConSecinfty(E \tensor F)_\Null$.
		\item $(i,j) \in (M \tensor N)^*_\Wobs = (M^* \strtensor N^*)_\Wobs$:
		Then we know that
		\begin{equation*}
			e^i \tensor f^j = J_{E^*,F^*}(e^i \tensor[\Cinfty(M)] f^j) \in \ConSecinfty(E^* \strtensor F^*)_\Wobs \simeq \ConSecinfty(E \tensor F)^*_\Wobs.
		\end{equation*}
		\item $(i,j) \in (M \tensor N)^*_\Null$:
		Then we know that
		\begin{equation*}
			e^i \tensor f^j = J_{E^*,F^*}(e^i \tensor[\Cinfty(M)] f^j) \in \ConSecinfty(E^* \strtensor F^*)_\Null
			\simeq \ConSecinfty(E \tensor F)^*_\Null.
		\end{equation*}
	\end{cptitem}
	This shows that $K$ is a constraint morphism, and therefore $I$ is an isomorphism.
	With completely analogous arguments, on can show that $J$ is an isomorphism as well.
	The unit object in $\ConVect(\mathcal{M})$ is for both products given by 
	$\mathcal{M} \times \Reals$.
	Since $\ConSecinfty(\mathcal{M} \times \Reals) \simeq \ConCinfty(\mathcal{M})$
	the section functor preserves the monoidal units, and hence gives a monoidal functor in both cases.
\end{proof}

From now on we will write $\tensor$ and $\strtensor$ instead of $\tensor[\ConCinfty(\mathcal{M})]$
and $\strtensor[\ConCinfty(\mathcal{M})]$.
We only indicate the algebra if it is not the algebra of functions on the underlying manifold.

Since $\ConSecinfty$ is compatible with direct sums and tensor products, we also get
\begin{align}
	\Anti^\bullet_{\tensor} \ConSecinfty(E)
	\simeq \ConSecinfty(\Anti^\bullet_{\tensor} E),\\
	\shortintertext{and}
	\Anti^\bullet_{\strtensor} \ConSecinfty(E)
	\simeq \ConSecinfty(\Anti^\bullet_{\strtensor} E)
\end{align}
for any constraint vector bundle $E$.

Moreover, we can use now \autoref{prop:DualTensorHomIsos}
to infer the following compatibilities of the constraint homomorphism bundle, dual and tensor products:

\begin{proposition}
	\label{prop:PropsOfStrongDuals}
	Let $\mathcal{M} = (M,C,D)$ be a constraint manifold and let 
	$E$, $F$, $G$ be constraint vector bundles over $\mathcal{M}$.
	\begin{propositionlist}
		\item \label{prop:PropsOfStrongDuals_1}
		We have $(E \oplus F)^* \simeq E^* \oplus F^*$.
		\item \label{prop:PropsOfStrongDuals_4}
		We have $\ConHom(E,F) \simeq E^* \strtensor F$.
		\item \label{prop:PropsOfStrongDuals_2}
		We have $(E \tensor F)^* \simeq E^* \strtensor F^*$.
		\item \label{prop:PropsOfStrongDuals_3}
		We have $(E \strtensor F)^* \simeq E^* \tensor F^*$.
		\item \label{prop:PropsOfStrongDuals_5}
		We have $\ConHom(E \tensor F, G) \simeq \ConHom(E, F^* \strtensor G)$.
	\end{propositionlist}
\end{proposition}

So far we only consider the sections functor $\ConSecinfty$ on the category
$\ConVect(\mathcal{M})$ of constraint vector bundles over a fixed constraint manifold.
Already in classical differential geometry the sections functor will not be functorial when extended 
to the category of vector bundles over arbitrary base manifolds.
Nevertheless, as in classical geometry, we can remedy this defect by instead taking sections of the dual bundle.

\begin{proposition}
	\label{prop:SectionsOverGeneralBase}
	For every morphism of constraint vector bundles
	$\Phi \colon E \to F$ over
	$\phi \colon \mathcal{M} \to \mathcal{N}$
	defining
	\begin{equation} \label{eq:PullbackMap}
		\Phi^*(\alpha)\at{p}(s_p) \coloneqq \alpha\at{\phi(p)}(\Phi(s_p))
	\end{equation}
	yields a morphism
	$\Phi^* \colon \ConSecinfty(F^*) \to \ConSecinfty(E^*)$
	of strong constraint modules along the constraint algebra morphism
	$\phi^* \colon \ConCinfty(\mathcal{N}) \to \ConCinfty(\mathcal{N})$
	given as in  
	\autoref{prop:FunctionsOnConManifolds}
\end{proposition}

\begin{proof}
	From classical differential geometry we know that $\Phi^*$ defines a module morphism along $\phi^*$ on the total components.
	Thus it only remains to show that $\Phi^*$ is indeed a constraint morphism.
	This is a straightforward check using \eqref{eq:PullbackMap}:
	If $\alpha \in \ConSecinfty(F^*)_\Wobs$, then for every $s_p \in E_\Null\at{p}$
	we get $\Phi(s_p) \in F_\Null\at{\phi(p)}$
	and thus $\alpha\at{\phi(p)}(\Phi(s_p)) = 0$.
	Moreover, for 
	$\nabla_X^{E^*} \overline{\alpha}$ vanishes
	for all $X \in \Secinfty(D)$
	since $\Phi$ is a constraint morphism, see \eqref{eq:ConstraintVBMorphism}.
	Thus we know that $\Phi^*$ preserves the $\WOBS$-component.
	Similarly, one obtains that $\Phi^*$ also preserves the $\NULL$-component.
\end{proof}

We finish this section with the compatibility of reduction with taking sections.

\begin{proposition}[Constraint sections vs. reduction]
	\label{prop:ConSecVSReduction}
	Let $\mathcal{M} = (M,C,D)$ be a constraint manifold.
	There exists a natural isomorphism making the following diagram commute:
	\begin{equation}
		\begin{tikzcd}
			\ConVect(\mathcal{M})
			\arrow[r,"\ConSecinfty"]
			\arrow[d,"\red"{swap}]
			& \strConProj(\ConCinfty(\mathcal{M}))
			\arrow[d,"\red"] \\
			\Vect(\mathcal{M}_\red)
			\arrow[r,"\Secinfty"]
			& \strConProj(\Cinfty(\mathcal{M}_\red))
		\end{tikzcd}
	\end{equation}
\end{proposition}

\begin{proof}
	Our goal is to construct an isomorphism
	$\eta_E \colon \ConSecinfty(E)_\red \to \Secinfty(E_\red)$
	for every constraint vector bundle $E$ over $\mathcal{M}$.
	Thus let $s \in \ConSecinfty(E)_\Wobs$ be given.
	For any (possibly non-smooth) section
	$\sigma \colon \mathcal{M}_\red \to C$
	of the quotient map $\pi_\mathcal{M}$
	we can define a map
	$\eta_E(s) \colon \mathcal{M}_\red \to E_\red$
	by
	$\eta_E(s)(p) \coloneqq [s(\sigma(p))]$,
	which is a section of
	the vector bundle projection $\pr_{E_\red}$.
	Note that this map is independent of the choice
	of the section $\sigma$, since $s \in \ConSecinfty(E)_\Wobs$.
	Thus $\eta_E(s)$ is also smooth, since locally we can choose
	$\sigma$ to be smooth.
	So we end up with $\eta_E(s) \in \Secinfty(E_\red)$.
	Note also that $\eta_E$ is clearly $\ConCinfty(\mathcal{M})_\Wobs$-linear
	along the projection
	$\pi_{\ConCinfty(\mathcal{M})} \colon \ConCinfty(\mathcal{M})_\Wobs
	\to \Cinfty(M_\red)$.
	Now suppose $\eta_E(s) = 0$.
	Then $[s(\sigma(p))] = 0$ for all $p \in \mathcal{M}_\red$ and every section $\sigma$.
	Thus $\iota^\#s \in \Secinfty(E_\Null)$.
	This means that $\ker \eta_E = \ConSecinfty(E)_\Null$
	and therefore it induces an injective morphism
	$\eta_E \colon \ConSecinfty(E)_\red \to \Secinfty(E_\red)$
	of $\ConCinfty(\mathcal{M})_\red \simeq \Cinfty(\mathcal{M}_\red)$-modules.
	It remains to show that $\eta_E$ is also surjective.
	For this let $t \in \Secinfty(E_\red)$ be given.
	Now choose a splitting $E_\Wobs \simeq E_\Null \oplus \iota^\#E_\red$
	using \autoref{prop:ReductionOfConVect} \ref{prop:ReductionOfConVect_2}
	and define
	$s(q) \coloneqq \Theta^{-1}(t(\pi_\mathcal{M}(q)))$ for all $q \in C$
	and extend it to a section of $E_\Total$ by use of a tubular neighbourhood.
	By \eqref{eq:ReductionOfConVect_Parallel} $s$ is covariantly constant
	as a section of $E_\Wobs / E_\Null$, and therefore we have
	$s \in \ConSecinfty(E)_\Wobs$.
	Finally, we have $\eta_E(s) = t$, showing that $\eta_E$ is surjective,
	and thus an isomorphism.	
\end{proof}

By \autoref{prop:ReductionAndDirectSumsHoms} we know that the reduction of constraint modules
is compatible with direct sums.
Using the constraint Serre-Swan Theorem we obtain an analogous argument for constraint vector bundles.
The fact that reduction is also compatible with inner homs and tensor products needs additional geometric features.

\begin{proposition}
	\label{prop:RedConstructionsConVect}
	Let $\mathcal{M} = (M_\Total,M_\Wobs,D_{\mathcal{M}})$ and 
	$\mathcal{N} = (N_\Total, N_\Wobs, D_{\mathcal{N}})$ be constraint manifolds
	and $\phi \colon \mathcal{M} \to \mathcal{N}$ a constraint map.
	Moreover, let $E, F \in \Vect(\mathcal{M})$ be constraint vector bundles over 
	$\mathcal{M}$ and let $G \in \Vect(\mathcal{N})$ be a
	constraint vector bundle over $\mathcal{N}$.
	\begin{propositionlist}
		\item \label{prop:RedConstructionsConVect_1}
		There exists a canonical isomorphism
		$(E \oplus F)_\red \simeq E_\red \oplus F_\red$.
		\item \label{prop:RedConstructionsConVect_2}
		There exists a canonical isomorphism
		$(E \tensor F)_\red \simeq E_\red \tensor F_\red$.
		\item \label{prop:RedConstructionsConVect_3}
		There exists a canonical isomorphism
		$(E \strtensor F)_\red \simeq E_\red \tensor F_\red$.
		\item \label{prop:RedConstructionsConVect_4}
		There exists a canonical isomorphism
		$\ConHom(E,F)_\red \simeq \Hom(E_\red, F_\red)$.
		\item \label{prop:RedConstructionsConVect_5}
		There exists a canonical isomorphism
		$(E^*)_\red \simeq (E_\red)^*$.
		\item \label{prop:RedConstructionsConVect_6}
		There exists a canonical isomorphism 
		$(\phi^\# G)_\red \simeq \phi^\#_\red G_\red$. 
	\end{propositionlist}
\end{proposition}

\begin{proof}
	Statement \ref{prop:RedConstructionsConVect_1}
	follows directly from \autoref{prop:ReductionAndDirectSumsHoms}.
	Moreover, statement \ref{prop:RedConstructionsConVect_5}
	follows from \ref{prop:RedConstructionsConVect_4}.
	And \ref{prop:RedConstructionsConVect_3} can be proved completely analogous to \ref{prop:RedConstructionsConVect_2}.
	Thus we only show \ref{prop:RedConstructionsConVect_2},
	\ref{prop:RedConstructionsConVect_4}
	and
	\ref{prop:RedConstructionsConVect_6}.
	
	For this we pull all vector bundles back to $M_\Wobs$ along 
	$\pi_\mathcal{M} \colon M_\Wobs \to \mathcal{M}_\red$ and then use 
	\autoref{prop:ReductionOfConVect} \ref{prop:ReductionOfConVect_2} to compare 
	them.
	We have
	\begin{align*}
		\pi_\mathcal{M}^\# (E \tensor F)_\red
		&\simeq \frac{(E \tensor F)_\Wobs}{(E \tensor F)_\Null}
		\simeq \frac{E_\Wobs}{E_\Null} \tensor \frac{F_\Wobs}{F_\Null}
		\simeq E_\red \tensor F_\red
		\simeq \pi_\mathcal{M}^\#(E_\red \tensor F_\red),
	\end{align*}
	as well as
	\begin{align*}
		\pi_\mathcal{M}^\#\ConHom(E,F)_\red
		&\simeq \frac{\ConHom(E,F)_\Wobs}{\ConHom(E,F)_\Null}
		\simeq \Hom(\frac{E_\Wobs}{E_\Null},\frac{F_\Wobs}{F_\Null})\\
		&\simeq \Hom(\pi_\mathcal{M}^\#E_\red,\pi_\mathcal{M}^\#F_\red)\\
		&\simeq \pi_\mathcal{M}^\#\Hom(E_\red,F_\red),
	\end{align*}
	and similarly
	\begin{equation*}
		\pi_\mathcal{M}^\#(\phi^\#G)_\red
		\simeq \frac{\phi_\Wobs^\#G_\Wobs}{\phi_\Wobs^\#G_\Null}
		\simeq \phi_\Wobs^\#\frac{G_\Wobs}{G_\Null}
		\simeq \phi_\Wobs^\#\pi_\mathcal{N}^\#G_\red
		\simeq \pi_\mathcal{M}^\#\phi_\red^\# G_\red.
	\end{equation*}
	Since $\pi_\mathcal{M}$ is a surjective submersion this is be enough to infer
	isomorphy on $\mathcal{M}_\red$.
\end{proof}

The above isomorphisms can be shown to be part of natural isomorphisms,
turning the functor $\red \colon \ConVect(\mathcal{M}) \to \Vect(\mathcal{M}_\red)$
into an additive, closed and monoidal functor with respect to both tensor products.

\begin{example}[Almost complex structures]
	\label{ex:AlmostComplex}
	Let $\mathcal{M} = (M,C,D)$ be a constraint manifold
	and let $J \in \Secinfty(\End(TM))$ be an almost complex structure on $M$, 
	i.e. $J^2 = - \id_{TM}$.
	If $J \in \ConSecinfty(\ConEnd(T\mathcal{M}))_\Wobs$,
	then it reduces to an almost complex structure $J_\red$
	on $\mathcal{M}_\red$.
\end{example}

\subsection{Constraint Cartan Calculus}
\label{sec:CartanCalculus}

In this last section on the geometry of constraint manifolds
we extend the classical Cartan calculus to the constraint setting.
In particular we want to study the relation between constraint differential forms and constraint multivector fields.
Let us start with a local characterization of constraint vector fields.
For this recall from \autoref{sec:ConIndexSets}
that the dual $n^*$ of a constraint index set
$n = (n_\Total, n_\Wobs, n_\Null)$
is given by
$n^* = (n_\Total,\, n_\Total \setminus n_\Null,\, n_\Total \setminus n_\Wobs)$.

\begin{lemma}
	\label{lem:LocalConVect}
	Let $\mathcal{M} = (M,C,D)$ be a constraint manifold of dimension 
	$n = (n_\Total,n_\Wobs,n_\Null)$ and consider
	$X \in \Secinfty(T M)$.
	\begin{lemmalist}
		\item \label{lem:LocalConVect_1}
		We have
		$X \in \ConSecinfty(T\mathcal{M})_\Wobs$
		if and only if for every adapted chart $(U,x)$ around $p \in C$ it holds
		\begin{align}
			X^i &\in \ConCinfty(\mathcal{M}\at{U})_\Wobs \text{ if } i \in (n^*)_\Wobs,\\
			X^i &\in \ConCinfty(\mathcal{M}\at{U})_\Null \text{ if } i \in (n^*)_\Null,
		\end{align}
		where $X\at{U} = \sum_{i=1}^{n_\Total} X^i \frac{\del}{\del x^i}$.
		\item \label{lem:LocalConVect_2}
		We have $X \in \ConSecinfty(T\mathcal{M})_\Null$
		if and only if for every adapted chart around $p \in C$ it holds
		\begin{align}
			X^i &\in \ConCinfty(\mathcal{M}\at{U})_\Null \text{ if } 
			i \in (n^*)_\Wobs,
		\end{align}
		where $X\at{U} = \sum_{i=1}^{n_\Total} X^i \frac{\del}{\del x^i}$.
	\end{lemmalist}
\end{lemma}

\begin{proof}
	By \autoref{ex:ConSections}
	locally we always find adapted coordinates such that
	\begin{align*}
		\iota^\#(\frac{\del}{\del x^i}) &\in \Secinfty(D\at{U}), \text{ if } i \in \{1, \dotsc, n_\Null\}, \\
		\shortintertext{and}
		\iota^\#(\frac{\del}{\del x^i}) &\in \Secinfty(TC\at{U}), \text{ if } i \in \{n_\Null +1, \dotsc, n_\Wobs\}.
	\end{align*}
	We have $X \in \ConSecinfty(T \mathcal{M})_\Wobs$
	if and only if $\iota^\# X \in \Secinfty(TC)$
	and $[Y, \iota^\#X] \in \Secinfty(D)$ hold for all $Y \in \Secinfty(D)$.
	The first condition exactly means that locally we have
	$X^i \in \ConCinfty(\mathcal{M}\at{U})_\Null$ for all $i \in \{ n_\Wobs +1, \dotsc, n_\Total \} = (n^*)_\Null$.
	Moreover, since $D$ is locally spanned by $\frac{\del}{\del x^1}, \dotsc, \frac{\del}{\del x^{n_\Null}}$
	the second condition shows $X^i \in \ConCinfty(\mathcal{M})_\Wobs$ for
	$i \in \{ n_\Null +1, \dotsc, n_\Total \} = (n^*)_\Wobs$.
	This shows the first part.
	The second part follows since $X \in \ConSecinfty(T\mathcal{M})_\Null$
	if and only if $\iota^\#X \in \Secinfty(D)$.
\end{proof}

With the help of this local characterization we can now identify 
constraint vector fields with constraint derivations, see 
\autoref{ex:ConLieAlg} \ref{ex:ConLieAlg_Derivations},
using the Lie derivative:

\begin{proposition}
	\label{prop:ConVectorFieldsAreConDer}
	Let $\mathcal{M} = (M,C,D)$ be a constraint manifold.
	Then
	\begin{equation}
		\Lie \colon \ConSecinfty(T\mathcal{M})
		\to \ConDer(\ConCinfty(\mathcal{M}))
	\end{equation}
	given by the Lie derivative is an isomorphism of constraint
	$\ConCinfty(\mathcal{M})$-modules.
\end{proposition}

\begin{proof}
	From classical differential geometry we know that $\Lie$ is an isomorphism on the
	$\TOTAL$-com\-po\-nents.
	To show that $\Lie$ is a constraint morphism consider
	$X \in \ConSecinfty(T \mathcal{M})_\Null$
	and $f \in \ConCinfty(\mathcal{M})_\Wobs$.
	Then
	\begin{equation*}
		(\Lie_X f)\at{C}
		= \Lie_{\iota^\#X} f\at{C}
		= 0,
	\end{equation*}
	since $X\at{C} \in \Secinfty(D)$.
	Thus $\Lie$ maps the $\NULL$-component to the $\NULL$-component.
	Now let $X \in \ConSecinfty(T \mathcal{M})_\Wobs$ be given.
	Then for $f \in \ConCinfty(\mathcal{M})_\Wobs$
	we get
	\begin{equation*}
		\Lie_Y (\Lie_X f)\at{C}
		= \Lie_{[Y,\iota^\#X]}f\at{C} 
		+ \Lie_{\iota^\#X} \underbrace{\Lie_Y f\at{C}}_{=0}
		= \Lie_{[Y,\iota^\#X]} f\at{C} = 0,
	\end{equation*}
	for all $Y \in \Secinfty(D)$, since $[Y, \iota^\# X] \in \Secinfty(D)$ by 
	\eqref{eq:constraintVectorFields}. 
	Finally, for $f \in \ConCinfty(\mathcal{M})_\Null$
	we have $f\at{C} = 0$ and therefore
	\begin{equation*}
		(\Lie_X f)\at{C} = \Lie_{\iota^\#X}f\at{C} = 0,
	\end{equation*}
	which shows that $\Lie$ is a constraint morphism.
	Since the $\TOTAL$-component of $\Lie$ is just the classical Lie 
	derivative, which is an isomorphism, $\Lie$ is a constraint 
	monomorphism.
	To show that $\Lie$ is also a regular epimorphism let
	$D \in \ConDer(\ConCinfty(\mathcal{M}))_\Wobs$ be given.
	Since $D$ is in particular a derivation of $\Cinfty(M)$
	we know that there exists $X \in \Secinfty(TM)$ such that
	$\Lie_X = D$.
	Choose an adapted chart $(U,x)$ around $p \in C$, then
	$X\at{U} = \sum_{i=1}^{n_\Total} X^i \frac{\del}{\del x^i}$.
	Since $\Lie_X$ is a constraint derivation we get
	$X^i = \Lie_X(x^i) \in \ConCinfty(\mathcal{M}\at{U})_\Null$
	for all
	$i \in \{n_\Wobs +1, \dotsc, n_\Total\} = (n^*)_\Null$
	and 
	$X^i = \Lie_X(x^i) \in \ConCinfty(\mathcal{M}\at{U})_\Wobs$
	for all $i \in \{n_\Null +1, \dotsc, n_\Total\} = (n^*)_\Wobs$,
	by \autoref{ex:ConFunctionsOnConMfld} 
	\ref{ex:ConFunctionsOnConMfld_1}.
	And thus $X \in \ConSecinfty(T\mathcal{M})_\Wobs$
	using \autoref{lem:LocalConVect} \ref{lem:LocalConVect_1}.
	With the same line of reasoning we obtain $X \in 
	\ConSecinfty(T\mathcal{M})_\Null$ if
	$D \in \ConDer(\ConCinfty(\mathcal{M}))_\Null$,
	showing that $\Lie$ is a regular epimorphism, and
	therefore an isomorphism.
\end{proof}

With this we can transport the constraint Lie algebra 
structure on
$\ConDer(\ConCinfty(\mathcal{M}))$,
see \autoref{ex:ConLieAlg}~\ref{ex:ConLieAlg_Derivations},
to $\ConSecinfty(T \mathcal{M})$,
giving a constraint map
\begin{equation}
	[\argument,\argument] \colon \ConSecinfty(T\mathcal{M}) \tensor \ConSecinfty(T\mathcal{M}) \to \ConSecinfty(T\mathcal{M}).
\end{equation}
This is just the usual Lie bracket of vector fields, but now we see 
that it is actually compatible with the constraint structure.
Alternatively, one could directly check that the classical Lie bracket of vector fields
yields a constraint Lie algebra structure.

\begin{example}[Nijenhuis and complex structures]
	\label{ex:ReducibleNijenhuis}
	Let $\mathcal{M} = (M, C, D)$ be a constraint manifold and let 
	$A \in \Secinfty(\End(TM))$ be a Nijenhuis tensor, i.e. the associated 
	Nijenhuis torsion $N_A$, defined by
	\begin{equation}
		\label{eq:NijenhuisTorsion}
		N_A(X,Y) \coloneq [A(X), A(Y)] - A([A(X),Y]) - A([X,A(Y)]) + A^2([X,Y]) 
	\end{equation} 
	for $X, Y \in \Secinfty(TM)$, vanishes.
	If $A \in \ConSecinfty(\ConEnd(T\mathcal{M}))_\Wobs$ is a constraint 
	section, then its Nijenhuis torsion is clearly a constraint morphism 
	$N_A \colon \ConSecinfty(T\mathcal{M}) \wedge \ConSecinfty(T\mathcal{M}) \to \ConSecinfty(T\mathcal{M})$,
	which can be reformulated as
	$N_A \in \ConSecinfty(\Anti_{\strtensor}^2T^*\mathcal{M} \strtensor T\mathcal{M})_\Wobs$
	by \autoref{prop:PropsOfStrongDuals}.
	Since the Nijenhuis torsion $N_A$ reduces to the Nijenhuis torsion $N_{A_\red}$ of the reduced Nijenhuis manifold, we immediately
	see that if $N_A$ vanishes, then so does $N_{A_\red}$. 
	Hence if $A$ is a Nijenhuis structure on $\mathcal{M}$, then $A_\red$ is a 
	Nijenhuis structure on $\mathcal{M}_\red$, see also \cite{vaisman:1996a}. 
	In case the Nijenhuis structure is given by an almost complex structure 
	$J \in \ConSecinfty(\ConEnd(T\mathcal{M}))_\Wobs$, see 
	\autoref{ex:AlmostComplex}, it is called complex structure and it descends to 
	a complex structure $J_\red$ on $\mathcal{M}_\red$. 
\end{example}

With this at hand let us introduce constraint differential forms.
Since there are two tensor products available we can define 
constraint differential
forms in two ways.

\begin{definition}[Constraint Differential Forms]
	Let $\mathcal{M} = (M,C,D)$ be a constraint manifold.
	We denote by
	\begin{align}
		\ConForms_{\tensor}^\bullet(\mathcal{M}) 
		\coloneqq \Anti^\bullet_{\tensor} \ConSecinfty(T^* \mathcal{M})
		= \bigoplus_{k=0}^{\infty} \Anti^k_{\tensor} 
		\ConSecinfty(T^*\mathcal{M})
		\shortintertext{and}
		\ConForms_{\strtensor}^\bullet(\mathcal{M}) 
		\coloneqq \Anti^\bullet_{\strtensor} \ConSecinfty(T^* \mathcal{M})
		= \bigoplus_{k=0}^{\infty} \Anti^k_{\strtensor} 
		\ConSecinfty(T^*\mathcal{M})
	\end{align}
	the graded strong constraint modules of
	\emph{constraint differential forms} on $\mathcal{M}$.
\end{definition}

Note that $\ConForms_{\tensor}^\bullet(\mathcal{M}) \simeq (\Anti_{\strtensor}^\bullet \ConSecinfty(T\mathcal{M}))^*$
and $\ConForms_{\strtensor}^\bullet(\mathcal{M}) \simeq (\Anti_{\tensor}^\bullet \ConSecinfty(T\mathcal{M}))^*$.
Thus $\alpha \in \ConForms_{\tensor}^k(\mathcal{M})$ can be evaluated at
$X_1 \tensor \dots \tensor X_k \in \Anti_{\strtensor}^\bullet \ConSecinfty(T\mathcal{M})$,
while $\alpha \in \ConForms_{\strtensor}^k(\mathcal{M})$ can be evaluated at
$X_1 \tensor \dots \tensor X_k \in \Anti_{\tensor}^\bullet \ConSecinfty(T\mathcal{M})$.
For $\ConForms_{\strtensor}^\bullet(\mathcal{M})$
there is a good constraint Cartan calculus as we see in the following.

\begin{proposition}[Cartan calculus]
	\label{prop:ConCartanCalculus}
	Let $\mathcal{M} = (M,C,D)$ be a constraint manifold.
	\begin{propositionlist}
		\item $\ConForms_{\strtensor}^\bullet(\mathcal{M})$ is a 
		graded commutative strong constraint algebra with respect to the 
		wedge product $\wedge$.
		\item The insertion of vector fields into forms defines a 
		constraint $\ConCinfty(\mathcal{M})$-module morphism
		\begin{equation}
			\ins \colon \ConSecinfty(T\mathcal{M}) \to 
			\ConDer^{-1}(\ConForms_{\strtensor}^\bullet(\mathcal{M})),
		\end{equation}
		with $\ConDer^{-1}(\ConForms_{\strtensor}^\bullet(\mathcal{M}))$
		denoting the graded constraint derivations of degree $-1$.
		\item The Lie derivative defines a $\Reals$-linear constraint morphism
		\begin{equation}
			\Lie \colon \ConSecinfty(T\mathcal{M}) \to 
			\ConDer^{0}(\ConForms_{\strtensor}^\bullet(\mathcal{M}))
		\end{equation}
		into the graded constraint derivations of degree $0$ of 
		$\ConForms_{\strtensor}^\bullet(\mathcal{M})$.
		\item The de Rham differential defines a graded constraint 
		derivation
		\begin{equation}
			\D \colon \ConForms_{\strtensor}^\bullet(\mathcal{M}) \to 
			\ConForms_{\strtensor}^{\bullet+1}(\mathcal{M})
		\end{equation}
		of degree $+1$,
		turning $\ConForms_{\strtensor}^\bullet(\mathcal{M})$
		into a differential graded strong constraint algebra.
	\end{propositionlist}
\end{proposition}

\begin{proof}
	In all cases we only need to show that the involved maps are 
	actually constraint maps.
	For the first part this is clear by the definition of
	$\ConForms^\bullet(\mathcal{M})$.
	For the insertion consider
	$X \in \ConSecinfty(T\mathcal{M})_\Null$.
	Then $\ins_X \alpha \in \ConSecinfty(T\mathcal{M})_\Null$
	for all
	$\alpha \in \ConSecinfty(T^*\mathcal{M})_\Wobs$.
	Since $\ins_X$ is a derivation of the wedge product
	it maps $\ConForms^\bullet(\mathcal{M})_\Wobs$
	to $\ConForms^\bullet(\mathcal{M})_\Null$.
	Now consider $X \in \ConSecinfty(T\mathcal{M})_\Wobs$.
	Then again by the derivation property it is easy to see that
	$\ins_X(\ConForms^\bullet(\mathcal{M})_\Wobs) \subseteq 
	\ConForms^\bullet(\mathcal{M})_\Wobs$
	and
	$\ins_X(\ConForms^\bullet(\mathcal{M})_\Null) \subseteq 
	\ConForms^\bullet(\mathcal{M})_\Null$.
	Thus $\ins$ is a constraint morphism.
	Since the Lie derivative is again a derivation and we know, by
	\autoref{prop:ConVectorFieldsAreConDer}
	and from the fact that
	$\Lie_X Y = [X,Y]$,
	that $\Lie_X$ is a constraint endomorphism of 
	$\ConSecinfty(T\mathcal{M})$,
	it follows that $\Lie$ is a constraint morphism.
	For the de Rham differential we can argue with the 
	formula
	\begin{align*}
		(\D\alpha)(X_0, \dots, X_k)
		&= \sum_{i=0}^{k} (-1)^k \Lie_{X_i}(\alpha(X_0,\dotsc, 
		\overset{i}{\wedge}, \dotsc, X_k))\\
		&\qquad + \sum_{i < j} (-1)^{i+j}
		\alpha([X_i,X_j],X_0,\dotsc, 
		\overset{i}{\wedge}, \dotsc, \overset{j}{\wedge}, \dotsc, 
		X_k),
	\end{align*}
	for $X_0 \tensor \dotsc \tensor X_k \in (\Anti_{\tensor}^\bullet 
	\ConSecinfty(T\mathcal{M}))_\Total$
	to see that $\D$ is a constraint morphism.
	For example, if $\alpha \in \ConForms_{\strtensor}^k(\mathcal{M})_\Null$ is given, we have
	\begin{equation*}
		(\D \alpha)(X_0, \dots, X_k) \in \ConCinfty(\mathcal{M})_\Null
	\end{equation*}
	for all
	$X_0, \dotsc, X_k \in \ConSecinfty(T\mathcal{M})_\Wobs$,
	since from
	\begin{equation*}
		\alpha(X_0, \dotsc, \overset{i}{\wedge},\dotsc, X_k) \in \ConCinfty(\mathcal{M})_\Null
	\end{equation*}
	it follows that
	\begin{equation*}
		\Lie_{X_i} \alpha(X_0, \dotsc, \overset{i}{\wedge},\dotsc, X_k) \in \ConCinfty(\mathcal{M})_\Null
	\end{equation*}
	and from
	$[X_i,X_j] \in \ConSecinfty(T\mathcal{M})_\Null$
	it follows
	\begin{equation*}
		\alpha([X_i,X_j],X_0,\dotsc,\overset{i}{\wedge},\dotsc,\overset{j}{\wedge},\dotsc,X_k) \in \ConCinfty(\mathcal{M})_\Null.
	\end{equation*}
	Thus we have $\D\alpha \in \ConForms_{\strtensor}^{k+1}(\mathcal{M})_\Null$.
	In a similar way we can argue for $\alpha \in \ConForms_{\strtensor}^k(\mathcal{M})_\Wobs$.
\end{proof}

Since $\ins$,  $\Lie$ and $\D$ are completely determined by their $\TOTAL$-components, we immediately get all the usual formulas from the classical Cartan calculus, e.g. Cartan's magic formula
\begin{equation}
	\Lie_X = [\ins_X,\D].
\end{equation}
We cannot expect a similarly well behaved Cartan calculus on
$\ConForms_{\tensor}^\bullet(\mathcal{M})$, since in this case
the de Rham differential is not well-defined, as the next example shows.

\begin{example} \label{ex:ConDRDifferential}
	Consider $\mathcal{M} = (\Reals^{n_\Total}, 
	\Reals^{n_\Wobs},\Reals^{n_\Null})$
	with $n_\Total > n_\Null \geq 1$
	and let $\alpha = x^{n_\Null} \D x^{n_\Total} \in 
	\ConSecinfty(T^*\mathcal{M})_\Null$.
	Then we have
	\begin{equation}
		\D \alpha = \D x^{n_\Null} \wedge \D x^{n_\Total}
		\in \ConSecinfty(T^*\mathcal{M})_\Total \wedge 
		\ConSecinfty(T^*\mathcal{M})_\Null
		\nsubseteq \ConForms_{\tensor}^2(\mathcal{M})_\Null.
	\end{equation}
\end{example}

Let us now turn our attention to constraint multivector fields.
As for constraint differential forms we can define multivector fields using both tensor products available.

\begin{definition}[Constraint multivector fields]
	Let $\mathcal{M} = (M,C,D)$ be a constraint manifold,
	then we denote by
	\begin{align}
		\ConVecFields_{\tensor}^\bullet(\mathcal{M}) 
		\coloneqq \Anti^\bullet_{\tensor} \ConSecinfty(T\mathcal{M})
		\simeq \ConSecinfty(\Anti^\bullet_{\tensor} T\mathcal{M})
		\shortintertext{and}
		\ConVecFields_{\strtensor}^\bullet(\mathcal{M}) 
		\coloneqq \Anti^\bullet_{\strtensor} \ConSecinfty(T\mathcal{M})
		\simeq \ConSecinfty(\Anti^\bullet_{\strtensor} T\mathcal{M})
	\end{align}
	the graded strong constraint modules of
	\emph{constraint multivector fields}
	on $\mathcal{M}$.
\end{definition}

The following example shows that every constraint manifold constructed from a coisotropic submanifold of a Poisson manifold
carries a constraint bivector field in $\ConVecFields_{\strtensor}^2(\mathcal{M})$,
while a Poisson submanifold yields a constraint bivector field in
$\ConVecFields_{\tensor}^2(\mathcal{M})$.

\begin{example}
	\label{ex:ConBivectorsFromPoissonManifolds}
	Let $(M,\pi)$ be a Poisson manifold.
	\begin{examplelist}
		\item \label{ex:ConPoissonMfld}
		If $C \subseteq M$ is a closed coisotropic submanifold allowing for a smooth reduction
		we denote by $\mathcal{M} = (M,C,D)$ the constraint manifold with $D$ the characteristic distribution
		of the coisotropic submanifold $C$.
		Let $n = (n_\Total,n_\Wobs,n_\Null)$ be its constraint dimension.
		Then $\pi \in \Anti^2 \Secinfty(TM)$ is a bivector field, fulfilling 
		$\iota^\#\pi \in \Secinfty(TC \wedge TC + \iota^\#TM \wedge D)
		= \Secinfty((\Anti^2_{\strtensor} T \mathcal{M})_\Wobs)$.
		In an adapted coordinate chart $(U,x)$ around $p \in C$, cf. \autoref{lem:AdaptedLocalFrames},
		we have
		\begin{equation}
			\cc{\iota^\#\pi\at{U \cup C}} = \sum_{i,j = n_\Null +1}^{n_\Wobs} \pi^{ij} \frac{\del}{\del x^i} \wedge \frac{\del}{\del x^j},
		\end{equation}
		and thus for all $\ell = 1, \dotsc n_\Null$
		\begin{align*}
			\nabla_{\frac{\del}{\del x^\ell}} \cc{\iota^\#\pi\at{U \cup C}}
			= \sum_{i,j = n_\Null +1}^{n_\Wobs} \pi^{ij} \left[\frac{\del}{\del x^\ell}, \frac{\del}{\del x^i}\right] \wedge \frac{\del}{\del x^j}
			+ \sum_{i,j = n_\Null +1}^{n_\Wobs} \pi^{ij} \frac{\del}{\del x^i} \wedge \left[\frac{\del}{\del x^\ell},\frac{\del}{\del x^j}\right]
			= 0
		\end{align*}
		holds.
		Here we crucially use that $\pi^{ij} \in \ConCinfty(U)_\Wobs$ for all $i,j = n_\Null + 1, \dotsc, n_\Wobs$.
		Since $D$ is locally spanned by $\frac{\del}{\del x^1},\dotsc, \frac{\del}{\del x^{n_\Null}}$
		we have $\pi \in \ConVecFields_{\strtensor}^2(\mathcal{M})_\Wobs$.
		\item Since every Poisson submanifold is in particular coisotropic, every closed Poisson submanifold
		gives a constraint manifold $\mathcal{M} = (M,C,0)$ the constraint manifold with trivial distribution.
		Let $n = (n_\Total,n_\Wobs,0)$ be its constraint dimension.
		Then $\pi \in \Anti^2\Secinfty(TM)$ restricts to a bivector field
		$\pi\at{C} \in \Anti^2\Secinfty(TC) = \Secinfty((\Anti^2 T\mathcal{M})_\Wobs)$.
		Since $D$ is trivial we thus get $\pi \in 
		\ConVecFields_{\tensor}^2(\mathcal{M})_\Wobs$.
		\item Every closed Poisson submanifold $C$ of a Poisson manifold $M$ can also be equipped with
		another distribution $D$ given by the symplectic leaves of $C$.
		In general, the leaf space will not be smooth, but e.g. for certain types of Poisson manifolds of compact type at 
		least an orbifold structure on the leaf space can be achieved, see 
		\cite{crainic.fernandes.martinez:2019b,crainic.fernandes.martinez:2019a}.
		Note that in the case of a smooth leaf space we obtain a constraint manifold
		$\mathcal{M} = (M,C,D)$ with a constraint Poisson structure $\pi \in \Anti^2\ConVecFields_{\tensor}^2(\mathcal{M})_\Wobs$.
		The reduced space then describes the transversal structure.
	\end{examplelist}
\end{example}

This suggests that a constraint manifold equipped with a constraint bivector field
$\pi \in \ConVecFields_{\strtensor}^2(\mathcal{M})_\Wobs$
fulfilling the Jacobi identity
induces a coisotropic structure on its submanifold.
On the other hand $\pi \in \ConVecFields_{\tensor}^2(\mathcal{M})_\Wobs$
fulfilling the Jacobi identity seems to induce a Poisson structure on 
$C$, which drops to $\mathcal{M}_\red$.
To make this precise we first introduce the Schouten bracket for constraint multivector fields.

\begin{proposition}[Constraint Schouten bracket]
	Let $\mathcal{M} = (M,C,D)$ be a constraint manifold.
	The classical Schouten bracket defines constraint graded Lie algebra structures of degree $-1$
	\begin{align}
		\Schouten{\argument, \argument} 
		\colon \ConVecFields_{\tensor}^{k}(\mathcal{M}) \tensor[\field{k}] \ConVecFields_{\tensor}^{\ell}(\mathcal{M})
		\to \ConVecFields_{\tensor}^{k+\ell - 1}(\mathcal{M})
		\shortintertext{on $\ConVecFields_{\tensor}^\bullet(\mathcal{M})$, and}
		\Schouten{\argument, \argument} 
		\colon \ConVecFields_{\strtensor}^{k}(\mathcal{M}) \tensor[\field{k}] \ConVecFields_{\strtensor}^{\ell}(\mathcal{M})
		\to \ConVecFields_{\strtensor}^{k+\ell - 1}(\mathcal{M})
	\end{align}
	on $\ConVecFields_{\strtensor}^\bullet(\mathcal{M})$.
\end{proposition}

\begin{proof}
	This follows directly from the formula
	\begin{equation*}
		\Schouten{X_0 \wedge \cdots \wedge X_k, Y_0 \wedge \cdots \wedge Y_\ell}
		= \sum_{i=0}^k\sum_{j=0}^\ell (-1)^{i+j} [X_i,Y_j] \wedge X_1 \cdots \overset{i}{\wedge} \cdots X_k \wedge Y_0 \wedge \cdots \overset{j}{\wedge} \cdots \wedge Y_\ell
	\end{equation*}
	and the fact that $[\argument, \argument]$ is a constraint Lie bracket on
	$\ConVecFields^1(\mathcal{M})$.
\end{proof}

It is important to note that even for 
$\ConVecFields_{\strtensor}^\bullet(\mathcal{M})$
we do \emph{not} obtain a strong constraint Lie algebra structure.
One way to see this is to note that $\ConDer(\ConCinfty(\mathcal{M}))$
is only a constraint Lie algebra, even though $\ConCinfty(\mathcal{M})$
is a strong constraint algebra.
Ultimately, this comes from the fact that $\ConHom$ is adjoint
to $\tensor$ and not $\strtensor$.

\begin{corollary}
	\label{cor:ConDGLAOfConMultVectFields}
	Let $\mathcal{M}$ be a constraint manifold.
	Then
	\begin{corollarylist}
		\item $\big(\ConVecFields_{\tensor}^{\bullet+1}(\mathcal{M}), \wedge, \Schouten{\argument, \argument}\big)$
		is a constraint Gerstenhaber algebra.
		\item $\big(\ConVecFields_{\strtensor}^{\bullet+1}(\mathcal{M}), \wedge, \Schouten{\argument, \argument}\big)$
		is a strong constraint Gerstenhaber algebra.
	\end{corollarylist}
\end{corollary}

In contrast to constraint differential forms there is no preferred choice of the tensor products, at least from the point of view of available algebraic structure.
Nevertheless, \autoref{ex:ConBivectorsFromPoissonManifolds} \ref{ex:ConPoissonMfld} shows that if we are interested
in coisotropic submanifolds we are forced to consider
$\ConVecFields_{\strtensor}^2(\mathcal{M})$
instead of $\ConVecFields_{\tensor}^2(\mathcal{M})$.

Both types of constraint forms and constraint multivector fields are well behaved under reduction:

\pagebreak

\begin{proposition}[Constraint forms and multivector fields vs. reduction]
	\label{prop:ConFormsFieldsVSReduction}
	Let $\mathcal{M}$ be a con\-str\-aint
	manifold.
	\begin{propositionlist}
		\item There exists a canonical isomorphism
		$\ConForms_{\tensor}^\bullet(\mathcal{M})_\red
		\simeq \Forms^\bullet(\mathcal{M}_\red)$
		of graded $\ConCinfty(\mathcal{M})_\red$-modules.
		\item There exists a canonical isomorphism
		$\ConForms_{\strtensor}^\bullet(\mathcal{M})_\red
		\simeq \Forms^\bullet(\mathcal{M}_\red)$
		of differential graded algebras.
		\item There exists a canonical isomorphism
		$\ConVecFields_{\tensor}^\bullet(\mathcal{M})_\red \simeq \VecFields^\bullet(\mathcal{M}_\red)$
		of Gerstenhaber algebras.
		\item There exists a canonical isomorphism
		$\ConVecFields_{\strtensor}^\bullet(\mathcal{M})_\red \simeq \VecFields^\bullet(\mathcal{M}_\red)$
		of Gerstenhaber algebras.
	\end{propositionlist}
\end{proposition}

\begin{proof}
	We combine established results to the following chain of canonical isomorphisms:
	\begin{align*}
		\ConForms_{\tensor}^\bullet(\mathcal{M})_\red
		&= \big( \bigoplus_{k=0}^\infty \Anti_{\tensor}^k \ConSecinfty(T^*\mathcal{M}) \big)_\red
		\simeq \bigoplus_{k=0}^\infty (\Anti_{\tensor}^k \ConSecinfty(T^*\mathcal{M}))_\red \\
		&\simeq  \bigoplus_{k=0}^\infty \Anti^k \ConSecinfty(T^*\mathcal{M})_\red
		\simeq \bigoplus_{k=0}^\infty \Anti^k \Secinfty(T^*\mathcal{M}_\red)
		= \Forms^k(\mathcal{M}_\red).
	\end{align*}
	Since we know that the reduction of $\strtensor$ and $\tensor$ agree and the reduced 
	de Rham differential $\D_\red$ fulfils the same local characterization as the de Rham differential on 
	$\mathcal{M}_\red$ the second part follows.
	Similar we have for the constraint multi-vector fields:
	\begin{align*}
		\ConVecFields_{\tensor}^\bullet(\mathcal{M})_\red
		&= \big( \bigoplus_{k=0}^\infty \Anti_{\tensor}^k \ConSecinfty(T\mathcal{M}) \big)_\red
		\simeq \bigoplus_{k=0}^\infty (\Anti_{\tensor}^k \ConSecinfty(T\mathcal{M}))_\red \\
		&\simeq  \bigoplus_{k=0}^\infty \Anti^k \ConSecinfty(T\mathcal{M})_\red
		\simeq \bigoplus_{k=0}^\infty \Anti^k \Secinfty(T\mathcal{M}_\red)
		= \VecFields^k(\mathcal{M}_\red).
	\end{align*}
	Since the defining equation of the Schouten bracket holds for the reduced 
	Schouten bracket we get an isomorphism of Gerstenhaber algebras.
	The last part follows again since $\strtensor$ and $\tensor$ agree after 
	reduction, and the Schouten bracket is given by the same formula.
\end{proof}

\begin{example}\
	\label{ex:ConBivectorsFromPoissonManifoldsReduction}
	\begin{examplelist}
		\item \label{ex:ConBivectorsFromPoissonManifoldsReduction_1} 
		For this let $(M,\pi)$ be a Poisson manifold.
		From \autoref{ex:ConBivectorsFromPoissonManifolds} we know that
		for every closed Poisson submanifold $C \subseteq M$ with smooth leaf space
		we have $\pi \in \ConVecFields_{\strtensor}^2(\mathcal{M})_\Wobs$,
		and clearly $\Schouten{\pi,\pi} = 0$.
		On the other hand, given any constraint manifold
		$\mathcal{M} = (M,C,D)$
		and a bivector field $\pi \in \ConVecFields_{\strtensor}^2(\mathcal{M})_\Wobs$
		with $\Schouten{\pi,\pi} = 0$
		it is easy to see that $C \subseteq M$ is a coisotropic submanifold 
		with characteristic distribution $D_C \subseteq D$
		and $\Cinfty_D(M)$ is closed under the Poisson bracket.
		Then $\pi_\red \in \VecFields^2(\mathcal{M}_\red)$
		satisfies $\Schouten{\pi_\red,\pi_\red} = 0$
		and hence $(\mathcal{M},\pi)$ reduces to a Poisson manifold 
		$(\mathcal{M}_\red, \pi_\red)$. 
		
		\item Let $\mathcal{M} = (M, C, D)$ be a constraint manifold and 
		$\omega \in \ConForms^2_{\strtensor}(\mathcal{M})_\Wobs$ be 
		presymplectic constraint 2-form. 
		This means $\iota^*\omega \in \Secinfty(\iota^*T^*\mathcal{M} \wedge TC^\ann + D^\ann \wedge D^\ann) 
		= \Secinfty((\Anti^2_{\strtensor}T^*\mathcal{M})_\Wobs)$ and an 
		evaluation of 
		$\iota^*\omega$ shows that $(TC)^\perp \subseteq D \subseteq TC$, which 
		makes $C$ a coisotropic submanifold. 
		Then $\omega_\red \in \Forms^2(\mathcal{M}_\red)$ satisfies 
		$\D_\red \omega_\red = 0$, which yields a reduced presymplectic 
		manifold $(M_\red, \omega_\red)$. 
		
		\item Let $\mathcal{M} = (M, C, D)$ be a constraint manifold together 
		with a constraint Poisson tensor 
		$\pi \in \ConVecFields^2_{\strtensor}(\mathcal{M})_\Wobs$ and a 
		constraint Nijenhuis tensor 
		$A \in \ConSecinfty(\ConEnd(T\mathcal{M}))_\Wobs$, see 
		\autoref{ex:ConBivectorsFromPoissonManifoldsReduction} 
		\ref{ex:ConBivectorsFromPoissonManifoldsReduction_1} and 
		\autoref{ex:ReducibleNijenhuis}, respectively. 
		The triple $(M, \pi, A)$ is called Poisson-Nijenhuis manifold, if 
		\begin{equation}
			\pi^\sharp \circ A^* = A \circ \pi^\sharp
		\end{equation}
		and the Schouten invariant $C_{\pi, A}$ defined by 
		\begin{equation}
			C_{\pi, A}(\alpha, X, \beta) \coloneq 
			\beta \left( (\Lie_{\pi^\sharp(\alpha)} A) X\right) 
			- \alpha \left( (\Lie_{\pi^\sharp(\beta)} A) X\right) 
			+ A(X)\pi(\alpha, \beta) 
			- X \left( \pi(A^*\alpha, \beta) \right),
		\end{equation}
		for $\alpha, \beta \in \Secinfty(T^*M)$ and $X \in \Secinfty(TM)$, 
		vanishes. 
		For the given data the first condition descends clearly, since on both 
		sides we have a composition of constraint morphisms, whereas the second 
		condition descends by the constraint calculus established before. 
		This leads to a reduced Poisson tensor 
		$\pi_\red \in \Secinfty(\Anti^2 T\mathcal{M}_\red)$ and a reduced 
		Nijenhuis tensor $A_\red \in \Secinfty(\End(T\mathcal{M}_\red))$ 
		satisfying the conditions of a Poisson-Nijenhuis manifold. 
		This example recovers the results of \cite{vaisman:1996a}.
	\end{examplelist}
\end{example}

\section{Applications of Constraint Reduction}
\label{sec:Applications}

In this section we show how the theory developed so far can be applied to the reduction of some geometric examples. 
In the first subsection, \autoref{sec:ConLieRinehartAlgebras}, we give the 
notion of constraint Lie-Rinehart algebra and study its reduction 
The results obtained there will then be used for the reduction of Lie (bi-)algebroids in \autoref{sec:ConLieAlgebroids}.
This will unify and generalize many well-known constructions of Lie algebroids.
Finally, in \autoref{sec:RedDirac} the reduction of 
Dirac structures is discussed.

\subsection{Reduction of Lie-Rinehart Algebras}
\label{sec:ConLieRinehartAlgebras}

We first state the definition of a constraint Lie-Rinehart algebra,
cf. \cite{rinehart:1963a,huebschmann:1990a} for the classical notion. 

\begin{definition}[Constraint Lie-Rinehart algebra]\
	\label{def:ConLieRinehartAlg}
	\begin{definitionlist}
		\item A \emph{constraint Lie-Rinehart algebra}
		$(\algebra{A},\liealg{g})$
		consists of the following data:
		\begin{definitionlist}[label=\alph*.)]
			\item A commutative constraint $\field{k}$-algebra $\algebra{A}$.
			\item A constraint $\algebra{A}$-module
			$\liealg{g}$ together with a constraint Lie algebra structure
			$[\argument, \argument]$ on $\liealg{g}$.
			\item An action of $\liealg{g}$ on $\algebra{A}$ by derivations, i.e. 
			a constraint morphism
			$\rho \colon \liealg{g} \to \ConDer(\algebra{A})$
			of constraint Lie algebras and constraint
			$\algebra{A}$-modules, called \emph{anchor}.
		\end{definitionlist}
		Additionally, these structures are supposed to be compatible via the \emph{Leibniz rule} 
		\begin{equation}
			[\xi, a \cdot \eta] = \left(\rho(\xi)a\right) \cdot \eta + a[\xi,\eta],
		\end{equation}
		for all $\xi, \eta \in \liealg{g}_{\Total}$ and $a \in 
		\algebra{A}_{\Total}$.
		\item A constraint Lie-Rinehart algebra $(\algebra{A},\liealg{g})$
		is called \emph{strong} if $\algebra{A}$ is a strong constraint algebra and 
		$\liealg{g}$ is a strong constraint $\algebra{A}$-module.
	\end{definitionlist}	
\end{definition}

Note, in the following the conditions in $iii.)$ are summarized by calling 
$\algebra{A}$ a \emph{constraint Lie module} over $\liealg{g}$ with action 
$\rho$. 
In \cite{kern:2023a} there was given a similar but equivalent definition. 

\begin{example}[Constraint Lie-Rinehart algebras]\
	\label{ex:ConLieRinehartAlgebras}
\begin{examplelist}
	\item A constraint Lie-Rinehart algebra of the form $((\algebra{A},\algebra{A},0),(\liealg{g},\liealg{g},0))$
	is the same as a classical Lie-Rinehart algebra $(\algebra{A},\liealg{g})$.
	\item In \cite{alburquerque.etal:2021a} subalgebras of Lie-Rinehart algebras are defined as follows:
	Let $(A,L)$ be a Lie-Rinehart algebra.
	Then a Lie-Rinehart subalgebra of $(A,L)$ is given by a Lie subalgebra $S \subseteq L$ such that $A\cdot S \subseteq S$.
	This is equivalent to a constraint Lie-Rinehart algebra $(\algebra{A},\liealg{g})$ of the form
	\begin{equation}
		\liealg{g} = (L,S,0)
		\qquad\text{and}\qquad
		\algebra{A} = (A,A,0).
	\end{equation}
	\item Again in \cite{alburquerque.etal:2021a} also a notion of ideal of a Lie-Rinehart algebra is defined.
	Such an ideal of a Lie-Rinehart algebra $(A,L)$ is given by a Lie ideal
	$I \subseteq L$
	satisfying $\rho(I)(A)L\subseteq I$.
	This is equivalently given by a constraint Lie-Rinehart algebra
	$(\algebra{A},\liealg{g})$
	of the form 
	\begin{equation}
		\liealg{g} = (L,L,I)
		\qquad\text{and}\qquad
		\algebra{A} = \big(A,A,\SP{\rho(I)(A)}\big),
	\end{equation}
	with $\SP{\rho(I)(A)}$ denoting the ideal generated by the subset
	$\rho(I)(A)\subseteq A$.
	\item A more general version of an ideal in a Lie-Rinehart algebra, a so called weak ideal,
	was introduced in \cite{jotz:2018a}.
	Such a weak ideal in a Lie-Rinehart algebra $(A,L)$ is given by
	a unital subalgebra $B \subseteq A$,
	a Lie subalgebra $\cc{L} \subseteq L$
	and a Lie ideal $J \subseteq \cc{L}$, s.t.
	\begin{cptitem}
		\item $\cc{L}$ is a $B$-module,
		\item $(\cc{L},B)$ is a Lie-Rinehart algebra,
		\item $\rho(J)(B) = 0$ and
		\item $J$ is an $A$-module.
	\end{cptitem}
	This data is equivalently given by a constraint Lie-Rinehart algebra 
	$(\algebra{A},\liealg{g})$ of the form 
	\begin{equation}
		\liealg{g} = (L,\cc{L},J)
		\qquad\text{and}\qquad
		\algebra{A} = (A,B,0).
	\end{equation}
\end{examplelist}
\end{example}

As for classical Lie-Rinehart algebras there exist different notions of structure preserving maps of constraint Lie-Rinehart algebras, see e.g.
\cite{higgins.mackenzie:1993a}.
In the following we will only consider morphisms.

\begin{definition}[Morphism of constraint Lie-Rinehart algebras]\
	\label{def:ConstraintLieRinehartAlgebraMorphism}
	Let $(\algebra{A},\liealg{g})$ and $(\algebra{B},\liealg{h})$ be 
	constraint Lie-Rinehart algebras.
	\begin{definitionlist}
		\item A \emph{morphism 
		$(\phi, \Phi) \colon (\algebra{A},\liealg{g}) \to (\algebra{B},\liealg{h})$
		of constraint Lie-Rinehart algebras}
		consists of a constraint Lie algebra morphism
		$\Phi \colon \liealg{g} \to \liealg{h}$
		and a constraint algebra morphism
		$\phi \colon \algebra{A} \to \algebra{B}$, 
		such that
		\begin{definitionlist}[label=\alph*.)]
			\item $\Phi$ is a morphism of constraint modules along the constraint 
			algebra morphism $\phi \colon \algebra{A} \to \algebra{B}$,
			\item $\phi$ is a morphism of constraint Lie modules along the constraint 
			Lie algebra morphism $\Phi$, i.e.
			\begin{equation}
				\rho_{\liealg{h}}(\Phi(\xi))\phi(a) 
				= \phi \left( \rho_{\liealg{g}}(\xi)a \right)
			\end{equation}
			holds for $\xi \in \liealg{g}$ and $a \in \algebra{A}$.
		\end{definitionlist}
		
		\item The category of constraint Lie-Rinehart algebras and constraint Lie-Rinehart 
		algebra morphisms is denoted by $\injConLieRineAlg$.
		The full subcategory of strong constraint Lie-Rinehart algebras is denoted by
		$\injstrConLieRineAlg$.		
	\end{definitionlist}	
\end{definition}

The reduction of all ingredients used in the definition of constraint Lie-Rinehart algebra have been studied in \autoref{sec:AlgebraicPreliminaries}.
Thus we obtain the following reduction procedure.
Note that similar results were given in \cite{kern:2023a}.

\begin{theorem}[Reduction of Lie-Rinehart algebras]\
	\label{prop:RedConLieRinehart}
	Let $(\algebra{A},\liealg{g})$ and $(\algebra{B}, \liealg{h})$ be constraint 
	Lie-Ri\-ne\-hart algebras. 
	\begin{propositionlist}
		\item \label{prop:RedConLieRinehart_1}
		The pair $(\algebra{A}_\red,\liealg{g}_\red)$ together with the anchor
		$\rho_\red \colon \liealg{g}_\red \to \Der(\algebra{A})_\red \subseteq \Der(\algebra{A}_\red)$
		is a Lie-Rinehart algebra.
		
		\item \label{prop:RedConLieRinehart_2}
		Let
		$(\phi,\Phi) \colon (\algebra{A},\liealg{g}) \to (\algebra{B},\liealg{h})$
		be a morphism of constraint Lie-Rinehart algebras.
		Then the pair of reduced morphisms
		$(\phi_\red, \Phi_\red) \colon (\algebra{A}_\red,\liealg{g}_\red) 
		\to (\algebra{B}_\red,\liealg{h}_\red)$ 
		is a morphism of Lie-Rinehart algebras. 
		
		\item \label{prop:RedConLieRinehart_3}
		Mapping every constraint Lie Rinehart algebra $(\algebra{A},\liealg{g})$
		to $(\algebra{A}_\red,\liealg{g}_\red)$
		and every map $(\phi,\Phi)$ of constraint Lie-Rinehart algebras to $(\phi_\red,\Phi_\red)$
		defines a functor
		\begin{equation}
			\red \colon \injConLieRineAlg \to \LieRineAlg.
		\end{equation}
	\end{propositionlist}
\end{theorem}

\begin{proof}
	The $\field{k}$-algebra structure of 
	$\algebra{A}_\red = \algebra{A}_\Wobs / \algebra{A}_\Null$
	as well as the $\algebra{A}_\red$-module structure of 
	$\liealg{g}_\red = \liealg{g}_\Wobs / \liealg{g}_\Null$
	follow from \autoref{sec:ConAlgebrasAndModules},
	while the Lie algebra structure on
	$\liealg{g}_\red$
	was discussed in \autoref{sec:ConLieAlgebras}.
	There it was also shown that the constraint $\algebra{A}$-module and constraint Lie algebra morphism $\rho$ 
	descends to 
	$\rho_\red \colon \liealg{g}_\red \to \Der(\algebra{A})_\red$.
	Since $\Der(\algebra{A})_\red \subseteq \Der(\algebra{A}_\red)$,
	see \eqref{eq:InsertionReductionDerivations},
	this is equivalent to a reduced Lie module structure on
	$\algebra{A}_\red$. 
	Together the reduced structures give rise to a Lie-Rinehart algebra 
	$(\algebra{A}_\red, \liealg{g}_\red)$, since the Leibniz rule can be easily shown to descends a well. 
	
	The morphisms $\phi$ and $\Phi$ of constraint $\field{k}$-algebras as well as 
	of constraint Lie algebras and constraint modules over $\field{k}$-algebras, 
	respectively, give rise to reduced morphisms $\phi_\red$ and $\Phi_\red$. 
	By using the quotient projections relating the reduced with the original 
	structures it follows that the pair $(\phi_\red, \Phi_\red)$ forms a reduced 
	morphism of Lie modules, which makes it a Lie-Rinehart algebra morphism. 
	 
	 The third part follows by the two previous ones and the verifications for 
	 $\red$ are straight-forward.
\end{proof}

\subsection{Reduction of Lie (Bi-)Algebroids}
\label{sec:ConLieAlgebroids}

The combination of constraint vector bundles, as introduced in \autoref{sec:ConVectorBundles},
with constraint Lie-Rinehart algebras enables us now to define constraint Lie algebroids in the 
following way. 

\begin{definition}[Constraint Lie algebroid]
	\label{def:ConLieAlgebroid}
	A \emph{constraint Lie algebroid}
	consists of
	\begin{definitionlist}
		\item a constraint vector bundle
		$A = (A_\Total, A_\Wobs, A_\Null,\nabla)$
		over a constraint manifold
		$\mathcal{M} = (M,C,D)$ and 
		
		\item a constraint Lie-Rinehart algebra structure on 
		$(\ConCinfty(\mathcal{M}), \ConSecinfty(A))$. 
	\end{definitionlist}
\end{definition}

In \cite{kern:2023a} a more hands-on definition was given. 
For a definition of classical Lie algebroids, see \cite{mackenzie:2005a}. 
The link between classical and constraint Lie algebroids becomes more 
transparent using the following reformulation.

\begin{lemma}
	Let $A = (A_\Total, A_\Wobs, A_\Null, \nabla)$ be a constraint vector bundle  
	over $\mathcal{M} = (M, C, D)$. 
	Then a constraint Lie algebroid structure on $A$ is equivalent to
	\begin{lemmalist}
		\item a constraint vector bundle morphism $\rho \colon A \to T\mathcal{M}$ 
		over $\id_{\mathcal{M}}$ and 
		
		\item a constraint Lie algebra on $\ConSecinfty(A)$ with Lie bracket 
		$[\argument, \argument]$,
	\end{lemmalist}
	such that 
	\begin{equation}
		[a, fb] = (\rho(a)f)b + f[a, b]
	\end{equation}
	holds for $a, b \in \ConSecinfty(A)_\Total$ and 
	$f \in \ConCinfty(\mathcal{M})_\Total$.
\end{lemma}

With this it follows that for any constraint Lie algebroid $A$ 
the $\TOTAL$-component $A_\Total$ is a classical Lie 
algebroid.
This is not true for the $\WOBS$-component:
Even though $\rho_\Wobs \colon A_\Wobs \to TC$ 
yields a candidate for the anchor
there is no canonical way to define a Lie bracket
on $\Secinfty(A_\Wobs)$.
To see this, recall that $\ConSecinfty(A)_\Wobs$ consists of certain
sections defined an all of $M$.
Nevertheless, constraint Lie algebroids reduce to Lie algebroids on the 
reduced manifold as we will see later. 

\begin{lemma}
	\label{lem:UnderlyingConstraintLRAlgOfConstraintLAlgd}
	Let $A = (A_\Total, A_\Wobs, A_\Null, \nabla)$ be a constraint Lie algebroid 
	over $\mathcal{M} = (M, C, D)$, then the constraint Lie-Rinehart algebra 
	$(\ConCinfty(\mathcal{M}), \ConSecinfty(A))$ is strong. 
\end{lemma}
\begin{proof} 
	Since $\ConCinfty(\mathcal{M})_\Null$ is by definition the vanishing ideal of $C$, it is a two-sided ideal in $\ConCinfty(\mathcal{M})_\Total$, 
	which makes $\ConCinfty(\mathcal{M})$ a strong constraint $\field{k}$-algebra. 
	Moreover, for $f \in \ConCinfty(\mathcal{M})_\Null$ and 
	$s \in \ConSecinfty(A)_\Total$ it follows that $fs \at{C} = 0$, 
	i.e. $fs \in \ConSecinfty(A)_\Null$. 
	For $f \in \ConCinfty(\mathcal{M})_\Total$ and $s \in \ConSecinfty(A)_\Null$ 
	we obtain $fs \at{C} \in \Secinfty(A_\Null)$, i.e. 
	$fs \in \ConSecinfty(A)_\Null$. 
	 Consequently, $\ConSecinfty(A)$ is a strong constraint 
	 $\ConCinfty(\mathcal{M})$-module and thus 
	 $(\ConCinfty(\mathcal{M}), \ConSecinfty(A))$ is strong as constraint 
	 Lie-Rinehart algebra. 
\end{proof}

We can now define morphisms of constraint Lie algebroids.
The classical notion of Lie algebroid morphism can be found in 
\cite{higgins.mackenzie:1990a,higgins.mackenzie:1993a}. 

\begin{definition}[Constraint Lie algebroid morphism]\
	Let $A \to \mathcal{M}$ and $B \to \mathcal{N}$ be constraint Lie 
	algebroids.
	\begin{definitionlist}
		\item A constraint vector bundle morphism $\Phi \colon A \to B$ 
		over $\phi \colon \mathcal{M} \to \mathcal{N}$
		is called 
		\emph{constraint Lie algebroid morphism},
		if the corresponding morphisms
		$\Phi \colon \ConSecinfty(A) \to \ConSecinfty(\phi^\#B)$
		and 
		$\phi^* \colon \ConCinfty(\mathcal{N}) \to \ConCinfty(\mathcal{M})$
		satisfy the following two conditions:
		\begin{definitionlist}[label=\alph*.)]
			\item The equation of constraint morphisms
			\begin{equation} \label{eq:AnchorLieAlgbdComorphism}
				T\phi \circ \rho_A = \rho_B \circ \Phi
			\end{equation} 
			holds.		
			\item For $\widehat{\Phi}(a) = \sum_{i}f^i \phi^\# a_i$
			and
			$\widehat{\Phi}(b) = \sum_{j} g^j \phi^\# b_j \in 
			\Con\Secinfty(\phi^\# B)_{\Total}$, one has 
			\begin{equation}
			\label{eq:LieBracketConditionOfLieAlgebroidMorphism}
				\widehat{\Phi} \left(\left[a, b\right]_A\right) 
				= \sum_{i,j} f^i g^j \phi^\#\left[a_i, b_j\right]_B 
				+ \sum_{j} \rho_A(a)g^j \phi^\#b_j 
				- \sum_{i} \rho_A(b)f^i \phi^\#a_i 
			\end{equation}
			for $a, b \in \Con\Secinfty(A)_{\Total}$,
			$a_i, b_j \in \Con\Secinfty(B)_{\Total}$
			and $f^i, g^j \in \Con\Cinfty(\mathcal{M})_{\Total}$,
			where 
			$\widehat{\Phi} \colon A \to \phi^\# B$
			is the induced morphism between constraint vector bundles over 
			$\id_{\mathcal{M}}$, see 
			\eqref{eq:universalVBMorphismToPullBackVB}. 
		\end{definitionlist}
	
	\item The category of constraint Lie algebroids with morphisms of constraint 
	Lie algebroids is denoted by $\injConLieAlgd$. 
	\end{definitionlist}
\end{definition}

Additionally, we can use the notion of morphism of constraint Lie-Rinehart algebras as introduced in \autoref{sec:ConLieRinehartAlgebras}
to define comorphisms of constraint Lie algebroids.
The definition of classical Lie algebroid comorphisms can be found in 
\cite{cattaneo.dherin.weinstein:2013a}.

\begin{definition}[Constraint Lie algebroid comorphism]\
	Let $A \to \mathcal{M}$ and $B \to \mathcal{N}$ be constraint Lie 
	algebroids. 
	\begin{definitionlist}
		\item A constraint vector bundle morphism $\Phi \colon B^* \to A^*$ 
		along $\phi \colon \mathcal{N} \to \mathcal{M}$
		is called a \emph{constraint Lie algebroid comorphism}, if the corresponding 
		morphism between constraint modules of sections 
		$\Phi^* \colon \Con\Secinfty(A) \to \Con\Secinfty(B)$ over 
		$\phi^* \colon \Con\Cinfty(\mathcal{M}) \to \Con\Cinfty(\mathcal{N})$ 
		forms a Lie-Rinehart algebra morphism. 
	
		\item The category of constraint Lie algebroids with comorphism of 
		constraint Lie algebroids is denoted by $\injConLieAlgd_{\mathrm{co}}$. 
	\end{definitionlist}
\end{definition}

The reduction of constraint Lie algebroids becomes functorial for both types 
of morphisms. 
The results of the following theorem coincide with those in \cite{kern:2023a}.

\begin{theorem}[Reduction of constraint Lie algebroids]\
	\label{thm:ReductionOfLieAlgebroids}
	\begin{theoremlist}
		\item \label{thm:ReductionOfLieAlgebroids_1}
		Let $A$ be a constraint Lie algebroid over $\mathcal{M}$
		with anchor $\rho$
		and Lie bracket $[\argument,\argument]$.
		Then $A_\red$ is a Lie algebroid over $\mathcal{M}_\red$ with anchor $\rho_\red$
		and Lie bracket $[\argument,\argument]_\red$.
		
		\item \label{thm:ReductionOfLieAlgebroids_2}
		Let $\Phi \colon A \to B$ be a Lie algebroid morphism over 
		$\phi \colon \mathcal{M} \to \mathcal{N}$. 
		Then $\Phi_\red \colon A_\red \to B_\red$ is a morphism of Lie algebroids 
		over $\phi_\red \colon \mathcal{M}_\red \to \mathcal{N}_\red$. 
		
		\item \label{thm:ReductionOfLieAlgebroids_3}
		Mapping every constraint Lie algebroid $A$ over $\mathcal{M}$ to 
		$A_\red$ over $\mathcal{M}_\red$ and every morphism $\Phi$ of constraint 
		Lie algebroids to $\Phi_\red$ defines a functor 
		\begin{equation}
			\red \colon \injConLieAlgd \to \LieAlgd. 
		\end{equation}
		
		\item \label{thm:ReductionOfLieAlgebroids_4}
		Let $\Phi \colon B^* \to A^*$ be a Lie algebroid comorphism over 
		$\phi \colon \mathcal{N} \to \mathcal{M}$.
		Then $\Phi_\red \colon B^*_\red \to A^*_\red$ is a Lie algebroid 
		comorphism over $\phi \colon \mathcal{N}_\red \to \mathcal{M}_\red$. 
		 
		\item \label{thm:ReductionOfLieAlgebroids_5}
		Mapping every constraint Lie algebroid $A$ over $\mathcal{M}$ to 
		$A_\red$ over $\mathcal{M}_\red$ and every comorphism $\Phi$ of constraint 
		Lie algebroids to $\Phi_\red$ defines a functor 
		\begin{equation}
			\red \colon \injConLieAlgd_{\mathrm{co}} \to \LieAlgd_{\mathrm{co}}. 
		\end{equation}
	\end{theoremlist}
\end{theorem}
\begin{proof}
	\ref{thm:ReductionOfLieAlgebroids_1}: 
	We only need to collect the necessary results:
	Every constraint vector bundle $A \to \mathcal{M}$
	reduces to a vector bundle $A_\red \to \mathcal{M}_\red$
	by \autoref{prop:ReductionOfConVect}.
	Moreover, the constraint tangent bundle $T\mathcal{M}$ reduces to the tangent bundle
	$T\mathcal{M}_\red$ by \autoref{prop:ReductionOfTangentBundle}
	and by \autoref{prop:RedFunctorVect} the vector bundle morphism $\rho$
	reduces to a vector bundle morphism $\rho_\red \colon A_\red \to T\mathcal{M}_\red$.
	Finally, the constraint module $\ConSecinfty(A)$ reduces to $\Secinfty(A_\red)$
	by \autoref{prop:ConSecVSReduction} and the constraint functions $\ConCinfty(\mathcal{M})$
	reduce to $\Cinfty(\mathcal{M}_\red)$ according to \autoref{prop:ConFunctionsVSReduction}.
	Together with \autoref{prop:RedConLieRinehart} the Leibniz rule descends as 
	well, showing that we obtain a reduced Lie algebroid. 
	
	\ref{thm:ReductionOfLieAlgebroids_2}:
	In order to prove this assertion note that by \autoref{prop:RedFunctorVect} 
	the reduction functor maps $\Phi$ to a reduced vector bundle morphism 
	$\Phi_\red$ over $\phi_\red$.
	By the same proposition it follows that condition \eqref{eq:AnchorLieAlgbdComorphism} descends to the reduction.
	It remains to show that the reduced vector bundle morphism also satisfies 
	condition \eqref{eq:LieBracketConditionOfLieAlgebroidMorphism}. 
	The  constraint vector bundle morphism $\Phi$ leads to a constraint vector 
	bundle morphism $\widehat{\Phi} \colon E \to \phi^\# F$ over 
	$\id_{\mathcal{M}}$, which by \autoref{thm:strConSerreSwan} becomes a 
	morphism between the constraint modules of sections. 
	Since taking the pull-back vector bundle commutes with reduction, see 
	\autoref{prop:RedConstructionsConVect} \ref{prop:RedConstructionsConVect_6}, 
	and taking sections commutes with reduction as well, see 
	\autoref{prop:ConSecVSReduction}, it follows that reduction commutes with 
	taking sections of a pull-back vector bundle. 
	Consequently, for $a \in \ConSecinfty(A)_\Wobs$, 
	$f^i \in \ConCinfty(\mathcal{M})_\Wobs$ and 
	$a_i \in \ConSecinfty(B)_\Wobs$, which are related by 
	$\widehat{\Phi}(a) = \sum_i f^i \phi^\# a_i$, we obtain 
	\begin{align*}
		\widehat{ \left( \Phi_\red \right)}(a_\red) 
		= \widehat{\Phi}(a)_\red 
		= \Big( \sum_i f^i \phi^\# a_i \Big)_\red 
		= \sum_i f^i_\red \phi^\#_{\red} (a_{i})_\red, 
	\end{align*}
	where we used the coincidence of $\widehat{(\Phi_\red)}$ with 
	$(\widehat{\Phi})_\red$ up 	to the natural isomorphism in 
	\autoref{prop:RedConstructionsConVect} \ref{prop:RedConstructionsConVect_6}. 
	Let $b \in \ConSecinfty(A)_\Wobs$, $g_j \in \ConCinfty(\mathcal{M})_\Wobs$ 
	and $b_j \in \ConSecinfty(B)_\Wobs$, which are related by 
	$\widehat{\Phi}(b) = \sum_j g_j \phi^\# b_j$, we obtain 
	\begin{align*}
		\widehat{\left( \Phi_\red \right) }(\left[ a_\red, b_\red \right]_{A, \red}) 
		&= \widehat{\Phi}\left( \left[ a, b \right]_A\right)_\red \\
		&= \sum_{i, j} f^i_\red g^j_\red \phi^\#_\red \left[ a_{i, \red}, b_{j, \red} \right]_{B, \red} 
		- \sum_i \rho_{A, \red}(b_\red) f^i_\red \phi^\#_\red a_{i, \red} \\
		& \qquad + \sum_j \rho_{A, \red}(a_\red) g^j_\red \phi^\#_\red b_{j, \red} 
	\end{align*}
	by using \eqref{eq:LieBracketConditionOfLieAlgebroidMorphism} and 
	exploiting the equation shown before. 
	This proves the condition on the Lie brackets and makes the vector bundle 
	morphism $\Phi_\red \colon A_\red \to B_\red$ over 
	$\phi_\red \colon \mathcal{M}_\red \to \mathcal{N}_\red$ a Lie algebroid 
	morphism. 
	
	\ref{thm:ReductionOfLieAlgebroids_3}:
	This follows by the first and the second assertion and 
	straight-forward computations show that $\red$ constitutes a covariant 
	functor. 
	
	\ref{thm:ReductionOfLieAlgebroids_4}:
	This follows directly from the reduction of morphism of Lie-Rinehart algebras, see \autoref{prop:RedConLieRinehart}, together with the fact that taking sections of the dual is functorial and commutes with reduction, see \autoref{prop:SectionsOverGeneralBase}.
	
	\ref{thm:ReductionOfLieAlgebroids_5}: 
	This is a combination of the first and the fourth one as well as 
	a straight-forward computation to verify that $\red$ yields a covariant 
	functor. 
\end{proof}

\begin{example}[Constraint Lie algebroids]\
	\begin{examplelist}
		\item A constraint Lie algebroid with trivial $\NULL$-components, i.e.
		a constraint Lie algebroid of the form
		$A = (A_\Total,A_\Wobs,0,0)$ over $\mathcal{M} = (M,C,0)$,
		corresponds to a Lie subalgebroid on the submanifold $C$.
		The reduction is given by this Lie algebroid
		$A_\Wobs$ over $C$.
		\item The classical extreme cases for Lie algebroids also hold in the constraint setting:
		Over a point
		$\mathcal{M} = (\{\pt\},\{\pt\},0)$
		a constraint Lie algebroid is exactly a constraint Lie algebra,
		see \autoref{sec:ConLieAlgebras}.
		Moreover, for any constraint manifold $\mathcal{M}$
		the constraint tangent bundle $T\mathcal{M}$
		is a constraint Lie algebroid with $\rho = \id$.	
		
		\item Let $C \subseteq M$ be a closed coisotropic submanifold with simple 
		characteristic distribution $D \subseteq TC$ and let $\mathcal{M} = (M,C,D)$ 
		be the corresponding constraint manifold.
		It is well-known that $T^*M$, equipped with the anchor
		$\pi^\# \colon T^*M \to TM$, given by 
		$\pi^\# (\alpha) = \pi(\alpha, \argument)$ for $\alpha \in \Secinfty(T^*M)$, 
		and the Lie bracket defined by 
		\begin{equation}
			\label{eq:LieBracketPoissonLieAlgebroid}
			[\alpha,\beta] \coloneqq \Lie_{\alpha^\#}\beta - \Lie_{\beta^\#}\alpha - \D(\pi(\alpha,\beta))
		\end{equation}
		for all $\alpha,\beta \in \Secinfty(T^*M)$,
		is a Lie algebroid.
		One could check by hand that all these structures actually define a constraint Lie algebroid on $T^*\mathcal{M}$.
		More conceptually we can also argue as follows:
		By \autoref{ex:ConBivectorsFromPoissonManifolds} \ref{ex:ConPoissonMfld} we know that 
		\begin{equation}
			\pi \in \ConSecinfty(\Anti_{\strtensor}^2T\mathcal{M})_\Wobs
			\subseteq \ConSecinfty(T\mathcal{M} \strtensor T\mathcal{M})_\Wobs
			\simeq \ConSecinfty(\ConHom(T^*\mathcal{M},T\mathcal{M}))_\Wobs,
		\end{equation}
		see also \autoref{prop:DualTensorHomIsos}.
		This shows that $\pi^\# \colon T^*\mathcal{M} \to T\mathcal{M}$
		is a a constraint morphism.
		Moreover, by \autoref{prop:ConCartanCalculus} we know that the Lie derivative and the de Rham differential are constraint morphisms, and hence \eqref{eq:LieBracketPoissonLieAlgebroid}
		is a constraint morphism.
		The fact that we have a classical Lie algebroid on the $\TOTAL$-component then ensures that all required compatibilities hold, and we thus get a constraint Lie algebroid on $T^*\mathcal{M}$.
		This constraint Lie algebroid then reduces to the classical Lie algebroid structure
		on $T^*\mathcal{M}_\red$ induced by the reduced Poisson structure on $\mathcal{M}_\red$.
		
		\item Infinitesimal ideal systems were introduced in \cite{jotzlean.ortiz:2014a}
		as the infinitesimal counterpart to foliated Lie groupoids, see also 
		\cite{carinena.nunesdacosta.santos:2005a,hawkins:2008a}.
		Given a classical Lie algebroid $A \to M$ with anchor $\rho$ and Lie bracket
		$[\argument,\argument]$ an infinitesimal ideal system is given by 
		an involutive distribution $F_M \subseteq TM$,
		a subalgebroid $K \subseteq A$, such that $\rho(K) \subseteq F_M$,
		and a flat $F_M$-connection $\nabla$ on $A/K$ such that the following properties are fulfilled:
		\begin{examplelist}[label=\textit{\alph*.)}]
			\item $[a,b] \in \Secinfty(K)$, for all $a \in \Secinfty(A)$, $b \in \Secinfty(K)$ with $\nabla a = 0$,
			\item $\nabla[a,b] = 0$, for all $a,b \in \Secinfty(A)$ with $\nabla a = 0 = \nabla b$,
			\item $\nabla^\Bott \rho(a) = 0$, for all $a \in \Secinfty(A)$
			with $\nabla a = 0$, where $\nabla^\Bott$ is the Bott connection associated to $F_M$.
		\end{examplelist}
		If $F_M$ is simple this corresponds exactly to a constraint Lie algebroid structure on the
		constraint vector bundle $(A,A,K,\nabla)$ over the constraint manifold $(M,M,F_M)$.
		Thus from the constraint point of view it becomes clear that infinitesimal 
		ideal systems capture the correct structure to reduce a Lie algebroid along 
		a distribution. 
	\end{examplelist}
\end{example}

Recall that in the the classical situation,
given a vector bundle $A$ over a manifold $M$,
there is a bijective correspondence 
\begin{equation}
	\left\{\text{Lie algebroid } \rho \colon A \to TM \right\} 
	\leftrightarrows
	\left\{\text{Gerstenhaber algebra } \Secinfty(\Anti^\bullet A) \right\}, 
\end{equation}
relating Lie algebroid structures on $A$ with Gerstenhaber algebra structures on 
$\Secinfty(\Anti^\bullet A)$.
See \cite{kosmannschwarzbach:1995a,xu:1999a} and the references therein for 
more details.
There is a second correspondence
\begin{equation}
	\left\{ \text{Lie algebroid } \rho \colon A \to TM \right\} 
	\leftrightarrows \left\{ \text{differential graded algebra } 
	\Secinfty(\Anti^\bullet A^*) \right\}, 
\end{equation}
characterizing Lie algebroid structures as differential graded algebra structures
on $\Secinfty(\Anti^\bullet A^*)$,
see \cite{kosmannschwarzbach:1995a,xu:1999a}.
In the constraint setting we have to carefully choose the tensor products used in both characterizations.

\begin{theorem}[Characterization of Lie algebroids]
	\label{thm:ConstraintCartanCalculus}
	Let $A = (A_\Total, A_\Wobs, A_\Null, \nabla)$
	be a constraint vector bundle over a constraint manifold
	$\mathcal{M} = (M,C,D)$.
	Then the following structures are equivalent:
	\begin{theoremlist}
		\item A constraint Lie algebroid structure
		$[\argument, \argument]$ on $A$ with anchor
		$\rho \colon A \to T\mathcal{M}$.
		
		\item A strong constraint Gerstenhaber structure
		$\GerstBracket{\argument, \argument}$
		on
		$\ConSecinfty(\Anti_{\strtensor}^\bullet A)$
		with respect to the associative product $\wedge$.

		\item A constraint Gerstenhaber structure
		$\GerstBracket{\argument, \argument}$
		on
		$\ConSecinfty(\Anti_{\tensor}^\bullet A)$
		with respect to the associative product $\wedge$.
		
		\item A differential graded strong constraint algebra structure
		on $\ConSecinfty(\Anti_{\strtensor}^\bullet A^*)$
		with differential $\D_A$ of degree $+1$ with respect to $\wedge$.
	\end{theoremlist}
	The relation between these structures are given as follows:
	\begin{cptitem}
		\item The Gerstenhaber bracket
		$\GerstBracket{\argument,\argument}$
		on
		$\ConSecinfty(\Anti_{\tensor}^\bullet A)$
		and
		$\ConSecinfty(\Anti_{\strtensor}^\bullet A)$
		is uniquely determined by
		\begin{equation} \label{eq:GerstBracketGenerators}
			\GerstBracket{f,g} = 0, \qquad
			\GerstBracket{a,f} = \rho(a)f = -\GerstBracket{f,a}, 
			\qquad\text{ and }\qquad			
			\GerstBracket{a,b} = [a,b]
		\end{equation}
		for $a,b \in \Secinfty(A_\Total)$ and
		$f,g \in \Cinfty(M)$.
		\item The differential $\D$ on 
		$\ConSecinfty(\Anti_{\strtensor}^\bullet A^*)$ is given by 
		\begin{equation} \label{eq:DifferentialByKoszulFormula}
		\begin{split}
			(\D\alpha)(a_0,\dotsc, a_{k})
			&= \sum_{i=0}^{k} (-1)^{i}\rho(a_i) \alpha(a_0,\dotsc, \overset{i}{\wedge}, \dotsc, \alpha_{k}) \\
			&\qquad + \sum_{i<j} (-1)^{i+j} \alpha([a_i,a_j],a_0,\dotsc,\overset{i}{\wedge},\dotsc,\overset{j}{\wedge},\dotsc, a_{k})
		\end{split}
		\end{equation}
		for $\alpha \in \Secinfty(\Anti^{k}A^*_\Total)$
		and
		$a_0, \dotsc, a_{k} \in \Secinfty(A_\Total)$.
	\end{cptitem}
\end{theorem}

\begin{proof}
	Suppose $A$ is a constraint Lie algebroid with anchor $\rho$ and bracket $[\argument, \argument]$.
	The exterior algebra
	$\ConSecinfty(\Anti^\bullet_{\strtensor}A)
	\simeq \Anti^\bullet_{\strtensor}\ConSecinfty(A)$ forms a 
	graded strong constraint algebra with respect to the associative 
	product $\wedge$ induced by $\strtensor$, see \autoref{ex:gradedConAlgs}. 
	Since $A_\Total$ is a classical Lie algebroid it
	corresponds to a classical Gerstenhaber structure on
	$\Anti^\bullet A_\Total$, which is given by \eqref{eq:GerstBracketGenerators}.
	Thus it only remains to show that the Gerstenhaber bracket
	$\GerstBracket{\argument,\argument}$
	is a constraint Lie bracket.
	For this note that \eqref{eq:GerstBracketGenerators}
	implies that the Gerstenhaber bracket is given by
	\begin{equation*} \label{eq:GerstBracketOnGeneratorsOfExteriorAlgebra}
		\GerstBracket{a, b}
		= \sum_{i = 0}^{k}\sum_{j = 0}^{\ell} (-1)^{i + j} 
		[a_i, b_j] \wedge a_0 \wedge \dotsc \overset{i}{\wedge} \dotsc 
		\wedge a_k \wedge b_0 \wedge \dotsc \overset{j}{\wedge} 
		\dotsc \wedge b_\ell,
		\tag{$*$}
	\end{equation*}
	for elementary tensors
	$a = a_0 \wedge \dotsc \wedge a_k \in \Secinfty(\Anti^{k+1} A_\Total)$
	and 
	$b = b_0 \wedge \dotsc \wedge b_\ell \in \Secinfty(\Anti^{\ell+1} A_\Total)$.
	see \cite{huebschmann:1998a}. 
	Now let e.g.
	$a = a_0 \wedge \dotsc \wedge a_k \in \ConSecinfty(\Anti^{k+1}_{\strtensor}A)_\Null$ 
	and
	$b = b_0 \wedge \dotsc \wedge b_\ell \in \ConSecinfty(\Anti^{\ell+1}_{\strtensor}A)_\Wobs$ 
	be elementary tensors, then $a$ contains at least one factor in 
	$\ConSecinfty(A)_\Null$,
	and thus all summands of \eqref{eq:GerstBracketOnGeneratorsOfExteriorAlgebra}
	end up in $\ConSecinfty(\Anti_{\strtensor}^\bullet A)_\Null$.
	The other cases can be checked equivalently.
	With this we obtain a strong constraint Gerstenhaber structure on
	$\Anti_{\strtensor}^\bullet\ConSecinfty(A)$.
	Completely analogous considerations show that \eqref{eq:GerstBracketOnGeneratorsOfExteriorAlgebra}
	also defines a constraint Gerstenhaber structure on
	$\Anti_{\tensor}^\bullet\ConSecinfty(A)$.
	
	Conversely, assume that 
	$\Anti_{\tensor}^\bullet\ConSecinfty(A)$
	carries a constraint Gerstenhaber structure.
	Using \eqref{eq:GerstBracketGenerators}
	we see that the Lie bracket $[\argument, \argument]$
	and the anchor $\rho$ are constraint maps.
	Moreover, since we already now that
	$A_\Total$
	is a classical Lie algebroid
	this is enough to obtain a constraint Lie algebroid structure on $A$.

	Moreover, every strong constraint Gerstenhaber algebra structure on
	$\Anti_{\strtensor}^\bullet\ConSecinfty(A)$
	restricts to a constraint Gerstenhaber algebra on 
	$\Anti_{\tensor}^\bullet\ConSecinfty(A)$,
	and thus induces a constraint Lie algebroid structure on $A$ as before.

	Let us now show the equivalence with differential graded algebras.
	For this suppose again that $A$ is a constraint Lie algebroid.

	It is well-known that the classical Lie algebroid structure on $A_\Total$
	corresponds to a classical differential graded algebra structure on 
	$\Anti^\bullet \Secinfty(A_\Total)^*$ with differential
	given by \eqref{eq:DifferentialByKoszulFormula}.
	Thus it only remains to show that 
	$\D \colon \Anti^\bullet_{\strtensor}\ConSecinfty(A^*) \to \Anti^{\bullet}_{\strtensor}\ConSecinfty(A^*)$ 
	is a constraint morphism.
	This can be done by evaluating \eqref{eq:DifferentialByKoszulFormula} on
	$a_0 \wedge \dotsc \wedge a_{k} \in \ConSecinfty(\Anti_{\tensor}^{k+1}A)_\Wobs$
	and
	$a_0 \wedge \dotsc \wedge a_{k} \in \ConSecinfty(\Anti_{\tensor}^{k+1}A)_\Null$.

	Conversely, starting with a differential graded strong constraint algebra structure on
	$\Anti^\bullet_{\strtensor}\ConSecinfty(A^*)$ we know that
	$(\algebra{A}_\Total, \liealg{g}_\Total)$
	is a classical Lie algebroid. 
	It remains to show that $\rho$ and $[\argument, \argument]$ are constraint morphisms. 
	Evaluating $\D$ on degree $0$ and $1$ elements leads to 
	\begin{equation*} \label{eq:DifferentialOnDegree0Generator} \tag{$**$}
		\rho(a)f = \D f (a)
	\end{equation*}
	for $f \in \Cinfty(M)$ and 
	\begin{equation*} \label{eq:DifferentialOnDegree1Generator} \tag{${**}*$}
		\eta([a, b])
		= \D (\eta (b)) (a) 
		- \D (\eta(a))(b) - (\D\eta)(a, b) 
	\end{equation*}
	for $\eta \in \Secinfty(A^*_\Total)$ and
	$a, b \in \Secinfty(A_\Total)$, see \cite{xu:1999a}. 
	Since the anchor is given by \eqref{eq:DifferentialOnDegree0Generator}
	a straightforward check shows that $\rho$ is indeed constraint.
	To show that the Lie bracket given by \eqref{eq:DifferentialOnDegree1Generator}
	is constraint we use the fact that
	$(A^*)^* \simeq A$.
	Then we need to show that for all
	$a \tensor b \in (\ConSecinfty(A) \tensor \ConSecinfty(A))_\Wobs$
	we have $\eta([a,b]) \in \ConSecinfty(A)_\Wobs$
	for $\eta \in \ConSecinfty(A^*)_\Wobs$
	and $\eta([a,b]) \in \ConSecinfty(A)_\Null$
	for $\eta \in \ConSecinfty(A^*)_\Null$.
	Moreover, for all $a \tensor b \in (\ConSecinfty(A) \tensor \ConSecinfty(A))_\Null$
	we need to have $\eta([a,b]) \in \ConSecinfty(A)_\Null$
	for $\eta \in \ConSecinfty(A^*)_\Wobs$.
	This follows directly by checking the right-hand side of
	\eqref{eq:DifferentialOnDegree1Generator}
	and using that $\D$ is constraint.
	Thus $A$ becomes a constraint Lie algebroid.
\end{proof}

Surprisingly, the non-strong Gerstenhaber structure on
$\Anti^\bullet_{\tensor}\ConSecinfty(A)$
is enough to obtain a Lie algebroid structure on $A$.
This comes from the fact that both the Gerstenhaber structure on
$\Anti^\bullet_{\tensor}\ConSecinfty(A)$
and
$\Anti^\bullet_{\strtensor}\ConSecinfty(A)$
are generated by degrees $0$ and $1$.
On the other hand, the constraint differential graded algebra 
structure needs to be strong.
On $\Anti^\bullet_{\tensor}\ConSecinfty(A^*)$
the differential $\D$ would not be a constraint morphism, cf. \autoref{ex:ConDRDifferential}.

\begin{remark} \label{rem:CoincidenceOfReductionOfLRAlgsWithItsCharacterizations}
	Since for modules of sections the tensor products $\tensor$ and $\strtensor$
	are both compatible with reduction, see \autoref{prop:RedConstructionsConVect},
	it is clear that the Gerstenhaber structures on
	$\ConSecinfty(\Anti^\bullet_{\strtensor} A)$
	and $\ConSecinfty(\Anti^\bullet_{\tensor} A)$
	both reduce to a Gerstenhaber structure on $\Secinfty(\Anti^\bullet A_\red)$
	corresponding to the reduced Lie algebroid $A_\red$.
	Similarly, the differential graded algebra structure on
	$\ConSecinfty(\Anti^\bullet_{\strtensor} A^*)$
	reduces to a differential graded algebra structure on
	$\Secinfty(\Anti^\bullet A_\red^*)$
	corresponding to the reduced Lie algebroid $A_\red$.
\end{remark}

\begin{example}\
\begin{examplelist}
	\item For $A = T\mathcal{M}$
	we obtain the constraint Schouten bracket and the constraint de Rham differential
	from \autoref{sec:CartanCalculus}.
	\item Given a constraint Lie algebra $\liealg{g}$
	considered as a constraint Lie algebroid over $(\{\pt\},\{\pt\},0)$
	the associated DGLA $\Anti_{\strtensor}^\bullet\liealg{g}^*$
	with differential $\D_\liealg{g}$ given by
	\begin{equation}
		(\D_\liealg{g}\alpha)(\xi_1, \dotsc, x_{k+1})
		= \sum_{i<j} (-1)^{i+j} \alpha([\xi_i,\xi_j],\xi_1,\dotsc,\overset{i}{\wedge},\dotsc,\overset{j}{\wedge},\dotsc, \xi_{k+1}),
	\end{equation}
	for $\xi_1, \dotsc, \xi_{k+1} \in \liealg{g}$,
	can be understood as the constraint Chevalley-Eilenberg complex
	with scalar coefficients.
\end{examplelist}
\end{example}

We finish this treatment of constraint Lie algebroids with a short discussion of the reduction of Lie bialgebroids.
Using the above characterization of constraint Lie algebroids we can introduce constraint
Lie bialgebroids as follows:

\begin{definition}[Constraint Lie bialgebroids]
	Let $A \to \mathcal{M}$ and $B \to \mathcal{N}$ be a constraint vector 
	bundles. 
	\begin{definitionlist}
		\item A \emph{constraint Lie bialgebroid over $\mathcal{M}$}
		is given by a constraint vector bundle $A$ together with constraint Lie algebroid
		structures on $A$ and its dual $A^*$, which satisfy 
		\begin{equation}
			\D_{A^*} \GerstBracket{a, b}_A 
			= \GerstBracket{\D_{A^*}(a), b}_A 
			+ (-1)^{\deg(a)} \GerstBracket{a, \D_{A^*}(b)}_A  
		\end{equation}
		for every 
		$a, b \in \ConSecinfty(\Anti_{\strtensor}^\bullet A)_{\Total}$, where 
		$\D_{A^*}$ denotes the differential and $\GerstBracket{\cdot, \cdot}_A$ 
		the Gerstenhaber bracket on 
		$\ConSecinfty(\Anti_{\strtensor}^\bullet A)$.  
		
		\item A constraint morphism of vector bundles $\Phi^* \colon B^* \to A^*$ 
		over $\phi \colon \mathcal{N} \to \mathcal{M}$ is called morphism 
		between constraint Lie bialgebroids, if $\Phi^*$ is a constraint Lie 
		algebroid morphism and a constraint Lie algebroid comorphism. 
		The morphism of constraint Lie bialgebroids with underlying constraint 
		vector bundle morphism $\Phi^* \colon B^* \to A^*$ is denoted by 
		$\Phi \colon (A, A^*) \to (B, B^*)$. 
		
		\item The category of constraint Lie bialgebroids and their morphisms 
		is denoted by $\injConLieBiAlgd$ and the full subcategory of constraint 
		Lie bialgebroids over the constraint manifold $\mathcal{M}$ is denoted 
		by $\injConLieBiAlgd(\mathcal{M})$. 
	\end{definitionlist}
\end{definition} 

\begin{theorem}[Reduction of Lie bialgebroids]\
	\label{thm:redOfLieBiAlgds}
	\begin{theoremlist}
		\item \label{thm:redOfLieBiAlgds_1}
		Let $(A, A^*)$ be a constraint Lie bialgebroid over $\mathcal{M}$. 
		Then $(A_\red, A^*_\red)$ is a Lie bialgebroid over $\mathcal{M}_\red$. 
		
		\item \label{thm:redOfLieBiAlgds_2}
		Let $\Phi^* \colon (A, A^*) \to (B, B^*)$ be a morphism of 
		constraint Lie bialgebroids over $\phi \colon \mathcal{N} \to \mathcal{M}$. 
		Then $\Phi^*_\red \colon (A_\red, A^*_\red) \to (B_\red, B^*_\red)$ is 
		a morphism of Lie bialgebroids over 
		$\phi_\red \colon \mathcal{N}_\red \to \mathcal{M}_\red$.
		
		\item \label{thm:redOfLieBiAlgds_3}
		Mapping every constraint Lie bialgebroid $(A, A^*)$ over 
		$\mathcal{M}$ to $(A_\red, A^*_\red)$ over $M_\red$ and every 
		morphism $\Phi^*$ of constraint Lie bialgebroids to $\Phi^*_\red$ 
		defines a functor
		\begin{equation}
			\red \colon \injConLieBiAlgd \to \LieBiAlgd.
		\end{equation} 
	\end{theoremlist}
\end{theorem}

\begin{proof}
	The first assertion is due to \autoref{thm:ReductionOfLieAlgebroids} applied 
	to $A$ as well as to $A^*$.
	Since the constraint vector bundles are dual of one another, the reduced 
	vector bundles $A_\red$ and $A^*_\red$ are dual to each other. 
	Similarly, the corresponding differential graded constraint algebra and the 
	constraint Gerstenhaber algebra descend and the reduced objects are equivalent  
	to each other, see 
	\autoref{rem:CoincidenceOfReductionOfLRAlgsWithItsCharacterizations}. 
	Consequently, by using the quotient projections the reduced differential and 
	the reduced Gerstenhaber bracket can be shown to be related, such that the 
	pair $(A_\red, A^*_\red)$ yields a Lie bialgebroid. 
	The second assertion follows as a combination of 
	\autoref{thm:ReductionOfLieAlgebroids} 
	\ref{thm:ReductionOfLieAlgebroids_2} and 
	\ref{thm:ReductionOfLieAlgebroids_4}. 
	The third assertion is a consequence of the first and the second assertion as well 
	as of a straight-forward computation proving the functorial properties. 
\end{proof}

\subsection{Reduction of Dirac Manifolds}
\label{sec:RedDirac}

As a last application of constraint geometry we study the reduction of Dirac manifolds.
For this we need to to consider
constraint bilinear forms on vector bundles and complements of subbundles.
Let $E \to \mathcal{M}$ be a constraint vector bundle equipped with a bilinear form
$\SP{\argument, \argument} \colon E \tensor E \to \mathcal{M} \times \Reals$,
which we suppose to be non-degenerate, meaning that the induced constraint map
\begin{equation}
	\flat \colon \ConSecinfty(E) \to \ConSecinfty(E^*),
	\qquad
	s^\flat \coloneqq \SP{s, \argument }
\end{equation}
is a regular monomorphism, see \autoref{prop:MonoEpisConModk}.
Given a constraint sub-bundle $i \colon L \hookrightarrow E$ we define the complement
$L^\perp$ with respect to $\SP{\argument, \argument}$ as
$L^\perp \coloneqq \ker(i^* \circ \flat)$,
where we consider $\flat$ as a map $\flat \colon E \to E^*$.
To see that $L^\perp$ is a well-defined constraint sub-bundle note that $\flat$ and $i^*$ have constant rank, thus \autoref{prop:ConKernSubbundle} applies
and yields
\begin{equation}
	\label{eq:complement}
\begin{split}
	(L^\perp)_\Total &= (L_\Total)^\perp, \\ 
	(L^\perp)_\Wobs &= (L_\Total)^\perp\at{C} \cap E_\Wobs, \\
	(L^\perp)_\Null &= (L_\Total)^\perp\at{C} \cap E_\Null.
\end{split}
\end{equation}
Note that $L^\perp$ does only depend on $(L_\Total)^\perp$ but not on the
$\WOBS$- or $\NULL$-components of $L$.
It follows that $(L^\perp)_\red \subseteq (L_\red)^\perp$, but in general we do not obtain an equality here.

Let now $\mathcal{M} = (M,C,D)$ be a constraint manifold, then we define
the \emph{generalized tangent bundle} of $\mathcal{M}$ by
$\genT \mathcal{M} \coloneqq T\mathcal{M} \oplus T^*\mathcal{M}$.
Recall that 
\begin{equation}
\begin{split}
	(\genT\mathcal{M})_\Total &= TM \oplus T^*M \\
	(\genT\mathcal{M})_\Wobs &= TC \oplus \Ann(D) \\
	(\genT\mathcal{M})_\Null &= D \oplus \Ann(TC)	
\end{split}
\end{equation}
On $\genT\mathcal{M}$ we have two important canonical structures:
First, there is a bilinear form 
\begin{equation}
\begin{split}
	\SP{\argument, \argument} \colon \ConSecinfty(\genT\mathcal{M}) &\tensor \ConSecinfty(\genT\mathcal{M}) \to \ConCinfty(\mathcal{M}) \\
	\SP{(X + \alpha) , (Y + \beta)} &\coloneqq \alpha(Y) + \beta(X), 
\end{split}
\end{equation}
which is clearly symmetric and non-degenerate.
Note that with respect to this bilinear form we have
$(\genT\mathcal{M})_\Wobs^\perp = (\genT\mathcal{M})_\Null$,
and thus $(\genT\mathcal{M})_\Wobs$ is coisotropic, i.e. 
$(\genT\mathcal{M})_\Wobs^\perp \subseteq (\genT\mathcal{M})_\Wobs$.
Moreover, the so-called \emph{Courant bracket} is given by
\begin{equation}
\begin{split}
	\Courant{\argument,\argument} &\colon \ConSecinfty(\genT\mathcal{M}) \tensor \ConSecinfty(\genT\mathcal{M}) \to \ConSecinfty(\genT\mathcal{M}) \\
	\Courant{(X + \alpha,Y + \beta)} 
	&\coloneqq [X,Y] + \Big(\Lie_X\beta - \Lie_Y\alpha + \frac 12 \D\big(\alpha(Y) - \beta(X)\big)\Big).
\end{split}
\end{equation}
As a composition of constraint morphisms the Courant bracket is clearly a well-defined constraint morphism.
Note that, as in the classical situation,
$\Courant{\argument,\argument}$
does not satisfy the Jacobi identity and hence it is not a (constraint) Lie bracket.

\begin{definition}[Constraint Dirac manifold]
A \emph{constraint Dirac manifold}
is given by a constraint manifold $\mathcal{M} = (M,C,D)$
together with a constraint subbundle $L \subseteq \genT\mathcal{M}$
satisfying the following conditions:
\begin{definitionlist}
	\item $L$ is a Lagrangian subbundle, i.e. $L = L^\perp$.
	\item $L$ is involutive with respect to the Courant bracket, i.e.
	$\Courant{\ConSecinfty(L),\ConSecinfty(L)} \subseteq \ConSecinfty(L)$.
\end{definitionlist}
\end{definition}

Recall from \autoref{ex:SectionsSubbundles} that constraint sections of a constraint subbundle form indeed a constraint submodule, and hence the Schouten bracket $\Schouten{\argument, \argument}$ can be evaluated on $\ConSecinfty(L)$.

Using \eqref{eq:complement} we see that $L \subseteq \genT \mathcal{M}$
is a Lagrangian subbundle if and only if $L_\Total \subseteq \genT M_\Total$
is Langrangian.
Thus a constraint Dirac manifold is the same as a Dirac manifold structure on $M_\Total$
such that $L_\Total \cap (\genT \mathcal{M})_\Wobs$
and $L_\Total \cap (\genT \mathcal{M})_\Null$ have constant rank.
The following result shows that such constraint Dirac manifolds can always be reduced.

\begin{proposition}[Reduction of Dirac manifolds]
	Let $\mathcal{M} = (M,C,D)$ be a constraint Dirac manifold with Lagrangian subbundle $L \subseteq \genT \mathcal{M}$.
	Then $\mathcal{M}_\red$ is a Dirac manifold with Lagrangian subbundle 
	$L_\red \subseteq \genT \mathcal{M}_\red$.
\end{proposition}

\begin{proof}
Since taking tangent bundles and taking duals commutes with reduction we have
$L_\red \subseteq \genT \mathcal{M}_\red$.
According to \autoref{prop:ConFormsFieldsVSReduction}
the constraint Lie derivative and de Rham differential reduce to the classical ones, and therefore
the Courant bracket on $\genT \mathcal{M}$
reduces to the Courant bracket on $\genT \mathcal{M}_\red$.
Since taking sections commutes with reduction we see that $L_\red$ is involutive with respect to the Courant bracket.
Moreover, as argued above, we have $L_\red = (L^\perp)_\red \subseteq (L_\red)^\perp$, i.e. $L_\red$ is isotropic.
Finally, for every $p \in C$ we know that
$(\genT\mathcal{M})_\Wobs\at{p}$ is coisotropic and
$L_\Wobs\at{p}$ is a Lagrangian subspace.
It is then well-known from the theory of symplectic vector spaces that
$L_\Wobs\at{p} / (\genT\mathcal{M})_\Wobs^\perp\at{p}$
is indeed coisotropic, see also \cite[Lecture 3]{weinstein:1977a}.
\end{proof}

It should be clear how the above reduction of Dirac manifolds can be extended
to the reduction of more general Dirac structures by introducing constraint Courant algebroids.

\appendix 
\section{Morphisms of Linear Connections}
\label{sec:morphismsOfLinearConnections}

The goal of this section is the notion of morphism between vector bundles 
which are equipped with a linear connections. 
Note, in this section linear connections admit elements of the tangent bundle 
in their first argument, whereas in the main part of this work we usually use 
partial linear connections, i.e. linear connections whose first argument is 
restricted to a given subbundle of the tangent bundle. 
Due to the linearity in the first argument, this is possible without further 
effort. 

We define morphisms between vector bundles with linear connections in the 
following way.

\begin{definition}[Morphism of linear connections]
	\label{def:MorphLinConnection}
	Let $(E \to M, \nabla^E)$ and $(F \to N, \nabla^F)$ be vector bundles
	with linear connections $\nabla^E \colon \Secinfty(TM) \times \Secinfty(E) \to \Secinfty(E)$
	and
	$\nabla^F \colon \Secinfty(TN) \times \Secinfty(F) \to \Secinfty(F)$
	on $E$ and $F$, respectively.
	A vector bundle morphism $(\Phi, \phi) \colon E \to F$ is called 
	\emph{morphism between vector bundles with linear connections 
	$(E, \nabla^E)$ and $(F, \nabla^F)$}, if 
	\begin{equation}
		\label{Eq:VBMorphismBetweenLinearConnectionsDefiningCondition}
		\left( \nabla^{E^*}_{v_p} \phi^*\alpha \right) (s_p) 
		= \left( \nabla^{F^*}_{T\phi(v_p)} \alpha \right) (\phi(s_p))
	\end{equation}
	holds for $v_p \in T_pM$, $\alpha \in \Secinfty(F^*)$, $s_p \in E_p$ and 
	$p \in M$, where $\nabla^{E^*}$ and $\nabla^{F^*}$ denote the dual linear 
	connections of $\nabla^E$ and $\nabla^F$, respectively. 
\end{definition} 

Indeed, this notion deserves to be called morphism, as the following statement 
shows.

\begin{proposition}
	Let $(E \to M, \nabla^E)$, $(F \to N, \nabla^F)$ and $(G \to X, \nabla^G)$ 
	be vector bundles with linear connections and let $(\Psi, \psi) \colon 
	(E, \nabla^E) \to (F, \nabla^F)$ and $(\Phi, \phi) \colon (F, \nabla^F) 
	\to (G, \nabla^G)$ be morphisms between the vector bundles with linear 
	connections.
	Then 
	\begin{propositionlist}
		\item the composition $(\Phi \circ \Psi, \phi \circ \psi) \colon E 
		\to G$ is a morphism between $(E, \nabla^E)$ and $(G, \nabla^G)$, 
		
		\item $\id_E$ is a morphism on $(E, \nabla^E)$.
	\end{propositionlist} 
\end{proposition}

In the following we give equivalent conditions for the condition in Equation 
\eqref{Eq:VBMorphismBetweenLinearConnectionsDefiningCondition}. 
Recall that every vector bundle morphism splits into a vector bundle morphism 
over the identity and a vector bundle morphism, which maps as identity between 
the typical fibres.
In detail, if $(\Phi, \phi)$ denotes a vector bundle morphism between the 
vector bundles $E \to M$ and $F \to N$, then there is a uniquely determined 
vector bundle morphism 
\begin{equation}
	\label{eq:universalVBMorphismToPullBackVB}
	(\widehat{\Phi}, \id_M) \colon E \ni s_p \mapsto (p, \Phi(s_p)) \in \phi^\# F
\end{equation}
over $\id_M$ and a 
vector bundle morphism 
\begin{equation} 
	\label{eq:InducedVBMorphismForPullBackVB}
	(\phi^\#, \phi) \colon 
	\phi^\# F \ni (p, r_{\phi(p)}) \mapsto r_{\phi(p)} \in F
\end{equation} 
associated to the pull-back vector bundle $\phi^\# F \to M$, such that 
\begin{equation}
	\label{Eq:SplitOfVectorBundleMorphism}
	\Phi = \phi^\# \circ \widehat{\Phi}.
\end{equation} 

For every $s \in \Secinfty(F)$ we denote the corresponding pull-back section 
by $\phi^\# s \in \Secinfty(\phi^\# F)$. 

For a smooth map $\phi \colon M \to N$ and every linear connection on 
$F \to N$ there is a uniquely determined linear connection on $\phi^\# F$, 
as stated in the following proposition. 

\begin{proposition}[{\cite{fernandes:2021a}}] 
	\label{prop:ExistenceOfPullBackLinearConnection}
	Let $\phi \colon M \to N$ be a smooth map and let $(F \to N, \nabla)$ 
	be a vector bundle with a linear connection.
	Then there is a unique linear $TM$-connection $\phi^\#\nabla$ on 
	$\phi^\# F$ defined by 
	\begin{equation}
		\phi^\#\nabla_{v_p} \phi^\# s  
		= \left( p, \nabla_{T\phi(v_p)}s \right)
	\end{equation}
	for $v_p \in T_pM$, $s \in \Secinfty(F)$ and $p \in M$.
\end{proposition}

We call the induced linear connection on a pull-back vector bundle 
\emph{pull-back linear connection}. 

Even though it usually is assumed to be general knowledge, we give a detailed 
proof of the following proposition, which is also given in \cite{kern:2023a} 
in a slightly different version. 

\begin{proposition}
	\label{prop:CharacterizationsMorConnection}
	Let $(E \to M, \nabla^E)$ and $(F \to N, \nabla^F)$ be vector bundles with 
	linear connections and let $(\Phi, \phi) \colon E \to F$ be a vector 
	bundle morphism.
	Then the following statements are equivalent:
	\begin{propositionlist}
		\item The morphism $(\Phi, \phi)$ is a morphism between $(E, \nabla^E)$ 
		and $(F, \nabla^F)$.
		
		\item The vector bundle morphism $(\widehat{\Phi}, \id_M)$ is a morphism 
		between $(E, \nabla^E)$ and $(\phi^\# F, \phi^\# \nabla^F)$. 
		Equivalently, this means 
		\begin{equation}
			\nabla^{E^*} \circ \left(\widehat{\Phi}\right)^* 
			= \left(\widehat{\Phi}\right)^* \circ (\phi^\# \nabla^F)^*.
		\end{equation}
		
		\item The linear connections $\nabla^E$ and $\phi^\# \nabla^F$ satisfy
		\begin{equation}
			\widehat{\Phi} \circ \nabla^E 
			= \phi^\# \nabla^F \circ \widehat{\Phi}.
		\end{equation} 
		
		\item The vector bundle morphism $\Phi$ preserves the horizontal 
		bundles $\mathrm{Hor}^E(TM)$ and $\mathrm{Hor}^F(TN)$ of $\nabla^E$ 
		and $\nabla^F$, respectively, i.e. 
		\begin{equation}
			T\Phi \left( \mathrm{Hor}^E(TM) \right) 
			\subseteq \mathrm{Hor}^F(TN). 
		\end{equation}
		
		\item \label{prop:CharacterizationsMorConnection_Parallel}For every open interval $I \subseteq \mathbb{R}$ and every path 
		$\gamma \colon I \to M$ the parallel transports $\Parallel^E$ and 
		$\Parallel^F$ of $\nabla^E$ and $\nabla^F$, respectively, satisfy 
		\begin{equation}
			\Parallel^F_{(\phi \circ \gamma)(t)} \circ \Phi 
			= \Phi \circ \Parallel^E_{\gamma(t)}
		\end{equation}
		for every $t \in I$. 
	\end{propositionlist} 
	
	\begin{proof}
		To show the equivalence of (i) and (ii) we first evaluate the equation 
		in (i).
		For $\alpha \in \Secinfty(F^*)$, $v_p \in T_pM$, $s_p \in E_p$ and $p \in M$ 
		we obtain on the one hand 
		\begin{align*}
			\left( \nabla^{F^*}_{T_p\phi_0(v_p)}\alpha \right) (\Phi(s_p)) 
			&= \left( \phi^\# \left(\phi^\#\nabla^{F^*}_{v_p} \phi^\# \alpha 
				\right)\right) \left( (\phi^\# \circ \widehat{\Phi})(s_p) \right) \\
			&= \left( \phi^\#\nabla^{F^*}_{v_p} \phi^\# \alpha \right) 
			\left( \widehat{\Phi}(s_p) \right) \\
			&= \left( (\widehat{\Phi})^* (\phi^\#\nabla^{F^*}_{v_p} \phi^\# \alpha) \right) (s_p) 
		\end{align*}
		by using Equation \eqref{Eq:SplitOfVectorBundleMorphism} and on the 
		other hand we obtain
		\begin{align*}
			\left( \nabla^{E^*}_{v_p} \Phi^* \alpha \right) (s_p) 
			&= \left( \nabla^{E^*}_{v_p}(\widehat{\Phi})^* (\phi^\# \alpha) \right) (s_p).
		\end{align*}
		By the non-degeneracy of the natural pairing the equivalence of (i) with the 
		equation in (ii) follows. 
		Moreover, since $\widehat{\Phi}$ maps over $\id_M$, the computation 
		\begin{align*}
			\left( \nabla^{E^*}_{v_p}(\widehat{\Phi})^* (\phi^\# \alpha) \right) (s_p)
			= \left( (\widehat{\Phi})^* (\phi^\#\nabla^{F^*}_{v_p} \phi^\# \alpha) \right) (s_p) 
			= \left( (\phi^\#\nabla^{F^*}_{v_p} \phi^\#\alpha \right) \left(\widehat{\Phi}(s_p) \right)  
		\end{align*}
		shows the equivalence of (i) and (ii). 
		
		Next, we show the equivalence between (ii) and (iii).
		On the one hand we have the computation
		\begin{align*}
			\left( \phi^\# \alpha \right) \left(\widehat{\Phi}(\nabla^E_X s) \right) 
			= \Lie_X \left( \phi^\# \alpha \left( \widehat{\Phi}(s) \right) \right)
			- \left( \nabla^{E^*}_X (\widehat{\Phi})^* \phi^\#\alpha \right) (s) 
		\end{align*}
		for $\alpha \in \Secinfty(F^*)$, $X \in \Secinfty(TM)$ and 
		$s \in \Secinfty(E)$, whereas on the other hand we have 
		\begin{align*}
			\left( \phi^\# \alpha \right) \left( \phi^\#\nabla^F_X \widehat{\Phi}(s) \right) 
			= \Lie_X \left( \phi^\# \alpha \left( \widehat{\Phi}(s) \right) \right) 
			- \left( (\widehat{\Phi})^* (\phi^\#\nabla^{F})^*_X \phi^\# \alpha \right) (s). 
		\end{align*}
		Then due to the non-degeneracy of the natural pairing and by identifying 
		$\phi^\#\nabla^{F^*}$ with $(\phi^\# \nabla^F)^*$ the equivalence follows.  
		
		Let us show the equivalence of (i) and (iv).
		Note that the horizontal bundle $\mathrm{Hor}^{E}(TM)$ and the linear 
		connection $\nabla^E$ are related via the horizontal lift by 
		$X^{\nabla^E}(\ell_\varepsilon) = \ell_{\nabla^{E^*}_X \varepsilon}$ for 
		$X \in \Secinfty(TM)$ and $\varepsilon \in \Secinfty(E^*)$,
		where $\ell_\varepsilon$ denotes the fiberwisely linear function 
		corresponding to $\varepsilon$ and $\cdot^{\nabla^E}$ denotes 
		the horizontal lift with respect to $\nabla^E$. 
		Let $v_p \in T_pM$, $s_p \in E_p$, $p \in M$ and 
		$\alpha \in \Secinfty(F^*)$, then on the one hand
		\begin{align*}
			T\Phi \left( \left. v_p^{\nabla^E} \right|_{s_p} \right) \ell_\alpha
			= \left. v_p^{\nabla^E} \right|_{s_p} \ell_{\Phi^* \alpha} 
			= \left( \nabla^{E^*}_{v_p} \Phi^*\alpha \right) (s_p)
		\end{align*}
		and on the other hand
		\begin{align*}
			\left. T\Phi_0(v_p)^{\nabla^F} \right|_{\Phi(s_p)} \ell_{\alpha} 
			= \left( \nabla^{F^*}_{T\phi(v_p)} \alpha \right) (\Phi(s_p)).
		\end{align*}
		Thus (i) is equivalent to 
		$T\Phi\left( \left. v_p^{\nabla^E} \right|_{s_p} \right) 
		= \left. T\phi(v_p)^{\nabla^F} \right|_{\Phi(s_p)}$, which is 
		equivalent to (iv). 
		
		Next, we assume (v) in order to obtain (iv). 
		The horizontal lift of $\nabla^E$ and the corresponding parallel 
		transport are related in the following way.
		For $\varepsilon \in \Secinfty(E^*)$, $r \in I$, a path 
		$\gamma \colon I \to M$ and $s_{\gamma(r)} \in E_{\gamma(r)}$ we obtain 
		\begin{align*}
			\left. \dot{\gamma}(t)^{\nabla^E}(\ell_\varepsilon) \right|
			_{\Parallel^E_{\gamma(t)}(s_{\gamma(r)})} 
			&= \left( \nabla^{E^*}_{\dot{\gamma}(t)}\varepsilon \right) 
				\left( \Parallel^E_{\gamma(t)}(s_{\gamma(r)}) \right) \\ 
			&= \left. \frac{\D}{\D u} \right|_{u = t} \varepsilon 
				\left( \Parallel^E_{\gamma(u)}(s_{\gamma(r)}) \right) 
			-  \varepsilon \left( \nabla^{E}_{\dot{\gamma}(t)}
				\Parallel^E_{\gamma(t)}(s_{\gamma(r)}) \right) \\
			&= \left. \frac{\D}{\D u} \right|_{u = t} \ell_{\varepsilon} 
			\left( \Parallel^E_{\gamma(u)}(s_{\gamma(r)}) \right),
		\end{align*}
		where we used the defining equation 
		$\nabla^{E}_{\dot{\gamma}(t)}\Parallel^E_{\gamma(t)}(s_{\gamma(r)})=0$
		of parallel transport.  
		This means $\Parallel^E_{\gamma(t)}(s_{\gamma(r)})$ is the path, 
		which differentiates to the family of tangent vectors 
		$\left. \dot{\gamma}(t)^{\nabla^E} \right|
		_{\Parallel^E_{\gamma(t)}(s_{\gamma(r)})}$ with $t \in I$. 
		It follows that $\left. \frac{\D}{\D t} \right|_{t = 0} 
		\Parallel^F_{(\phi \circ \gamma)(t)}(\Phi(s_{\gamma(0)})) 
		\in \mathrm{Hor}^F_{\Phi(s_{\gamma(0)})}(TN)$ and by using (v) we 
		obtain 
		\begin{align*}
			T\Phi \left( \left. \frac{\D}{\D t} \right|_{t = 0} 
			\Parallel^E_{\gamma(t)}(s_{\gamma(0)})\right) 
			= \left. \frac{\D}{\D t} \right|_{t = 0} \Phi 
			\left( \Parallel^E_{\gamma(t)}(s_{\gamma(0)})\right) 
			\in \mathrm{Hor}^F_{\Phi(s_{\gamma(0)})}(TN), 
		\end{align*}
		which proves (iv). 
		
		To obtain (v) we assume (iii).
		Let $\gamma \colon I \to M$ be a path with $I \subseteq \mathbb{R}$ an 
		open interval.
		This leads to another path $\phi \circ \gamma \colon I \to N$.
		The pull-back linear connection $\phi^\#\nabla$ can be pulled back to 
		$\nabla^{(\phi \circ \gamma)^\# F}$ on 
		$(\phi \circ \gamma)^\# F \cong \gamma^\# \phi^\# F \to I$, $\nabla^E$ 
		can be pulled back to $\nabla^{\gamma^\# E}$ on $\gamma^\# E \to I$ and 
		$(\widehat{\Phi}, \id_M) \colon E \to \phi^\# F$ restricts to 
		$(\widehat{\Phi}, \id_I) \colon \gamma^\# E \to (\phi \circ \gamma)^\# F$. 
		Note that the parallel transport of $\nabla^E$ along $\gamma$ coincides 
		with the parallel transport of $\nabla^{\gamma^\# E}$ on $I$ and a 
		similar relation applies to the parallel transports of $\nabla^F$ along 
		$\phi \circ \gamma$ and $\nabla^{(\phi \circ \gamma)^\# F}$ on $I$.  
		Since we restrict the base manifolds of $E$ and $F$ to $I$ when 
		considering the pull-back vector bundles $\gamma^\# E$ and 
		$(\phi \circ \gamma)^\# F$, the condition in (iii) leads to 
		$\widehat{\Phi} \circ \nabla^{\gamma^\# E} 
		= \nabla^{(\phi \circ \gamma)^\# F} \circ \widehat{\Phi}$. 
		Then 
		\begin{align*}
			\nabla^{(\phi \circ \gamma)^\# F} \left( \widehat{\Phi} 
			\left( \Parallel^{\gamma^\# E}_t (s_{\gamma(0)})\right)\right)
			= \widehat{\Phi} \circ \nabla^{\gamma^\# E} 
			\left( \Parallel^{\gamma^\# E}_t (s_{\gamma(0)}) \right)
			= 0. 
		\end{align*}
		Thus $\widehat{\Phi} \left( \Parallel^{\gamma^\# E}_t (s_{\gamma(0)})\right)$ 
		is $\nabla^{(\phi \circ \gamma)^\# F}$-parallel, 
		$\Phi \left( \Parallel^{\gamma^\# E}_t (s_{\gamma(0)})\right)$ is 
		$\nabla^F$-parallel and, due to the uniqueness of parallel transport, 
		$\Phi \left( \Parallel^{\gamma^\# E}_t (s_{\gamma(0)})\right)$ is a 
		parallel transport with respect to $\nabla^F$ along 
		$\phi \circ \gamma$. 
		Consequently, 
		\begin{align*}
			\Phi \left( \Parallel^{E}_{\gamma(t)} (s_{\gamma(0)})\right) 
			= \Parallel^F_{(\phi \circ \gamma)(t)}(\Phi(s_{\gamma(0)}))
		\end{align*}
		for every $t \in I$.
	\end{proof}
\end{proposition} 

\begin{corollary}
	Let $(E, \nabla^E)$ and $(F, \nabla^F)$ be vector bundles with linear 
	connections over the base manifold $M$. 
	Then a vector bundle morphism $(\Phi, \id_M) \colon E \to F$ is a morphism 
	between linear connections if and only if  
	\begin{equation}
		\Phi \circ \nabla^E = \nabla^F \circ \Phi.
	\end{equation}
\end{corollary}

\begin{corollary}
	Let $(F, \nabla)$ be a vector bundle with linear connection over the 
	base manifold $N$ and let $\phi \colon M \to N$ be a map. 
	Then the vector bundle morphism $\phi^\# \colon \phi^\# F \to F$ 
	defined in \eqref{eq:InducedVBMorphismForPullBackVB} is a morphism between 
	linear connections, where $\phi^\# F$ is equipped with the pull-back linear 
	connection $\phi^\# \nabla$. 
\end{corollary}

\begin{corollary}
	Let $(E, \nabla^E)$ and $(F, \nabla^F)$ be vector bundles with linear 
	connections over $M$ and $N$, respectively, and let 
	$(\Phi, \phi) \colon (E, \nabla^E) \to (F, \nabla^F)$ be a vector bundle 
	morphism between linear connections. 
	Then $\Phi$ splits into vector bundle morphisms of linear 
	connections 
	$(\phi^\#, \phi) \colon (\phi^\# F, \phi^\#\nabla^F) \to (F, \nabla^F)$ 
	and 
	$(\widehat{\Phi}, \id_M) \colon (E, \nabla^E) \to (\phi^\# F, \phi^\# \nabla^F)$, 
	such that $\Phi = \phi^\# \circ \widehat{\Phi}$.
\end{corollary}

{
	\footnotesize
	\printbibliography[heading=bibintoc]
}

\end{document}